\documentclass[reqno, 12pt]{amsart}

\usepackage{amsfonts, amsthm, amsmath, amssymb}
\usepackage[french,english]{babel}
\usepackage[utf8]{inputenc}

\usepackage{tikz}
\tikzset{node distance=2cm, auto}

\usepackage{hyperref}
\usepackage{a4wide}

\RequirePackage{mathrsfs} \let\mathcal\mathscr

\numberwithin{equation}{section}

\newtheorem{theorem}{Theorem}[section]
\newtheorem{lemma}[theorem]{Lemma}
\newtheorem{proposition}[theorem]{Proposition}
\newtheorem{corollary}[theorem]{Corollary}

\theoremstyle{definition}
\newtheorem*{ack}{Acknowledgements}
\newtheorem{rem}[theorem]{Remark}
\newtheorem*{rem*}{Remark}
\newtheorem{definition}[theorem]{Definition}

\renewcommand{\d}{\mathrm{d}}
\renewcommand{\phi}{\varphi}
\renewcommand{\rho}{\varrho}
\newcommand{\1}{\mathbf{1}}

\newcommand{\0}{\mathbf{0}}

\renewcommand{\AA}{\mathbb{A}}

\newcommand{\FF}{\mathbb{F}}
\newcommand{\ZZ}{\mathbb{Z}}

\newcommand{\NN}{\mathbb{Z}_{>0}}

\newcommand{\QQ}{\mathbb{Q}}
\newcommand{\RR}{\mathbb{R}}
\newcommand{\CC}{\mathbb{C}}

\DeclareMathOperator{\EE}{{\mathbb{E}}}

\newcommand{\cA}{\mathcal{A}}

\newcommand{\cM}{\mathcal{M}}
\newcommand{\cP}{\mathcal{P}}

\newcommand{\cU}{\mathcal{U}}
\newcommand{\cV}{\mathcal{V}}
\newcommand{\cW}{\mathcal{W}}

\newcommand{\Gal}{{\rm Gal}}

\renewcommand{\leq}{\leqslant}

\renewcommand{\geq}{\geqslant}

\renewcommand{\bar}{\overline}

\newcommand{\ma}{\mathbf}

\renewcommand{\a}{\mathbf{a}}
\renewcommand{\b}{\mathbf{b}}
\renewcommand{\c}{\mathbf{c}}
\newcommand{\bd}{\mathbf{d}}
\newcommand{\h}{\mathbf{h}}
\renewcommand{\k}{\mathbf{k}}
\newcommand{\m}{\mathbf{m}}
\newcommand{\q}{\mathbf{q}}
\renewcommand{\u}{\mathbf{u}}
\newcommand{\bv}{\mathbf{v}}
\newcommand{\x}{\mathbf{x}}
\newcommand{\y}{\mathbf{y}}

\newcommand{\s}{\mathbf{s}}

\newcommand{\fo}{\mathfrak{o}}
\newcommand{\fa}{\mathfrak{a}}

\newcommand{\fp}{\mathfrak{p}}
\newcommand{\fq}{\mathfrak{q}}
\newcommand{\fr}{\mathfrak{r}}

\newcommand{\ve}{\varepsilon}
\newcommand{\e}{\mathbf e}

\newcommand{\bla}{{\boldsymbol{\lambda}}}

\newcommand{\bve}{\boldsymbol{\varepsilon}}
\newcommand{\beps}{\boldsymbol{\epsilon}}
\newcommand{\bom}{\boldsymbol{\omega}}
\newcommand{\bka}{\boldsymbol{\kappa}}
\newcommand{\bxi}{\boldsymbol{\xi}}

\DeclareMathOperator{\Pic}{Pic}

\DeclareMathOperator{\supp}{supp}
\DeclareMathOperator{\sign}{sign}

\DeclareMathOperator{\Image}{Im}
\DeclareMathOperator{\Res}{Residue}

\DeclareMathOperator{\n}{N}
\DeclareMathOperator{\nf}{\mathbf{N}}
\DeclareMathOperator{\Ker}{Ker}
\DeclareMathOperator{\nm}{Nm}

\DeclareMathOperator{\Mod}{mod} 
\renewcommand{\bmod}[1]{\,(\Mod{#1})}

\newcommand{\V}{\mathcal{V}}

\renewcommand{\=}{\equiv}
\renewcommand{\t}{\mathbf{t}}

\newcommand\bN{\mathbf{N}}

\DeclareMathOperator{\vol}{vol}
\DeclareMathOperator{\lcm}{lcm}
\newcommand{\eps}{\varepsilon}

\newcommand\<{\langle}
\renewcommand\>{\rangle}

\newcommand{\dsum}{\sideset{}{^{\prime}}{\sum}}
\newcommand{\dagsum}{\sideset{}{^{\dagger}}{\sum}}

\begin{document}

\title[Norm forms and linear polynomials]
{Norm forms for arbitrary number fields 
\\  as products of linear polynomials}

\author{T.D. Browning}
\address{School of Mathematics\\
University of Bristol\\ Bristol\\ BS8 1TW}
\email{t.d.browning@bristol.ac.uk}

\author{L. Matthiesen}
\address{KTH\\
Department of Mathematics\\
10044 Stockholm}
\email{lilian.matthiesen@math.kth.se}

\date{\today}



\begin{abstract}
Given a number field $K/\QQ$ and a polynomial $P\in \QQ[t]$, all of whose
roots are in $\QQ$, let $X$ be the variety defined by the equation
$\nf_K(\x) = P(t)$.
Combining additive combinatorics with descent we show that the
Brauer--Manin obstruction is the only obstruction to the Hasse principle
and weak approximation on any smooth and projective model of $X$.
\newline
\newline
\noindent
\textsc{R\'esum\'e.}
\'Etant donn\'e un corps de nombres $K/\QQ$ et un polyn\^ome $P \in \QQ [ t ]$, 
dont toutes les
racines sont dans $\QQ$, soit $X $ la vari\'et\'e d\'efinie par l'\'equation
$\nf_K ( \x ) = P (t )$. En
combinant la combinatoire additive avec la descente, nous montrons que 
l'obstruction Brauer--Manin est le seul obstacle au principe de Hasse
et  \`a l'approximation faible sur un mod\`ele projectif et lisse de $X$.
\end{abstract}

\maketitle
\setcounter{tocdepth}{1}
\tableofcontents

\section{Introduction}

Let $K/\QQ$ be a finite  extension of number fields of degree $n\geq 2$ 
and fix a basis $\{\omega_1,\dots,\omega_n\}$ for $K$ as a vector space
over $\QQ$. 
We will denote by 
$$
\nf_K(x_1,\dots,x_n)=
N_{K/\QQ}(x_1\omega_1+\dots+x_n\omega_n)
$$ 
the corresponding norm form, where $N_{K/\QQ}$ denotes the field norm.
The objective of this paper is to study the Hasse principle
and weak approximation for the class of varieties $X\subset \AA^{n+1}$
satisfying the Diophantine equation
\begin{equation} \label{eq:norm-bundle}
P(t)=\nf_K(x_1,\dots,x_n), 
\end{equation}
where $P(t)$ is a product of linear polynomials all defined over $\QQ$. 
If $r$ denotes the number of distinct roots of $P$,
then $P$ takes the form
\begin{equation}\label{eq:P}
P(t)=c^{-1}\prod_{i=1}^r (t-e_i)^{m_i}, 
\end{equation}
for $c\in \QQ^*$, $m_1,\dots,m_r \in \NN$ and pairwise distinct 
$e_1,\dots,e_r \in \QQ$.

We let $X^c$ be a smooth and projective model of $X$.
Such a model $X^c$ need not satisfy the Hasse principle and weak
approximation, 
as has been observed by Coray (see \cite[Eq.~(8.2)]{ct-salb}).
Specifically,  when $P(t)=t(t-1)$ and $K$ is the cubic
extension $\QQ(\theta)$, obtained by adjoining a root  $\theta$ of
$x^3-7x^2+14x-7=0$, then the set $X^c(\QQ)$ is not dense in $X^c(\QQ_7)$.
It has, however, been conjectured by Colliot-Th\'el\`ene (see \cite{bud})
that all counter-examples to the Hasse principle and weak approximation
for $X^c$ are accounted for by the Brauer--Manin obstruction.

This conjecture covers the more general case where $X$ arises from an
equation of the form \eqref{eq:norm-bundle}, but the ground field may be
an arbitrary number field $k$ instead of $\QQ$ and the polynomial need
not factorise completely over $k$.
In this more general setting the problem of establishing
Colliot-Th\'el\`ene's conjecture has been addressed under various
assumptions on the  extension $K/k$ and upon the
polynomial $P(t)$. Thus the conjecture is now known to be true for 
Ch\^atelet surfaces 
($[K : k] = 2$ and $\deg(P(t)) \leq  4$) 
by work of  Colliot-Th\'el\`ene, Sansuc and Swinnerton-Dyer 
\cite{crelle-a, crelle-b},  
a family of singular cubic hypersurfaces ($[K : k] = 3$ and $\deg(P(t)) \leq 3$) 
by work of Colliot-Th\'el\`ene and Salberger
\cite{ct-salb}, the case where $K/k$ is arbitrary and $P (t)$ is split over $k$ 
with at most two distinct roots (see \cite{CTHS,HBS, s-s, SJ}) and the case where 
$K/\QQ$ is arbitrary  and $P(t)$ is an irreducible quadratic polynomial over $\QQ$ 
(see  \cite{BHB, DSW}).
Finally, if one assumes Schinzel's hypothesis, then it is true for  $K/k$ cyclic 
and $P(t)$ arbitrary, by work of
 Colliot-Th\'el\`ene, Skorobogatov and Swinnerton-Dyer 
 \cite{98a}.

Suppose now that $k=\QQ$ and $P$ is given by \eqref{eq:P}. 
Until recently, Colliot-Th\'el\`ene's conjecture was only known to hold
unconditionally when $r\leq 2$. 
When $r\leq 1$, the variety $X$ is a principal homogeneous space for
the algebraic torus $R_{K/\QQ}^1$, and so the conjecture  follows
from work of Colliot-Th\'el\`ene and Sansuc \cite{CTSb}.
When $r=2$,  Heath-Brown and Skorobogatov \cite{HBS} prove it
under the additional assumption that $\gcd(n,m_1,m_2)=1$, while
Colliot-Th\'el\`ene, Harari and Skorobogatov \cite[Thm.~3.1]{CTHS}
establish it in general.
Our primary result establishes the conjecture for any $r\geq 1$.

\begin{theorem}\label{t:1}
The Brauer--Manin obstruction is the only obstruction to the Hasse
principle and weak approximation on $X^c$.
\end{theorem}

By combining Theorem \ref{t:1} with the Brauer group calculation in
\cite[Cor.~2.7]{CTHS} we obtain the following corollary.

\begin{corollary}
Suppose that $\gcd(m_1,\dots,m_r)=1$ and $K$ does not contain a proper
cyclic extension of $\QQ$. Then $X^c(\QQ)\neq \emptyset$ and $X^c$
satisfies weak approximation.
\end{corollary}

In \cite{HBS}, 
for the first time, 
Heath-Brown and Skorobogatov combined 
the descent theory of Colliot-Th\'el\`ene and Sansuc \cite{D2} with 
 the Hardy--Littlewood circle method, 
 in order to study the Hasse
principle and weak approximation.
In  joint work with Skorobogatov~\cite{bms}, we introduced 
additive combinatorics into this subject and showed how it may usefully
be combined with descent.
This approach allowed us to study the variety $X$ when $K/\QQ$ is
quadratic. The case $n=2$ of Theorem \ref{t:1} is a special case of 
\cite[Thm~1.1]{bms}. 
Subsequently, Harpaz, Skorobogatov and Wittenberg \cite{HSW} succeeded
in showing  how the finite complexity case of the generalised
Hardy--Littlewood conjecture for primes, as established by Green and Tao
\cite{GT} and Green--Tao--Ziegler \cite{GTZ}, 
can be  used in place of
Schinzel's hypothesis to  study rational points on varieties using 
fibration arguments.
Their work \cite[Cor.~4.1]{HSW} leads to a version of Theorem~\ref{t:1}
in which the extension $K/\QQ$ is assumed to be cyclic,
a fact that was previously only available  under Schinzel's hypothesis, as
a special case of work by Colliot-Th\'el\`ene and Swinnerton-Dyer
\cite{ct-swd} on pencils of Severi--Brauer varieties. 
Building on work of Wei \cite{wei}, they also handle 
(see \cite[Thm.~4.6]{HSW}) the case in which $K$ is a non-cyclic extension
of $\QQ$ of prime degree such that the Galois group of the normal closure
of $K$ over $\QQ$ has a non-trivial abelian quotient.
We emphasise that the results of the present paper are unconditional and
make no assumptions on the degree of the field extension, nor upon the
type of the extension, other than that the ground field is $\QQ$.

Our approach is based upon the strategy of \cite{bms}.
We use descent theory to reduce Theorem~\ref{t:1} to establishing the
Hasse principle and weak approximation for some auxiliary varieties, which
can be analysed using additive combinatorics. 
To introduce these  varieties, let 
$$f_1,\dots,f_r\in \QQ[u_1,\dots,u_s]$$ 
be a system of pairwise non-proportional homogeneous linear polynomials, 
with $s\geq 2$.
For each $1\leq i\leq r$, let $K_i$ denote a number field of degree
$n_i = [K_i:\QQ] \geq 2$.
Central to our investigation will be the smooth variety 
$\V\subset \AA_\QQ^{n_1+\dots+n_r +s}$, defined by 
\begin{equation}\label{eq:torsor}
0\neq \nf_{K_i}(\x_i) = f_i(u_1,\dots,u_s), \quad  (1\leq i\leq r),
\end{equation}
where $\x_i=(x_{i,1}, \dots,x_{i,n_i})$.
For this variety we establish the following theorem, whose 
proof forms the bulk of this paper. 

\begin{theorem}\label{t:ut}
The variety $\V$ defined by \eqref{eq:torsor} satisfies the Hasse
principle and weak approximation. 
\end{theorem}

In fact (see Theorem \ref{t:NB}) we shall produce an asymptotic formula  
for the number of suitably constrained integral points on $\V$ of bounded height.
When  $\V$ only involves quadratic
extensions, Theorem \ref{t:ut} recovers \cite[Thm.~1.2]{bms}.
The latter result was established using work of the second author
\cite{lm1,lm2}.
We will build on  this work in order to obtain the general case of Theorem
\ref{t:ut}.
When  $K_1,\dots,K_r$ are all assumed to be cyclic extensions of $\QQ$, 
a shorter proof of Theorem \ref{t:ut} can be found in \cite[Thm.~1.3]{HSW}.

\subsection{Overview}
We indicate how Theorem \ref{t:ut} implies Theorem \ref{t:1} at the end
of this introduction.
The remainder of this paper is organised as follows.
The overall goal is to prove Theorem \ref{t:ut} by asymptotically
counting points of bounded height in  $\V(\ZZ)$, taking into account the
additional constraints that are imposed by the weak approximation
conditions.
The associated counting function is introduced in Section \ref{s:counting}.
The asymptotic formula obtained in Theorem \ref{t:NB} for this counting
function may prove to be of independent interest.
Theorem \ref{t:NB} is proved using Green and Tao's nilpotent
Hardy--Littlewood method (see \cite{GT}) in combination with the
Green--Tao--Ziegler inverse theorem  \cite{GTZ}.

While containing mostly classical material,
Section \ref{s:technical} fixes the notation for the rest of the paper and
describes a certain fundamental domain that is specific to our counting
problems.
Section \ref{s:prelim} contains a variety of technical results required
at later stages in the paper and may 
be consulted as needed.
Section \ref{s:norms} studies the number of solutions to a congruence
$\nf_{K}(\x)\= A \bmod{p^m}$.
These results are used in Section \ref{s:counting} in order to analyse the
non-archimedean local densities that appear in the statement  of Theorem
\ref{t:NB}.
Section~\ref{s:nilsequences} establishes those estimates for the
Green--Tao method that correspond to the minor arc estimates in the
classical Hardy--Littlewood method. 
These are the estimates needed in order to apply the
Green--Tao--Ziegler inverse theorem \cite{GTZ}.
Section~\ref{s:gen-maj} generalises the construction of the divisor
function majorant from \cite{lm0} to a fairly wide class of positive
multiplicative functions.
Section \ref{s:majorant} combines this majorant for a specific function
with a sieve majorant (as appears in Green and Tao's work \cite{GT} on
primes) to form a majorant for our main counting function.
Section \ref{s:linear-forms} shows that this majorant is pseudorandom,
which finally allows us in Section \ref{s:proof} to employ the Green--Tao 
method in combination with the inverse result \cite{GTZ} to prove  
Theorem \ref{t:NB}.

\subsection{Descent}
We close our introduction with the deduction of Theorem \ref{t:1} from 
Theorem~\ref{t:ut}.
We use the construction of ``vertical'' torsors due to
Schindler and Skorobogatov \cite{s-s}.
These are introduced in \cite{s-s}, in order to study varieties given by
equations of a form similar to \eqref{eq:torsor} via the circle method.
Let $\pi:X\rightarrow \AA^1$ be the morphism which maps $(t;\x)$ to
$t$. Let $U_0\subset \AA^1$ be the open subset on which
$\prod_{i=1}^r (t-e_i)\neq 0$ and let $U=\pi^{-1}(U_0).$
Let $T=R_{K/\QQ}^1$ be the torus given by the affine equation
$\nf_K(\x)=1$.
In \cite[\S 2]{s-s}, a partial compactification $Y$ of $X$ is constructed
and  vertical torsors $\mathcal{T}\rightarrow Y$ are shown to exist.
These are torsors
$\mathcal{T}\rightarrow Y$ whose type is the injective map of
$\Gal(\overline{\QQ}/\QQ)$-modules 
$\widehat{T}^{r}\rightarrow \Pic (\bar Y)$.
It follows from \cite[Lemma 2.2]{s-s} that 
the restriction $\mathcal{T}_U$ of $\mathcal{T}$ to $U\subset Y$ is 
$E\times V$, where
$E$ is a principal homogeneous space for $T$ and 
$V\subset \AA^{nr+1}$ is defined by 
\begin{align*}
t-e_i=\lambda_{i}\nf_{K}(\x_i)\neq 0, \quad (1\leq i\leq r),
\end{align*}
for $\lambda_1,\dots,\lambda_r\in \QQ^*$.
Finally, it follows from \cite[Thm.~2.1]{s-s} 
that Theorem \ref{t:1} holds when $V$ is shown to satisfy the Hasse  
principle and weak approximation for any $\lambda_1,\dots,\lambda_r\in \QQ^*$. 
But $V$ is isomorphic to the variety cut out by the system of equations
$$
e_1-e_i=\lambda_i\nf_K(\x_i)-\lambda_1\nf_K(\x_1), \quad (2\leq i\leq r).
$$
By an obvious change of variables it suffices to establish the Hasse
principle and weak approximation for the variety in $\AA^{nr+r}$ defined
by the system of equations
$$
0\neq (e_1-e_i)\nf_K(\y)
=\lambda_i\nf_K(\x_i)-\lambda_1\nf_K(\x_1), \quad (2\leq i\leq r).
$$
But this variety is isomorphic to the variety
\begin{align*}
v&=\nf_{K}(\y)\neq 0,\\
u-e_iv&=\lambda_{i}\nf_{K}(\x_i)\neq 0, \quad (1\leq i\leq r),
\end{align*}
which is a special case of the varieties $\mathcal{V}$ considered in
Theorem~\ref{t:ut}.
This concludes our deduction of Theorem \ref{t:1} from Theorem \ref{t:ut}.

\subsection*{Notation}
In addition to the usual asymptotic notations, we write  $U \asymp V$ to mean
that $U \ll V$ and $U \gg V$, and we write $V=U^{o(1)}$ to express that 
$V= O_\delta(U^{\delta})$ for every $\delta>0$.
If $\cU$ is a finite set, then we define 
$\EE_{u \in \cU} = |\cU|^{-1} \sum_{u\in \cU}$.
We will write $\1_{u\in \mathcal{P}}$, or equivalently 
$\1_{\mathcal{P}}(u)$, to denote the characteristic function of an element
$u$ satisfying property $\mathcal{P}$. 

\begin{ack}
While working on this paper the first  author was
supported by ERC grant \texttt{306457} and the second author was 
supported by EPSRC grant \texttt{EP/E053262/1} and by
ERC grant \texttt{208091}.  
Some of this work was carried out during the programme ``Arithmetic and
geometry'' in 2013 at the {\em Hausdorff Institute} in Bonn.
We are grateful to J.-L.\ Colliot-Th\'el\`ene and A.\ Skorobogatov for
their interest in this work. 
We would also like to thank  U. Derenthal for useful comments on Section \ref{s:counting}
and A.\ Skorobogatov for pointing out a 
simplification in the descent argument above.
Special thanks are due to the anonymous referee for numerous useful comments and for 
giving us a much  simpler proof of  Lemma~\ref{lem:A>0} and Lemma ~\ref{lem:7.1'}.
\end{ack}
  
\bigskip
 \section{Algebraic number theory}
\label{s:technical}
The purpose of this section is threefold.
First, in Section \ref{s:3ways}, we recall mostly standard material from
algebraic number theory (as found in \cite{landau_alg} and \cite{marcus}),
in order to fix the notation for the rest of the paper.
Next, in Section \ref{s:units}, we will turn to our specific situation.
We will ultimately require a counting function that assigns to each
integer $m$ its number of representations by the norm form
$N_{K/\QQ}(x_1\omega_1 + \dots + x_n \omega_n)$.
Writing $\x=(x_1,\dots,x_n)$ and $\bom=(\omega_1,\dots,\omega_n)$, 
this problem will be turned into a finite counting problem by identifying
representations $m= N_{K/\QQ}(\x.\bom)$ and $m= N_{K/\QQ}(\y.\bom)$ if 
$\x.\bom$ and $\y.\bom$ are associated by a unit in the ring of integers 
$\fo_K$ of $K$.
Such a unit is necessarily of norm $+ 1$.
With this in mind, we will need to describe a fundamental domain for the
action by (the free part of) the group of norm $+1$ units, and its
properties relevant to us. 
In particular, in order to apply a lattice point counting result from the
geometry of numbers we will need to show that the regions we work with
have a sufficiently nice boundary. 
Finally, Section~\ref{s:dedekind} collects together some  analytic
information about the Dedekind zeta function.

\subsection{Three ways to view a number field}\label{s:3ways}
Let $K$ be a number field of degree $n$ over $\QQ$.
We let $D_K$ denote its discriminant, let $\fo=\fo_K$ be the ring of
integers and let  $U_K$ be the unit group.
Given any $\alpha \in K$ we will denote its norm by $\n_{K/\QQ}(\alpha)$.
For any integral ideal $\fa\subset \fo$ we write $\n \fa= \#\fo/\fa$ for
its ideal norm.

Let $r_1$ (resp.\ $2r_2$) be the number of distinct real 
(resp.\ complex) embeddings of $K$. Hence $n=r_1+2r_2$.
The $r_1$ distinct real embeddings are denoted by
$\sigma_{1}, \dots, \sigma_{r_1}$, while 
$\sigma_{r_1+1},\dots , \sigma_{r_1+2r_2}$ denote a complete set of
$2r_2$ distinct complex embeddings, with $\sigma_{r_1+i}$ conjugate to
$\sigma_{r_1+r_2+i}$ for $1\leq i\leq r_2$.

The map 
$\phi: \alpha\mapsto (\sigma_1(\alpha),\dots,\sigma_{r_1+r_2}(\alpha))$
canonically embeds $K$ into the $n$-dimensional commutative $\RR$-algebra
$V = K\otimes_\QQ \RR \cong \RR^{r_1}\times \CC^{r_2}$.
We will typically write $v^{(l)}$, for superscripts $1\leq l\leq r_1+r_2$, for
the projection of any $v\in V$ onto the $l$th component, which is uniquely
determined by $\phi$. 
Thus any $v\in V$ can be written 
$v=(v^{(1)},\dots, v^{(r_1+r_2)})$. 
We identify $K$ with its image $\phi(K)$ in $V$.
Under this identification our fixed $\QQ$-basis
$\{\omega_1,\dots,\omega_n\}$ for $K$ gives rise to an $\RR$-basis
$\{\phi(\omega_1),\dots,\phi(\omega_n)\}$ for $V$ and we
may consider $V$ to be the set 
$\{x_{1}\omega_{1}+\dots +x_{n}\omega_{n}: x_{i}\in \RR\}$.
This allows us to associate to $v\in V$ the corresponding 
vector $\x\in \RR^n$ and vice versa.
For $v\in V$ we define 
$$
\nm(v) = 
v^{(1)}\dots v^{(r_1)}|v^{(r_1+1)}|^{2}\dots |v^{(r_1+r_2)}|^{2},
$$
which in accordance with our convention, we shall also denote as
$\nm(\x)$.
This gives us a formal extension of the norm form 
$\nf_{K}:\QQ^n \to \QQ$ to $\RR^n$.
Indeed, if $v= \phi(\alpha)$ for some 
$\alpha = x_1\omega_1+\dots +x_n\omega_n \in K$,
then
$$\nm(\phi(\alpha))=N_{K/\QQ}(\alpha)=\nm(\x)=\nf_{K}(\x).$$

A third way of viewing $K$ is through logarithmic
coordinates (see \cite[\S 5]{marcus} for details).
Writing
$V^*=(K\otimes_{\QQ} \RR)^* \cong (\RR^*)^{r_1}\times (\CC^*)^{r_2}$,
we define the homomorphism $L: V^*\rightarrow \RR^{r_1+r_2}$ to be
$$
v\mapsto (\log |v^{(1)}|, \dots, \log |v^{(r_1)}|, 
2\log |v^{(r_1+1)}|, \dots, 2\log |v^{(r_1+r_2)}|).
$$
Composing $L$ with the embedding $\phi$ from above, we obtain the
diagram
$$\begin{tikzpicture}
  \node (K) {$K^*$};
  \node (calK) [right of=K] {$V^*$};
  \node (A) [below of=calK] {$\RR^{r_1+r_2}$};
  \draw[->] (K) to node {$\phi$}(calK);
  \draw[->] (K) to node[left] {$\psi\,$}(A);
  \draw[->] (calK) to node  {$L$} (A);
\end{tikzpicture}
$$
where $\psi=L\circ \phi$.
For any $\alpha\in K^*$, the coordinate sum of $\psi(\alpha)$ is given
by $$\psi(\alpha).\1=\log|N_{K/\QQ}(\alpha)|,$$
where $\1=(1,\dots,1)\in \RR^{r_1+r_2}$. 
In particular $\psi(U_K)$ is contained in the hyperplane 
\begin{equation}\label{eq:H}
H=\{\bv\in \RR^{r_1+r_2}: \bv.\1=0\} \subset \RR^{r_1+r_2}.
\end{equation}
If $\u_{r_1+r_2}=(1,\dots,1,2,\dots,2)\in \RR^{r_1+r_2}$, then
$\u_{r_1+r_2} \not\in H$. For any $v\in V^*$ we may write
$$
L(v)= \u_H + \xi_{r_1+r_2} \u_{r_1+r_2} 
$$
for some $\u_H \in H$ and $\xi_{r_1+r_2}\in \RR$.
This decomposition allows us to  understand easily the norm of an element
$v$, since
\begin{equation}\label{eq:log-norm}
\log|\nm(v)|=L(v).\1 = \xi_{r_1+r_2} \u_{r_1+r_2}.\1 = n\xi_{r_1+r_2}.
\end{equation}
It follows that $|\nm(v)|\leq 1$ if and only if $\xi_{r_1+r_2}\leq 0$.
Finally, note that $\ker(\psi)=\mu_K$, where $\mu_K < U_K$
denotes the subgroup of roots of unity.
Thus the map $L$ separates the free part of the group of unity from its
torsion part.

\subsection{Units of norm $+1$}
\label{s:units}

We are now ready to discuss the subgroup of $U_K$ relevant to us and its
action on $V$.
Recall that $N_{K/\QQ}(\eta)=\pm 1$ for any $\eta \in U_K$.
We shall work with the subgroup $U_K^{(+)}$ of $\eta \in U_K$ such that
$N_{K/\QQ}(\eta) = 1$. 
Since $U_K^{(+)}$ is  the kernel of the group homomorphism
$N_{K/\QQ}: U_K \rightarrow \{\pm 1\}$, we deduce that 
either $U_K = U_K^{(+)}$ or $U_K /U_K^{(+)} \cong \ZZ/2\ZZ$. 
In particular, $U_K$ and $U_K^{(+)}$ share the same rank
$r=r_1+r_2-1$.

Let $\mu_K^{(+)}$ be the norm $+1$ subgroup of $\mu_K$.
Then we have decompositions
$$
U_K=\mu_K\times Y_K,
\quad 
U_K^{(+)}=\mu_K^{(+)}\times Y_K^{(+)},
$$
where $Y_K\cong Y_K^{(+)}\cong \ZZ^{r}$.
The relation between these decompositions can be described more precisely.
This is only interesting in the case where $U_K^{(+)}$ is a
proper subgroup of $U_K$, which we assume for now.
If $K$ has a root of unity of norm $-1$ then one can ensure that
each generator of $Y_K$ has norm $1$, since each generator may be replaced
by the product of itself and a root of unity.
This allows for decompositions where $Y_K = Y_K^{(+)}$.
If $K$ has no root of unity of norm $-1$, then $\mu_K=\mu_K^{(+)}$, and
one can ensure that exactly one generator of $Y_K$ has norm $-1$.
To see this, suppose $\eta_1 \dots, \eta_r$ is a system of
fundamental units such that $N_{K/\QQ}(\eta_1)=-1$.
We keep $\eta_1$, but replace any other generator $\eta_i$ of norm $-1$
by the product $\eta_1 \eta_i$.
The resulting system of units $\eta_1, \eta'_2 \dots, \eta'_r$, say, 
still generates $Y_K$ and has the required property.
Furthermore, $\eta_1^2, \eta'_2 \dots, \eta'_r$ forms a system of
generators for $Y_K^{(+)}$.

Our main interest in $Y_K^{(+)}$ lies in the action it induces on $V$
and the associated coordinate space $\RR^n$. 
In general, the action of $U_K$ on $V$ by multiplication induces a natural
action on the coordinate space $\RR^n$ as follows.
For any pair $(\eta,\x) \in U_K \times \RR^n$ we let  
$\eta.\x = \y \in \RR^n$ denote the coordinate vector of 
$\eta(x_1\omega_1+\dots +x_n\omega_n) 
= y_1 \omega_1 + \dots + y_n \omega_n$.

We require a fundamental domain $\mathfrak{F}_+ \subset V^*$ for the
action of $\phi(Y_K^{(+)})$ on $V^*$ that is explicit enough to 
allow  lattice point counting arguments to be applied.
In the case where $Y_K^{(+)}$ is replaced by $Y_K$, the construction of
such a domain is classical (see \cite[\S\S5,6]{marcus}), and it is not
difficult to adapt the construction so as to apply to our situation.
This construction builds on the observation that the action of $U_K$ is
easier to understand in the logarithmic space.
This is useful since the restriction of $L$ to
$\phi(Y_K^{(+)})$ is an isomorphism, which allows us to describe a
fundamental domain for the action of $\phi(Y_K^{(+)})$ on $V^*$
in terms of a fundamental domain for the action of $\psi(Y_K^{(+)})$ on
$\RR^{r_1+r_2}$.

Let $\delta_1,\dots,\delta_r$ be generators for
$Y_K^{(+)}$ and let
$\u_i=\psi(\delta_i)$ for $1\leq i< r_1+r_2$.
Then $\psi(U_K^{(+)}) = \psi(Y_K^{(+)})$ is a lattice of rank $r$
contained in the hyperplane $H$ that was defined in \eqref{eq:H}.
We denote this lattice by $\Lambda_+$ and note that it is generated by
$\u_1, \dots, \u_{r_1+r_2-1}$.
Let $F_+ \subset H$ be a fundamental parallelotope for $\Lambda_+$, and
recall that the vector $\u_{r_1+r_2}=(1,\dots,1,2,\dots,2)\in
\RR^{r_1+r_2}$ does not belong to $H$.
Thus $F_+ \oplus \RR\u_{r_1+r_2} $ describes a fundamental domain for
the action of $\psi(Y_K^{(+)})$ on $\psi(K^*)=\RR^{r_1+r_2}$. 
Since $Y_K^{(+)}$ maps isomorphically onto $\Lambda_+$, an application of
\cite[Lemma 1 in \S6]{marcus} implies the following result.
\begin{lemma}\label{lem:F+}
The set
$\mathfrak{F}_+= \{v\in V^*: L(v)\in F_+\oplus \RR\u_{r_1+r_2}\}
$
is a fundamental domain for the action of $\phi(Y_K^{(+)})$ on $V^*$.
\end{lemma}

We now turn to the desirable properties of the domain $\mathfrak{F}_+$, 
that ultimately facilitate lattice point counting.
Recall that a region $S\subset \RR^n$ is said to be a {\em cone} when 
$\x\in S$ if and only if $\lambda \x\in S$, for any $\lambda\in \RR_{>0}$.
Moreover, if $S\subset \RR^n$ is bounded, its boundary is called 
{\em $(n-1)$-Lipschitz parametrisable}
(see \cite[p.166]{marcus}) if it is contained in the union of the images 
of finitely many Lipschitz functions $f:[0,1]^{n-1}\to \RR^{n}$.
It is easy to see that $\mathfrak{F}_+$ is a cone. 
We are interested in the set
\begin{align*}
\mathfrak{F}_+(1)
&=\{v \in \mathfrak{F}_+: |\nm(v)| \leq 1\}\\
&=\{v \in V^*: L(v)\in F_+\oplus \RR_{\leq0}\u_{r_1+r_2}\},
\end{align*}
where the second equality follows from  \eqref{eq:log-norm}.
The proof contained in \cite[pp.168--172]{marcus} applies \emph{mutatis
mutandis} to our situation and establishes the following result.

\begin{lemma} \label{lem:Lipschitz}
The domain $\mathfrak{F}_+(1)$ has an $(n-1)$-Lipschitz parametrisable
boundary.
\end{lemma}

We will mainly be working in the coordinate space $\RR^n$.
The map 
$$v: \RR^n\setminus \{\0\} \to (\RR^{*})^{r_1}\times(\CC^{*})^{r_2}$$
that takes $\x$ to $v(\x) 
=  x_1 \omega_1+ \dots +  x_n\omega_n 
= (v^{(1)}, \dots, v^{(r_1+r_2)})$ is a linear isomorphism and
preserves Lipschitz parametrisability.
In particular, if
\begin{equation}\label{eq:D+}
\mathfrak{D}_+
= \{\x \in \RR^n: v(\x) \in \mathfrak{F}_+ \} 
\end{equation}
denotes the preimage of the fundamental domain in $\RR^n$, and if 
$$\mathfrak{D}_+(1)=\{ \x \in \mathfrak{D}_+: |\nf_K(\x)| \leq 1\},$$
then Lemma \ref{lem:Lipschitz} implies that
$\mathfrak{D}_+(1)$ has an $(n-1)$-Lipschitz parametrisable
boundary. 

We slightly refine the sets under consideration.
The sign of $\nf_K$ is invariant under the action of $Y_K^{(+)}$.
Thus, for $\epsilon\in \{\pm\}$ and $T>0$ we define the sets
$$
\mathfrak{D}_{+}^\epsilon
= 
\{\x \in \mathfrak{D}_{+}: 0<\epsilon\nf_K(\x)\}
$$
and
\begin{equation}\label{eq:G+}
\mathfrak{D}_{+}^\epsilon (T)
= \{\x \in \mathfrak{D}_+: 0<\epsilon\nf_K(\x)\leq T\}.
\end{equation}
Since $\mathfrak{F}_+$ is a cone, the same is true for $\mathfrak{D}_+$
and $\mathfrak{D}_{+}^\epsilon$.   We deduce that
$\mathfrak{D}_{+}^\epsilon(1)$ has an $(n-1)$-Lipschitz parametrisable
boundary from the same property for $\mathfrak{D}_+(1)$.
Furthermore, we have 
$\mathfrak{D}_{+}^\epsilon(T)
=T^{1/n}\mathfrak{D}_{+}^\epsilon(1)$.

Note that the same facts hold true in the classical setting for
\begin{equation}
\mathfrak{D}^\epsilon(T)
=\{\x\in \mathfrak{D}: 0<\epsilon\nf_K(\x)\leq T\},
\end{equation}
where $\mathfrak{D}=\{\x\in \RR^n: v(\x)\in \mathfrak{F}\}$ and 
$\mathfrak{F}\subset V^*$ is the fundamental domain for $\phi(Y_K)$.

\subsection{Dirichlet coefficients of $\zeta_K$}
\label{s:dedekind}
The construction of the majorant in Section \ref{s:majorant}
relies on a careful analysis of the sequence of Dirichlet coefficients of
the Dedekind zeta function of a number field $K$.
Here we recall the essential properties of $\zeta_K$ and its Dirichlet
coefficients, as found in Landau \cite{landau_alg} or Marcus \cite{marcus}, 
and deduce some preliminary facts required in Section \ref{s:majorant}.

The Dedekind zeta function is defined to be 
\begin{equation}\label{eq:Dedekind}
\zeta_K(s)=\sum_{(0)\neq\fa\subset \fo} \frac{1}{(\n \fa)^s} =
\sum_{m=1}^\infty \frac{r_K(m)}{m^{s}},
\end{equation}
for $s\in \CC$ with $\Re(s)>1$, with 
$$
r_K(m)=\#\{\fa \subset \fo: \n \fa=m\}.
$$
The Dedekind zeta function admits a meromorphic continuation to all of $\CC$ with a simple 
pole at $s=1$ and $\Res_{s=1} \zeta_{K}(s) = h \kappa$, where $h$ is the
class number,
\begin{equation}
\kappa
= \frac{2^{r_1} (2\pi)^{r_2}R_K}{|\mu_{K}| \sqrt{|D_{K}|}},
\end{equation}
and $R_K$ is the regulator. 
It follows from \cite[Thms.~39 and 40]{marcus} that
\begin{equation}\label{eq:210'}
\sum_{m\leq x} r_K(m)=h\kappa x +O(x^{1-1/n}),
\end{equation}
so that the average order of $r_K$ is constant.

The function $r_K$ is multiplicative.
To describe its behaviour at prime powers, let $p$ be any rational prime
and recall that the principal ideal $(p)$ factorises into a product of
prime ideals in $\fo$. That is,
\begin{equation}
 (p)=\fp_1^{e_1}\dots \fp_{r}^{e_r},
\end{equation}
where $e_i=e_{\fp_i}(p), r=r(p)\in \NN$ and 
each $\fp_i \subset \fo$ is a prime ideal satisfying $\n\fp_i=p^{f_i}$,
for some $f_i=f_{\fp_i}(p) \in \NN$.
As in \cite[\S3]{marcus}, we have 
$\sum_{i=1}^re_if_i=n.$
Thus
$$
r_K(p^m)=\#\left\{\fp_1^{m_1}\dots \fp_r^{m_r}\subset \fo:
f_1m_1+\dots+f_r m_r=m\right\},
$$
for any $m\in \NN$. It follows from this that 
\begin{equation}
\label{eq:r-upper}
r_K(p^m)\leq (m+1)^n.
\end{equation}
At rational primes we obtain 
\begin{equation}\label{eq:r}
r_K(p)=\#\{i\in \{1,\dots,r\}: f_i=1\}=\#\{\fp \mid (p): f_{\fp}(p)=1\}.
\end{equation}
In view of \eqref{eq:r}, we partition the set of rational primes into 
three sets 
\begin{equation}\label{eq:PPP}
\begin{split}
\mathcal{P}_0
&=\{p\mid D_K\},\\
\mathcal{P}_1
&=\{p\nmid D_K: 
\exists ~\fp\mid (p) \mbox{ such that } f_{\fp}(p)=1\},\\
\mathcal{P}_2
&=\{p\nmid D_K: 
f_{\fp}(p)\geq 2 ~\forall \fp\mid (p)\}.
\end{split}
\end{equation}
The contributions to $r_K$ from $\cP_0 \cup \cP_1$ and from $\cP_2$
will be dealt with separately.

We end this section with some technical results 
concerning the restricted Euler product
\begin{equation}\label{eq:F(s)}
F(s)=\prod_{p\in \mathcal{P}_2} \Big(1-\frac{1}{p^s}\Big)^{-1},
\end{equation}
for $s\in \CC$ with $\Re(s)>1$.
The following result describes the analytic structure of $F(s)$.

\begin{lemma}\label{lem:F(s)}
There exists $\delta \in \QQ$ satisfying $1/n \leq \delta \leq 1$
such that 
 $\mathcal{P}_1$ has Dirichlet 
 density $\delta$.
Furthermore, there is a function $G(s)$, which is holomorphic and non-zero
in the closed half-plane $\Re(s)\geq 1$, such that 
$F(s)=\zeta(s)^{1-\delta}G(s)$.
\end{lemma}

\begin{proof}
The first part follows from the \v{C}ebotarev density theorem 
(cf.\ \cite[Cor~13.6]{neukirch}), with  $\delta = 1/n$  if and only if
$K/\QQ$ is a Galois extension.  
This implies that there exists a function $G_1(s)$, which is holomorphic
and non-zero in the closed half-plane $\Re (s)\geq 1$, such that 
\begin{equation}\label{eq:yellow}
\prod_{p\in \mathcal{P}_1}
\Big( 1-\frac{1}{p^{s}}\Big)^{-1} =\zeta(s)^{\delta} G_1(s).
\end{equation}
On the other hand, 
\begin{align*}
\prod_{p\in \mathcal{P}_1}
\Big(1-\frac{1}{p^{s}}\Big)^{-1} 
&=
\zeta(s)
\prod_{p\in \mathcal{P}_0\cup \mathcal{P}_2} 
\Big(1-\frac{1}{p^{s}}\Big)
=
\zeta(s)F(s)^{-1}G_2(s),
\end{align*}
where $G_2(s)$ is entire and non-zero.
Combining these expressions we conclude the proof of the lemma
by taking $G(s)=G_1(s)^{-1}G_2(s)$.
\end{proof}

\begin{corollary}\label{cor:log-asymp}
We have
$$\prod_{\substack{p\in \cP_2 \\p \leq T}}
\left(1-\frac{1}{p}\right)^{-1}
\asymp (\log T)^{1-\delta}.$$
\end{corollary}

\begin{proof}
By means of the Tauberian theorem \cite[Thm.\ 5.11]{MV}, applied to the
Dirichlet series $F(s)$, we deduce from Lemma \ref{lem:F(s)} that
$$
(\log T)^{1-\delta} 
\asymp 
\sum_{m\leq T} \frac{\1_{\<\cP_2\>}(m)}{m}
\leq \prod_{\substack{p\in \cP_2 \\p \leq T}}
\left(1-\frac{1}{p}\right)^{-1}.
$$
This gives the correct lower bound.
To establish the  upper bound, we deduce from \eqref{eq:yellow} that
$$
(\log T)^{\delta} 
\asymp 
\sum_{m\leq T} \frac{\1_{\<\cP_1\>}(m)}{m}
\leq \prod_{\substack{p\in \cP_1 \\p \leq T}}
\left(1-\frac{1}{p}\right)^{-1}.
$$
But then 
$$
(\log T)^{\delta} 
\prod_{\substack{p\in \cP_2 \\p \leq T}}
\left(1-\frac{1}{p}\right)^{-1}
\ll \prod_{\substack{p\in \cP_1 \cup \cP_2 \\p \leq T}}
\left(1-\frac{1}{p}\right)^{-1}
\ll \log T,
$$
as required.
\end{proof}

\bigskip
  \section{Technical tools}\label{s:prelim}

\subsection{Geometry of numbers}\label{s:geometry-of-numbers}

We will need to be able to estimate the number of lattice points in
shifts of sufficiently well-behaved expanding regions.
Let $n\in \NN$ and let $\mathcal{B}$ be any bounded subset of $\RR^n$.
Write $T\mathcal{B}=\{T\x: \x\in \mathcal{B}\}$ for the dilation by $T>0$.
The following result is classical.

\begin{lemma}\label{lem:2}
Assume that $\mathcal{B}$ is bounded and that for any $\eps\in (0,1)$
the $\eps$-neighbourhood of the boundary $\partial \mathcal{B}$ 
has volume $O(\eps)$.
Let $\a \in \RR^n$ and let $T\geq 1$. 
Then 
$$
\#\left(\ZZ^n\cap(T\mathcal{B} + \a)\right)
= \vol(\mathcal{B}) T^n + O(T^{n-1}).
$$ 
\end{lemma}

\begin{proof}
Observe that the $\eps$-neighbourhood of $\partial T\mathcal{B}$ arises
as dilation by $T$ of the $\eps T^{-1}$-neighbourhood of
$\partial \mathcal{B}$ and has volume $O(\eps T^{n-1})$.
The lemma follows (cf.\ \cite[App.~A]{GT}) since 
$$\#\left(\ZZ^n\cap(T\mathcal{B} + \a)\right)=
\vol\left( (\ZZ^n \cap(T\mathcal{B} + \a))+[0,1)^n\right),$$
and the set in the latter volume agrees with $T\mathcal{B} + \a$ outside an
$O(1)$-neighbourhood of $\partial T\mathcal{B} + \a$.
\end{proof}

It is not hard to see that any non-empty bounded set in $\RR^n$ whose boundary is
$(n-1)$-Lipschitz parametrisable satisfies the hypotheses of Lemma
\ref{lem:2}.
This follows, for example, from the proof of Lemma 2 in \cite[\S6]{marcus}.
Similarly, bounded convex sets in $\RR^n$  satisfy the hypotheses of
the lemma (see \cite[Cor.~A.2]{GT}, for example).

Given a finite set of fixed regions
$\mathcal B_1, \dots, \mathcal{B}_m$ to which Lemma~\ref{lem:2} applies, 
the hypotheses of the lemma are also met by 
any set which arises through unions and intersections of these sets. 
In particular it applies to intersections of bounded convex sets with bounded sets having
$(n-1)$-Lipschitz parametrisable boundary.

\subsection{Complex analysis} 
Throughout this section we will write $\sigma$ for the real part of a complex number $s\in \CC$. 
In the course of Sections \ref{s:majorant} and \ref{s:linear-forms} we
will encounter several truncated Euler products of the following form. 
For a given constant $C>0$, given $x\geq 1$ and a given multiplicative
arithmetic function $h:\NN \to \CC$, define the Euler product 
$$
E_{C,x}(s;h)=
 \prod_{C < p < x} 
 \left(1 + \sum_{k\geq 1} \frac{h(p^k)}{ p^{sk}}\right),
$$
for $\sigma>1$.
Since the product is truncated, one expects that $E_{C,x}(1+s_0;h)$ is
well approximated by its value at $s=1$, provided that $|s_0|$ is sufficiently
small and one has some control on $h$.
The following result makes this statement precise.

\begin{lemma} \label{lem:E-asymp}
Let $c>0$ be a constant. 
Let $H\geq 1$ and suppose $h:\NN \to \CC$ is a multiplicative function
satisfying $|h(p^k)| \leq H^k$ at all prime powers $p^k$.   
Then  the Euler product $E(s)=E_{3eH,x}(s;h)$ satisfies
$$
E(1+s_0) = E(1) + O(|s_0|(\log x)^{O(1)}). 
$$
uniformly in $x$, 
for $s_0 \in \CC$ with $|s_0|\leq x^{-c}$.
Furthermore, we have 
\begin{equation}\label{eq:number}
|E(1)| \asymp
\left| \prod_{3eH < p < x} 
 \Big(1 + \frac{h(p)}{p}\Big)\right|.
\end{equation}
The implied constants in these estimates are allowed to depend on $c$ and $H$.
\end{lemma}

For $h$ satisfying $|h(p)|\leq H$ on the primes, we may combine
\eqref{eq:number} with Mertens's theorem to deduce that
$$
(\log x)^{-H} \ll |E(1)|\ll (\log x)^H.
$$
This shows that the main term dominates the error term in our asymptotic
formula for 
$E(1+s_0)$, when $|s_0|\leq x^{-c}$.

\begin{proof}[Proof of Lemma \ref{lem:E-asymp}]
We allow our implied constants to depend on $c$ and $H$.
Let 
$E(s)=E_{3eH,x}(s;h)$ and let  $s \in \CC$ be such that  
$|s| < (\log x)^{-1}$.
Then 
$$
\left|\sum_{k\geq 1} \frac{h(p^k)}{p^{(1+s)k}}\right|
\leq \frac{H}{p^{1+\sigma} - H}.
$$
This is at most $1/2$ for $3eH<p<x$, since 
$$p^{1+\sigma} \geq p^{1-(\log x)^{-1}} \geq p e^{-1} \geq 3H.$$ 
This shows that $E(1+s)$ is non-zero for $s$ satisfying
$|s| < (\log x)^{-1}$ and, furthermore, that $E$ is
holomorphic on a domain containing this disc.
The Taylor expansion about $1$ is given by
$$
E(1+s) = \sum_{j\geq 0} s^j \frac{E^{(j)}(1)}{j!}.
$$
Cauchy's inequality yields
\begin{align*}
\frac{|E^{(j)}(1)|}{j!}
&\leq (\log x)^j 
  \max_{|s|=(\log x)^{-1}}|E(1+s)|. 
\end{align*}
But the right hand side is bounded by
\begin{align*}
&\leq (\log x)^j 
 \prod_{3eH< p < x} 
 \Big(1 + \sum_{k\geq 1} \frac{|h(p^k)|}{ p^{(1 - 1/\log x)k}}\Big) \\
&\leq 
 (\log x)^j
 \prod_{3eH < p < x}
 \Big(1 + \frac{H}{p^{1-\frac{1}{\log x}} - H} \Big)\\
&\leq 
 (\log x)^j
 \prod_{3eH < p < x}
 \Big(1 + \frac{3eH}{2p} \Big)\\
&\ll
 (\log x)^{j+O(1)}.
\end{align*}
Thus for 
$|s_0| < x^{-c}$
we have 
\begin{align*}
|E(1)- E(1+s_0)|
\ll \sum_{j\geq 1} |s_0|^j(\log x)^{j+O(1)}
\leq |s_0| (\log x)^{O(1)}.
\end{align*}

To check the  final claim of the lemma, we recall that 
$|h(p^k)|\leq H^k$. 
Using the logarithmic series we therefore  deduce that
\begin{align*}
\log\left(
 E(1)\prod_{3eH<p<x}\Big(1+\frac{h(p)}{p}\Big)^{-1}\right)
&=
  \sum_{3eH<p<x} \left(
  \log \Big(1 + \sum_{k \geq 1} \frac{h(p^k)}{p^{k}}\Big)
  -\log \Big(1 + \frac{h(p)}{p}\Big) \right) \\
&=
  \sum_{3eH<p<x} \Big(
  \sum_{k \geq 2} \frac{h(p^k)}{p^{k}} +O\Big(\frac{H^2}{p^{2}}\Big)\Big).
\end{align*}
But this is $O(1)$, which therefore concludes the proof.
\end{proof}

\subsection{Lifting lemmas}

This section establishes two fairly general results of Hensel type,
the second of which will be applied in Sections \ref{s:norms},
\ref{s:counting}, \ref{s:nilsequences} and \ref{s:proof}.
Let $p$ denote a prime number and let $v_p(n_1,\dots,n_s)$ denote the $p$-adic order 
of the greatest common divisor of any $s$-tuple of integers $(n_1,\dots,n_s)$. 

\begin{lemma}\label{lem:hensel}
Let $m,\ell, \delta\in \ZZ_{\geq 0}$, with 
$$
m\geq 2\delta+1, \quad 0\leq \delta\leq m-\ell.
$$
Suppose we are given a polynomial $F\in \ZZ[t_1,\dots,t_s]$,  $A\in \ZZ$
and $\a\in \ZZ^s$. 
Let
$$
R_\delta(p^m,A; p^\ell)=\left\{\t\in (\ZZ/p^m\ZZ)^s: 
\begin{array}{l}
F(\t)\=A\bmod{p^m}, ~v_p(\nabla F(\t))=\delta\\
\t\=\a\bmod{p^{\ell}}
\end{array}
\right\}.
$$
Then we have 
$$
\frac{\#R_\delta(p^m,A; p^\ell)}{p^{m(s-1)}} = 
\frac{\#R_\delta(p^{m+1},A+kp^m; p^\ell)}{p^{(m+1)(s-1)}},
$$ 
uniformly for $k\in \ZZ/p\ZZ$.
\end{lemma}

\begin{proof}
For any $\t\in R_\delta(p^m,A;p^\ell)$ and any $\t'\in\ZZ^{s}$, 
the condition $m \geq 2\delta+1$ implies that
\begin{align*}
F(\t+p^{m-\delta}\t')
&\equiv F(\t) + p^{m-\delta} \t'.\nabla F(\t) \bmod{p^m}\\
&\equiv A \bmod{p^m}.
\end{align*} 
Similarly, we deduce that
\begin{align*}
\nabla F(\t+p^{m-\delta}\t') -\nabla F(\t)
&\equiv \mathbf{0}
   \bmod{p^{m-\delta}}\\
&\equiv \mathbf{0} \bmod{p^{\delta+1}},
\end{align*}
and, since $\ell \leq m-\delta$, we also have
$\t+p^{m-\delta}\t'\=\a\bmod{p^\ell}$.
Thus $R_\delta(p^m,A;p^\ell)$ consists of cosets modulo $p^{m-\delta}$.

Let  $\t \in (\ZZ/p^{m+1}\ZZ)^s$ such that 
$\t \bmod{p^{m}} \in R_\delta(p^m,A;p^\ell)$.
Then $\t + p^{m-\delta}\t'$ runs through $p^s$ different cosets modulo
$p^{m+1-\delta}$ as $\t'$ runs through $\ZZ^s$.
Moreover, for any $k\in \ZZ/p\ZZ$, we have 
$\t+p^{m-\delta}\t'\in R_\delta(p^{m+1},A+kp^m; p^\ell)$
if and only if 
$$
p^{-m}(F(\t)-A)+p^{-\delta}\t'.\nabla F(\t)\equiv k \bmod{p},
$$
for which there are precisely $p^{s-1}$ incongruent solutions in $\t'$
modulo $p$.
This establishes the lemma.
\end{proof}

Now let 
 $G\in \ZZ[x_1,\dots,x_n]$ be a homogeneous polynomial of degree $n$
and let $A\in \ZZ$.  For given $\a\in \ZZ^n$ and $m,\ell\in \ZZ_{\geq
0}$, 
let
$$
\gamma(p^m,A; p^\ell)=\#\left\{\x\in (\ZZ/p^{m}\ZZ)^n: 
\begin{array}{l}
G(\x)\=A\bmod{p^m}\\
\x\=\a\bmod{p^\ell}
\end{array}
\right\}.$$
The counting function $\gamma$ satisfies the following lifting property.

\begin{lemma}[cf.\ {\cite[Cor.~6.4]{lm1}}]
\label{lem:C6.4}
Assume that $m\geq 1$, $A \not= 0$ and 
$$\ell+v_p(A)+v_p(n)< \frac{m}{2}.
$$
Then we have 
$$
\frac{\gamma(p^m,A; p^\ell)}{p^{m(n-1)}} = 
\frac{\gamma(p^{m+1},A+kp^m; p^\ell)}{p^{(m+1)(n-1)}},
$$ 
uniformly for $k\in \ZZ/p\ZZ$.
\end{lemma}

\begin{proof}
Let $m\geq 1$ and let 
$\x\bmod{p^m}$ be such that $G(\x)\=A\bmod{p^m}$ and 
$\x\=\a\bmod{p^\ell}$ and 
$p^\delta\mid \nabla G(\x)$, for some 
$\delta\geq 0$. Since $\x.\nabla G(\x)=nG(\x)$ we conclude that 
$
p^{\min\{\delta,m\}} \mid nA,
$
whence $\delta<\frac{m}{2}-\ell$
under the hypotheses of the lemma.  
In particular this inequality implies that
$m\geq 2\delta+1$ and $\delta\leq m-\ell$.
Observe that 
\begin{align*}
\gamma(p^m,A; p^\ell)
&=\sum_{0\leq \delta <\frac{m}{2}-\ell} 
  \#R_{\delta}(p^m,A; p^\ell),\\
\gamma(p^{m+1},A+kp^m; p^\ell)
&=\sum_{0\leq \delta <\frac{m}{2}-\ell} 
  \#R_{\delta}(p^{m+1},A+kp^m; p^\ell),
\end{align*}
for any $k\in \ZZ/p\ZZ$, in the notation of Lemma \ref{lem:hensel}.
The statement of Lemma~\ref{lem:C6.4} therefore follows from 
Lemma~\ref{lem:hensel}  with $F=G$ and $s=n$.
\end{proof}

\bigskip
\section{Norm forms modulo $p^m$}\label{s:norms}

Throughout this section $K/\QQ$ will denote a finite extension of degree
$n$, with integral basis $\{\omega_1,\dots,\omega_n\}$ for the ring of
integers $\fo=\fo_K$.
Suppose we are given an integral ideal $\fa\subset \fo$, with
corresponding $\ZZ$-basis $\{\alpha_1,\dots,\alpha_n\}$. 
These bases are both $\QQ$-bases for  $K/\QQ$.
We let $\Delta(\alpha_1,\dots,\alpha_n)=|\det(\sigma_i(\alpha_j))|^2$,
and similarly for $\{\omega_1,\dots,\omega_n\}$.
Let $c_{k \ell}\in \ZZ$ be such that 
\begin{equation}\label{eq:change}
\alpha_k=\sum_{\ell=1}^n c_{k \ell} \omega_\ell, 
\end{equation}
for $1\leq k\leq n$.
Then according to \cite[Satz 40 and 103]{landau_alg}, we have 
$$
\Delta(\alpha_1,\dots,\alpha_n)=(\n \fa)^2 |D_K|=
|\det(c_{k \ell})|^2
\Delta(\omega_1,\dots,\omega_n).
$$
In particular  $\n \fa=|\det(c_{k \ell})|$.

The norm forms we discuss in this section take the more general shape
\begin{equation}\label{eq:blue}
\nf(\x;\fa)=N_{K/\QQ}(x_1\alpha_1+\dots+x_n\alpha_n),
\end{equation}
which defines a homogeneous polynomial of degree $n$ with coefficients in
$\ZZ$. 
Note that $\nf(\x;\fo)=\nf_K(\x)$ in our earlier notation, which we will often abbreviate by $\nf(\x)$.
Given $A\in \ZZ$, $\x_0\in \ZZ^n$ and $M,q\in \NN$ with 
$M\mid q$, we define the counting function
\begin{equation}\label{eq:def-rho}
\rho(q,A,\fa;M)=\#\left\{\x\in (\ZZ/q\ZZ)^n:  
\begin{array}{l}
\nf(\x;\fa)\= A \bmod{q}\\
\x\=\x_0 \bmod{M}
\end{array}
\right\}.
\end{equation}
Such counting functions appear naturally when analysing weak
approximation conditions at non-archimedean places.
In the special case $\fa=\fo$, we put 
\begin{equation}\label{eq:ask}
\rho(q,A;M)=
\rho(q,A,\fo;M).
\end{equation}
Likewise, when $M=1$, we set
$$
\rho(q,A,\fa)=\rho(q,A,\fa;1), \quad 
\rho(q,A)=
\rho(q,A,\fo;1).
$$

This section is devoted to a detailed  analysis of
the quantities $\rho(q,A,\fa)$ and $\rho(q,A)$.
When $M\neq 1$ it will suffice for our purposes to note that 
$\rho(q,A,\fa;M) \leq \rho(q,A,\fa)$ and apply the results for $M=1$. 
By the Chinese remainder theorem we may consider
$\rho(q,A,\fa)$ and $\rho(q,A)$ in the special  case  $q=p^m$ for a
rational prime $p$ and $m\in \NN$. We will mainly be concerned 
with the situation for $p \nmid  D_K \n\fa$.
Our first result shows that any two norm forms are locally equivalent.

\begin{lemma}[cf.\ {\cite[Lemma 4.2]{lm1}}]\label{lem:equivalent'}
Let $m\in \NN$ and let  $p\nmid \n \fa$. Then we have
$$\rho(p^m,A,\fa)=\rho(p^m,A).
$$
\end{lemma}

\begin{proof}
Let $\ma{C}\in M_n(\ZZ)$ be the matrix with coefficients $c_{k \ell}$ as
in \eqref{eq:change}.
Since $\n \fa=|\det \ma{C}|$, it follows that $p\nmid \det \ma{C}$,
whence $\ma{C}$ and $\ma{C}^t$ are invertible in $\ZZ_p$. 
Let $\bom=(\omega_1,\dots,\omega_n)$. 
Then, for any $\x\in \ZZ^n$, we have 
\begin{align*}
\nf (\x \ma{C}^t; \fo)
&= N_{K/\QQ}\left(\bom. (\x \ma{C}^t)\right)\\
&= N_{K/\QQ}\left((\ma{C}\bom). \x   \right)\\
&= \nf(\x;\fa).
\end{align*}
It follows that
$\nf(\x;\fa)$  and $\nf(\x;\fo)$
are equivalent over $\ZZ_p$, which suffices for the lemma.
\end{proof}

We now have everything in place to record our main result in this
section.

\begin{lemma}\label{lem:A>0}
Let $A\in \ZZ$, let $m\in \NN$, let $p$ be a prime and let $k=v_p(A).$
Then we have 
$$
\frac{\rho(p^m,A)}{p^{m(n-1)}}\ll \min\left\{k+1,m\right\}^n.
$$
Suppose that $p\nmid D_K$ and $k<m$.
Then we have 
$$
\frac{\rho(p^m,A)}{p^{m(n-1)}}
=
r_K(p^k)
\left(1-\frac{1}{p}\right)^{-1}
\prod_{\fp\mid p} \left(1-\frac{1}{\n \fp}\right).
$$
\end{lemma}

\begin{proof}
Our proof of this result was suggested to us by the anonymous referee and is 
based on the observation that $\rho(p^m,A)$ is equal to the number of 
$\alpha\in \fo$, modulo $p^m$, for which $N_{K/\QQ}(\alpha)\equiv A \bmod{p^m}$.
It will be convenient to temporarily  abbreviate $N_{K/\QQ}$ by $N$ in what follows. 

We first consider the special case where $p^m\mid A$. 
In this case every $\alpha$ that is counted by $\rho(p^m,A)$ will have an ideal 
divisor $\fq$ such that $p^m\mid \n \fq$, and with the property that $p^m\nmid \n 
\fq'$ for every proper divisor $\fq'\mid \fq$.
Thus $\fq=\prod_\fp  \fp^{e_\fp}$ for prime ideal divisors $\fp\mid (p)$, with 
$e_\fp \leq m$ for each $\fp$. Since there are at most $n$ prime ideal factors of 
$\fp$, there are at most $(m+1)^n$ possibilities for $\fq$. 
For each such $\fq$ the number of $\alpha \bmod {p^m}$ with $\fq\mid (\alpha)$ is 
$p^{mn}/\n \fq\leq p^{m(n-1)}.$ All together, this yields
\begin{equation}\label{eq:rhb-1}
\rho(p^m,A)\leq (m+1)^n p^{m(n-1)} \ll m^n p^{m(n-1)}
\end{equation}
whenever $p^m\mid A$.

Suppose now that $p^k\| A$ with $0\leq k<m$. Then, for any $\alpha$ as above, 
$p^k\| N(\alpha)$ and there is a unique ideal $\fq$ containing $\alpha$, 
with $\n \fq=p^k$. 
Note that $\fq$ contains $p^k$.
It follows that 
$$
\rho(p^m,A)=\sum_{\fq}
\#\{\alpha \bmod{p^m}:  \alpha\in \fq, ~ N(\alpha)\equiv A\bmod{p^m}\},
$$
where the sum is extended over integral ideals $\fq$ of norm $p^k$.
The next goal is to relate, for any of these $\fq$, the cardinality above to 
$\rho(p^{m-k},B)$ for some $B$ that is coprime to $p$.
To this end, recall that there exists a prime ideal $\fr$ in the ideal class 
$[\fq]$ which is coprime to $(p)$.
Suppose that $\fq(\beta)=\fr(\gamma)$, so that $\alpha\in \fq$ if and only if 
$\alpha\beta\gamma^{-1}\in \fr$.
We now have 
\begin{align*}
\#\{\alpha &\bmod{p^m} :   \alpha\in \fq, ~ N(\alpha)\equiv A\bmod{p^m}\}\\
&=
\#\{\alpha \bmod{p^m}:  \alpha\beta\gamma^{-1}\in \fr, ~ N(\alpha\beta\gamma^{-1})\equiv A
N(\beta\gamma^{-1})
\bmod{p^m N(\beta\gamma^{-1})}\}\\
&=
\#\{\nu \bmod{p^m\beta\gamma^{-1}}:  \nu\in \fr, ~ N(\nu)\equiv B
\bmod{p^{m-k} \n \fr}\},
\end{align*}
where $B= A p^{-k} \n \fr$, and where we note that $p^m \beta \gamma^{-1}\in \fr$ 
since $p^m\in \fq$. 
Since $\n \fr | N (v)$ for any $v \in \fr$, we can replace the final congruence 
condition above by $N(\nu)\equiv B \bmod{p^{m-k}}$.
Now choose a $\ZZ$-basis $\{\xi_1,\dots,\xi_n\}$ for $\fr$ and recall the 
definition \eqref{eq:blue} of the associated norm form $\nf(\x;\fr).$ 
Then the above counting function is equal to the number of integer vectors 
$\x \in \ZZ^n$ producing distinct $\nu \bmod{p^m \beta \gamma^{-1}}$ for which 
$\nf(\x;\fr) \equiv B \bmod{p^{m-k}}.$
Our task therefore falls to counting solutions of 
$\nf(\x;\fr)\equiv B \bmod{p^{m-k}}$ lying in cosets of a certain lattice.
To describe this lattice, note that $\fq \mid (p^k)$ and therefore 
$p^{m-k}\fr\mid p^m \fr \fq^{-1}$. 
Further, if $\nu$ and $\nu'$ agree modulo $p^{m-k}\fr$ then we have $\x\equiv \x' 
\bmod{p^{m-k}}$ and
so $\nf(\x;\fr)$ and $\nf(\x';\fr)$ coincide modulo $p^{m-k}$.
It therefore follows that 
\begin{align*}
\#\{\nu &\bmod{p^m\beta\gamma^{-1}}:  \nu\in \fr, ~ N(\nu)\equiv B
\bmod{p^{m-k}}\}\\
&=\frac{\n (p^m\fr\fq^{-1})}{\n (p^{m-k}\fr)}\#\left\{
\x \bmod{p^{m-k}}: \nf(\x;\fr)\equiv B \bmod{p^{m-k}}\right\} \\
&= p^{k(n-1)} \rho(p^{m-k}, B),
\end{align*}
by Lemma \ref {lem:equivalent'}, since $p\nmid \n \fr$.
Observing that the number of ideals $\fq$ of norm $p^k$ is just $r_K(p^k)$, we 
have therefore shown that for any prime $p$,
there exists $B= A p^{-k}\nf \fr\in \ZZ$ such that $p\nmid B$ and 
\begin{equation}\label{eq:rhb}
\rho(p^m,A)=r_K(p^k) p^{k(n-1)} \rho(p^{m-k},B),
\end{equation}
whenever $0\leq k<m$.

It remains to analyse $\rho(p^m,A)$ when $p\nmid A$. 
Consider the group homomorphism
$$\theta: (\fo/(p^m))^*\to (\ZZ/(p^m))^*$$ 
that is induced by the norm. 
Since $\rho(p^m,A) = \#\theta^{-1}(A) \leq \# \Ker{\theta}$,
we proceed by bounding its kernel.
Suppose first that $p>2$ and let $g$ be a primitive root for $p^m$.
Note that the image of $\theta$ contains the subgroup $H$ generated by
$n$th powers of elements of $(\ZZ/(p^m))^*$. Hence 
$$
[(\ZZ/(p^m))^*: \Image (\theta)]
=\frac{[(\ZZ/(p^m))^*: H]}{[\Image (\theta):H]}\leq [(\ZZ/(p^m))^*: H] 
=\frac{\phi(p^m)}{\#H}.
$$
But $\#H=\phi(p^{m})/\gcd(n,\phi(p^m))$, which readily implies that 
$\Image(\theta)$  has index at most $n$ in $(\ZZ/(p^m))^*$. 
When $p=2$ we argue similarly, using the fact that elements of $(\ZZ/(2^m))^*$ 
can be expressed uniquely as $(-1)^u5^v$ for $u\in \{1,2\}$ and 
$v\in \{1,\dots,2^{m-2}\}$, to deduce that 
$[(\ZZ/(p^m))^*: \Image (\theta)]\leq 2n$.
For any prime power $p^m$ it therefore follows that  we have
$\Ker(\theta)\leq 2n\phi_K(p^m)/\phi(p^m)$, 
where $\phi_K$ is the Euler totient function associated to $K$.
Hence
\begin{equation}\label{eq:rhb-2}
\rho(p^m,A)\leq 2n\frac{\phi_K(p^m)}{\phi(p^m)}\ll p^{m(n-1)},
\end{equation}
whenever $p\nmid A$.
We can be more precise when $p$ is further assumed to be unramified.
Assuming that $p\nmid AD_K$,  we claim that 
 \begin{equation}\label{eq:rhb-3}
\frac{\rho(p^m,A)}{p^{m(n-1)}}
=
\left(1-\frac{1}{p}\right)^{-1}
\prod_{\fp\mid p} \left(1-\frac{1}{\n \fp}\right).
\end{equation}
Taking $G=\nf$ and $\ell=0$ in Lemma \ref{lem:C6.4}, we see that
it suffices to establish this fact when $m=1$.
The strategy is to show that the map $\theta$ is onto, which immediately implies 
that $\rho(p,A)=\#\Ker(\theta)=\phi_K(p)/\phi(p)$, so that the case $m=1$ of 
\eqref{eq:rhb-3} follows.
To show that $\theta$ is onto we must show that there exists  $\x\in \FF_p^n$ such 
that  $\nf(\x)\equiv A \bmod{p}$. For this we deduce from the Chevalley--Warning  
theorem (see \cite[\S I.2.2]{serre}) 
that the number of projective solutions is divisible by $p$.  
Moreover, the number of solutions on the hyperplane at infinity is  
 $$
 \#\{\x\in \FF_p^n: \nf(\x)\equiv 0\bmod{p}\}=p^n-\#\{\x\in \FF_p^n: p\nmid \nf(\x)\}=p^n-\phi_K(p).
 $$
Since $p\nmid \phi_K(p)$ for an unramified prime $p$, we may conclude that the 
number of affine solutions to the congruence  $\nf(\x)\equiv A \bmod{p}$ is not 
divisible by $p$.  This shows that $\theta$ is onto, as required. 

We may now conclude the proof of Lemma \ref{lem:A>0}. 
The second part follows from \eqref{eq:rhb} and \eqref{eq:rhb-3}. 
Recalling from \eqref{eq:r-upper} that $r_K(p^k)\leq (k+1)^n$, 
the first part follows from \eqref{eq:rhb-1}, \eqref{eq:rhb} and \eqref{eq:rhb-2}.
\end{proof}

\bigskip
\section{Counting points on systems of norm form equations} 
\label{s:counting}

While the previous two sections described background, notation and
technical tools, we now begin with the proof of our main theorem.
In the first two parts of this section we state and discuss our main
auxiliary result which may be interpreted as an asymptotic formula for
the number of integral points of bounded height on an integral model for
the variety $\V\subset \AA_\QQ^{n_1+\dots+n_r +s}$ defined in
\eqref{eq:torsor}.
In the final part of this section we deduce Theorem \ref{t:ut} from this
asymptotic formula.

\subsection{Representation function and asymptotic formula}
\label{s:repr-fn}
After a change of variables we may assume that we are working with an
integral model for $\V$, defined by the system of equations
\begin{equation*}
 0\neq \nf_{K_i}(\x_i) = f_i(u_1,\dots,u_s), \quad  (1\leq i\leq r),
\end{equation*}
where each $K_i$ is a number field of degree $n_i>1$, each $f_i$ is a
linear form defined over $\ZZ$, and the forms $f_i$ are pairwise
non-proportional. 
We further assume that each $\nf_{K_i}$ is defined using a $\ZZ$-basis 
$\{\omega_{i,1},\dots,\omega_{i,n_i}\}$ for the ring of integers of
$\fo_{K_i}$, so that it too has integer coefficients. 

We will phrase the problem of counting integral points on $\V$ in terms
of representation functions $R_i:\ZZ \to \ZZ_{\geq 0}$ that, in the simplest
instance, count the number of representations $m=\nf_{K_i}(\x_i)$ of
each non-zero integer $m$, where $\x_i$ runs through equivalence classes
with respect to the action of the free part $Y_{K_i}^{(+)}$ of $U_{K_i}^{(+)}$.
Our application to Theorem \ref{t:ut} requires us to incorporate some
flexibility into the definition of $R_i$ as to exactly which representations are counted.
To describe these restrictions, we use the
notation of Section~\ref{s:technical}.
In particular, recall that
$$
\mathfrak{D}_{i,+} = \{\x \in \RR^{n_i}: 
x_1 \omega_{i,1} + \dots + x_{n_i} \omega_{i,n_i}\in \mathfrak{F}_{i,+} \}
$$
is a fundamental domain for the action of $Y_{K_i}^{(+)}$ on the coordinate space
$\RR^{n_i}$. Furthermore, we recall from \eqref{eq:G+} that 
$$
\mathfrak{D}_{i,+}^\epsilon (T)
= \{\x \in \mathfrak{D}_{i,+}: 0<\epsilon\nf_{K_i}(\x)\leq T\},
$$
for $\epsilon\in \{\pm\}$ and  $T\geq 1$.

\begin{definition}[Representation function]\label{def:repr-fn}
Let $i\in \{1,\dots,r\}$ and
let $\mathfrak{X}_i \subset \mathfrak{D}_{i,+}$ be a cone such that each
of the bounded sets 
$\mathfrak{X}_i \cap \mathfrak{D}^{\epsilon}_{i,+}(1)$
has an $(n_i-1)$-Lipschitz parametrisable boundary, unless it is empty.
Let $M\in \NN$ and let 
$\b_i\in (\ZZ/M\ZZ)^{n_i}$ for $1\leq i\leq r$.
For any  $m\in \ZZ$ we define 
$$
R_i(m;\mathfrak{X}_i,\b_i,M)=
\1_{m \not= 0} \cdot~
\#\left\{\x\in \ZZ^{n_i}\cap\mathfrak{X}_{i}:
\begin{array}{l}
\bN_{K_i}(\x)=m\\
\x\=\b_i\bmod{M}
\end{array}
\right\}.
$$
We shall abbreviate $R_i(m)=R_i(m;\mathfrak{X}_i,\b_i,M)$, once
$\mathfrak{X}_i$,$\b_i$ and $M$ are fixed.
\end{definition}

Next, 
let $\mathfrak{K} \subset \RR^{s}$ be any convex bounded
set.
Our interest lies in the counting function
\begin{equation}\label{eq:def-count}
N(T)=
\sum_{\substack{
\u \in \ZZ^s\cap T\mathfrak{K}\\
\u\= \a\bmod{M}}}
\prod_{i=1}^r
R_i(f_i(\u)),
\end{equation}
for given $\a\in (\ZZ/M\ZZ)^s$.
By unravelling the definition of $R_i$, this is seen to express the
number of suitably constrained  points in $\V(\ZZ)$.
For technical reasons, we restrict attention to  $\a\in (\ZZ/M\ZZ)^s$ such that  $p^{v_p(M)} \nmid f_i(\a)$
for any $p\mid M$ and any $1 \leq i \leq r$.

From now on we will view $M, K_1,\dots,K_r$, together with $\a,\b_i$ and
the coefficients of $\V$ as being fixed once and for all.
Any implied constants in our work will therefore be allowed to depend on
these quantities in any way.
Moreover, the regions $\mathfrak{X}_1,\dots,\mathfrak{X}_r$ 
 are also to be considered fixed, with any implied constant
being allowed to depend on the Lipschitz constants of the maps
parametrising the boundaries.

Before revealing our asymptotic formula for $N(T)$ we require a bit more notation. 
For given $q\in \NN$  and $A\in \ZZ$, with $M\mid q$, we let 
$$
\rho_i(q,A;M)=\#\left\{\x\in (\ZZ/q\ZZ)^{n_i}: 
\begin{array}{l}
 \nf_{K_i}(\x)\= A \bmod{q}\\
\x\equiv \b_i \bmod{M}
\end{array}
\right\},
$$
for $1\leq i\leq r$, as in  \eqref{eq:ask}.
Moreover, for $\epsilon\in \{\pm\}$, we define
$$
\RR_{\epsilon}=\{x\in \RR: \epsilon x>0\} 
$$
and 
\begin{equation}\label{eq:def-kappa}
\kappa_i^{\epsilon}(\mathfrak{X}_i) 
= \vol\left(\mathfrak{D}_{i,+}^\epsilon(1) \cap 
  \mathfrak{X}_{i}\right).
\end{equation}
Finally, we denote by $\mathbf{f}:\RR^s\rightarrow \RR^r$ the linear map
defined by the system $\mathbf{f}=(f_1,\dots,f_r)$ of linear forms.
Bearing this notation in mind we have the following result. 

\begin{theorem}\label{t:NB}
Let $f_1, \dots, f_r\in \ZZ[u_1,\dots,u_s]$ be 
pairwise non-proportional linear forms and assume that 
$|f_i(\mathfrak{K})|\leq 1$, for $1\leq i\leq r$.
Suppose that $M$, $\b_i$ and $\a$ are as above;
in particular, $p^{v_p(M)} \nmid f_i(\a)$ for any $p\mid M$ and any $1 \leq i \leq r$.
Then we have 
$$
N(T) = \beta_\infty \prod_p \beta_p \cdot T^s+o(T^{s}), 
\quad (T\rightarrow \infty),
$$
where
$$ 
\beta_{\infty}
=\sum_{\beps \in \{\pm\}^{r}}
 \vol\left(
  \mathfrak{K} \cap 
  \mathbf{f}^{-1}(\RR_{\epsilon_1} \times\dots\times \RR_{\epsilon_r})
 \right)
\prod_{i=1}^r \kappa_i^{\epsilon_i}(\mathfrak{X}_i)
$$
and 
$$
\beta_p=
 \lim_{m\rightarrow \infty} 
 \frac{1}{p^{ms}}\sum_{
 \substack{
 \u\in (\ZZ/p^m\ZZ)^s\\
 \u\=\a\bmod{p^{v_p(M)}}}}
 \prod_{i=1}^r
 \frac{\rho_i(p^m,f_i(\u);p^{v_p(M)})}{p^{m(n_i-1)}},
$$
for each prime $p$.
Furthermore, the product $\prod_p \beta_p$ is absolutely convergent.
\end{theorem}

We will show how Theorem \ref{t:NB} implies Theorem \ref{t:ut} in 
Section \ref{s:implies}.
The proof of Theorem~\ref{t:NB} takes up most of the remainder of this
paper.
The first part is established in the course of
Sections \ref{s:nilsequences}--\ref{s:proof}, 
while the final part is dealt with in Section \ref{s:convergence} below.

\begin{rem}
Our proof uses the machinery developed in Green and Tao \cite{GT}. As
such, it in fact covers the case where in the statement of
Theorem \ref{t:NB} each linear form $f_i$ is replaced by a
linear polynomial $f_i + a_i$, for an integer $a_i = O(T)$.
\end{rem}

\begin{rem}\label{rem:kappa}
In the special case where $\mathfrak{X}_i = \mathfrak{D}_{i,+}$ for 
$1\leq i\leq r$  and $M=1$ in $N(T)$, it is straightforward to adapt the
calculation  in \cite[\S 6]{marcus} to find a precise value for 
$  \kappa_i^{\epsilon}(\mathfrak{D}_{i,+})
= \vol\left(\mathfrak{D}_{i,+}^\epsilon(1) \right)$. 
Let us drop the index $i$ and work with a typical field $K$ of degree
$n$. 
Let $\delta_1,\ldots,\delta_{r_1+r_2-1}$ be generators for $Y_K^{(+)}$.
We define a modified regulator $R_K^{(+)}$ to be the absolute value of the
determinant of the $(r_1+r_2)\times (r_1+r_2)$ matrix, whose rows are
given by $\psi(\delta_1), \dots, \psi(\delta_{r_1+r_2-1}), \u_{r_1+r_2}
\in \RR^{r_1+r_2}$, in the notation of 
Section \ref{s:units}. 
Then one finds that 
$$
 \kappa_i^{\epsilon}(\mathfrak{D}_{i,+})
=
\begin{cases}
0, &\mbox{if $\epsilon=-$ and $r_1=0$,}\\
2^{r_{1}-1} (2\pi)^{r_{2}}R_{K}^{(+)}/\sqrt{|D_{K}|}, & \mbox{otherwise}.
\end{cases}
$$
Observing that 
$\psi(\eta_1^2)=2\psi(\eta_1)$, furthermore, 
an inspection of the explicit choice of generators for $Y_K^{(+)}$ given in
Section \ref{s:units} shows that 
$R_{K}^{(+)} = [Y_K:Y_K^{(+)}] R_{K}$.
Theorem \ref{t:NB} recovers \cite[Thm.~1.1]{lm2} when
$K_1,\dots,K_r$ are all taken to be quadratic.
\end{rem}

\subsection{Convergence of the product of local densities}
\label{s:convergence} 
In this section we prove the absolute convergence of the product
$\prod_p \beta_p$ from Theorem \ref{t:NB}, by establishing an asymptotic
estimate for the local density $\beta_p$, valid whenever $p$ is large
compared to
\begin{equation}\label{eq:def-L}
L=\max_{1\leq i\leq r}
  \left\{\|f_i\|, s, r, |D_{K_i}| \right\} 
\end{equation}
and $p \nmid M$.
Here $\|f_i\|$ denotes the maximum modulus of the coefficients of 
$f_i$.

\begin{proposition}\label{p:1}
We have $\beta_p=1+O_L(p^{-2})$ whenever $p \nmid M$ and
$\beta_p=O_{L}(1)$ when $p|M$.
In particular, there exists $L'=O_L(1)$, which is independent of $M$,
such that $\beta_p > 0$ whenever $p>L'$ and $p \nmid M$.
\end{proposition}

This proposition immediately implies the convergence of the product 
$\prod_p \beta_p$.  
The proof of Proposition \ref{p:1} splits into two cases according to
whether $p$ is large or small compared to $L$, and follows that
of \cite[Lemma 8.3]{lm1}.
The main ingredients are the information that Lemma \ref{lem:A>0} provides
about $\rho(q,A)$, and the properties of local divisor densities,
which we discuss next.

Let
\begin{equation}\label{eq:def-Um}
\mathcal{U}_m=
\{\u\in (\ZZ/p^m\ZZ)^s: \u\=\a\bmod{p^{v_p(M)}}\},
\end{equation}
for any $m\in \NN$.
For given $\c\in \ZZ_{\geq 0}^r$ and 
a given system $\mathbf{f}=(f_1,\dots,f_r)$ as above, 
  we define the {\em local divisor density}
(cf.\ \cite[p.1831]{GT} and \cite[Def.~8.4]{lm1}) to be 
\begin{align}\label{eq:def-div.density}
\alpha_{\mathbf{f}}(p^{c_1},\dots,p^{c_r}) = 
\frac{1}{p^{ms}} \sum_{\u\in \mathcal{U}_m}
\prod_{i=1}^{r} \1_{p^{c_i}\mid f_i(\ma{u})},
\end{align}
where
$m=\max\{c_1,\dots,c_r\}$.
Let $n(\c)$ denote the number of non-zero components of $\c$.
Then
\begin{equation}\label{eq:ev-alpha}
\alpha_{\mathbf{f}}(p^{c_1},\dots,p^{c_r}) \begin{cases}
=1, &\mbox{if $n(\c)=0$,}\\
= p^{-\max_i \{c_i\}}, &  \mbox{if $p\gg_L 1$, $p \nmid M$ and $n(\c)=1$,}\\
\leq p^{-\max_{i\neq j}\{c_i+c_j\}},
      & \mbox{if $p\gg_L 1$, $p\nmid M$ and $n(\c)>1$,}\\
\ll_{L} p^{-\max_i \{c_i\}}, &
 \mbox{otherwise.}
\end{cases}
\end{equation}
It is important to note here that even when $p\mid M$ and $n(\c)\geq 1$ the implied constant in the final estimate does not depend on $M$.
Moreover, here (and elsewhere) we take $p\gg_L 1$ to mean that $p$ is sufficiently large in terms of $L$. 
An easy way to bound sums over divisor densities uses the observation that
there are at most $rJ^{r-1}$ choices of $\k\in\ZZ_{\geq 0}^r$ such that
$\max_{i} k_i=J$ and therefore
\begin{equation}\label{eq:jen}
\sum_{J\geq J_0}
\sum_{\substack{\k\in \ZZ_{\geq 0}^r\\ \max\{k_1,\dots, k_r\}=J}}
\frac{J^T}{p^{J}} 
\leq r\sum_{J\geq J_0}
\frac{J^{T+r-1}}{p^{J}}
\ll_{T,r,J_0} \frac{1}{p^{J_0}} 
\sum_{J\geq 0} \frac{J^{T+r-1}}{2^{J}}
\ll_{T,r,J_0} \frac{1}{p^{J_0}}
\end{equation}
for any $T,J_0>0$.

\begin{proof}[Proof of Proposition \ref{p:1}]
We may write $\beta_p=\lim_{m\rightarrow \infty} \beta_p(m)$,
with
\begin{equation}\label{eq:beta_p(m)}
\beta_p(m)=
\frac{1}{p^{ms}}
\sum_{\substack{\k\in \ZZ_{\geq 0}^r}}
\sum_{\substack{\u\in \mathcal{U}_m\\v_p(f_i(\u))=k_i}}
\prod_{i=1}^{r}\frac{\rho_i(p^m,f_i(\u);p^\mu)}{p^{m(n_i-1)}}
\end{equation}
and $\mu=v_p(M)$.
We begin by analysing $\beta_p$ when $p$ is small.
In fact we will show that  $\beta_p=O_{L}(1)$, for any prime $p$, 
which suffices for Proposition \ref{p:1}.

Since
$$
\rho_i(p^m,f_i(\u);p^\mu)\leq \rho_i(p^m,f_i(\u)),
$$ 
an application of the first part of Lemma~\ref{lem:A>0} in
\eqref{eq:beta_p(m)}
shows that
\begin{align*}
\beta_p(m)
&\leq 
\frac{1}{p^{ms}}
\sum_{J\geq 0}
\sum_{\substack{\k\in \ZZ_{\geq 0}^r\\ \max\{k_1,\dots,k_r\}=J}}
\sum_{\substack{\u\in \mathcal{U}_m\\ v_p(f_i(\u))=k_i}}
\prod_{i=1}^{r}\frac{\rho_i(p^m,f_i(\u))}{p^{m(n_i-1)}}\\
&\ll
\frac{1}{p^{ms}}
\sum_{J\geq 0} 
\sum_{\substack{\k\in \ZZ_{\geq 0}^r\\ \max\{k_1,\dots,k_r\}=J}}
\sum_{\substack{\u\in \mathcal{U}_m\\ p^{k_i}\mid f_i(\u)}}
\min\{m,J+1\}^{n_1+\dots +n_r}.
\end{align*}
Next we invoke
\eqref{eq:ev-alpha}
and \eqref{eq:jen} to obtain
\begin{align*}
\beta_p(m)
&\ll_{L}
\sum_{J\geq 0} 
\sum_{\substack{\k\in \ZZ_{\geq 0}^r\\ \max\{k_1,\dots,k_r\}=J}}
\frac{\min\{m,J+1\}^{n_1+\dots +n_r}}{p^J}
\ll_{L} 1.
\end{align*}
Taking the limit $m\rightarrow \infty$, this shows that $\beta_p=O_{L}(1)$,
as required for Proposition \ref{p:1}.

We proceed to  analyse $\beta_p$ when $p\gg_L 1$ and $p\nmid M$.
In particular, we have $\mu=0$ and $\mathcal{U}_m=(\ZZ/p^m\ZZ)^s$.
Let $\mathcal{K}_1=(\ZZ\cap[0,m))^r$ and let
$\mathcal{K}_2=\ZZ_{\geq 0}^r \setminus [0,m)^r$.
Accordingly, we write $\beta_p(m)=\beta_p^{(1)}(m)+\beta_p^{(2)}(m)$, where 
$\beta_p^{(i)}(m)$ is the contribution from $\k\in \mathcal{K}_i$.

Since $p\gg_L 1$, it follows from \eqref{eq:def-L} that $p\nmid D_{K_i}$ for each $1\leq i\leq r$. Thus the second part of Lemma~\ref{lem:A>0} implies that  
\begin{align*}
\beta_p^{(1)}(m)
&=
c_p(K_1)\dots c_p(K_r)
\frac{1}{p^{ms}}
\sum_{\k\in \mathcal{K}_1}
\sum_{\substack{\u\in \mathcal{U}_m\\
v_p(f_i(\u))=k_i
}}
\prod_{i=1}^{r} r_{K_i}(p^{k_i})\\
&=
c_p(K_1)\dots c_p(K_r)
\frac{1}{p^{ms}}
\sum_{\k\in \mathcal{K}_1}
\left(\prod_{i=1}^{r} r_{K_i}(p^{k_i})\right)
\sum_{\substack{\u\in \mathcal{U}_m\\
v_p(f_i(\u))=k_i}}1,
\end{align*}
where
$$
c_p(K_i)=\left(1-\frac{1}{p}\right)^{-1}
\prod_{\fp \mid p,\, \fp \subset \fo_{K_i}} \left(1-\frac{1}{\n
\fp}\right).
$$
For given $\k\in \mathcal{K}_1$, we have
\begin{align*}\frac{1}{p^{ms}}
\sum_{\substack{\u\in \mathcal{U}_m\\
v_p(f_i(\u))=k_i}}1 
&=
\sum_{\bve\in \{0,1\}^r}
(-1)^{\ve_1+\dots+\ve_r}
\alpha_{\mathbf{f}}(p^{k_1+\ve_1},\dots,p^{k_r+\ve_r})\\
&=
\begin{cases}
1-rp^{-1} +O_r(p^{-2}),
&\mbox{if $\k=\ma{0}$,}\\
p^{-1} +O_r(p^{-2}),
&\mbox{if $\k\in \{(1,0,\dots,0), \dots,(0,\dots,0,1)\}$,}\\
O_r(p^{-\max \{k_1,\dots,k_r\}}),
&\mbox{if $n(\k)=1$, $\max\{k_1,\dots,k_r\}>1$,}\\
O_r(p^{-1-\max \{k_1,\dots,k_r\}}),
&\mbox{otherwise,}
\end{cases}
\end{align*}
by \eqref{eq:ev-alpha}.

Since $r_{K_i}(p^k)=O((k+1)^{n_i})$, by \eqref{eq:r-upper}, we deduce from
\eqref{eq:jen} that
$$
\sum_{J \geq 2} 
\sum_{\substack{\k\in \mathcal{K}_1\\ \max\{k_1,\dots, k_r\}=J\\
n(\k)=1}}
p^{-J}
\prod_{i=1}^{r} r_{K_i}(p^{k_i})
+
\sum_{J\geq 1}
\sum_{\substack{\k \in \mathcal{K}_1\\ \max\{k_1,\dots, k_r\}=J \\
n(\k)>1 }}
p^{-J-1}
\prod_{i=1}^{r} r_{K_i}(p^{k_i})
\ll_r \frac{1}{p^2}.
$$
Hence \eqref{eq:r} implies that
\begin{align*}
\beta_p^{(1)}(m)
&=
c_p(K_1)\dots c_p(K_r)
\left(1+
\sum_{i=1}^r\frac{r_{K_i}(p)-1}{p}+O_r\left(\frac{1}{p^2}\right)\right)\\
&=
1+O_r\left(\frac{1}{p^2}\right).
\end{align*}
Putting everything together, we conclude that 
$$
\left| \beta_p(m) - 1\right|
\ll_r
\frac{1}{p^2}+  
\beta_p^{(2)}(m).
$$
The first part of Lemma \ref{lem:A>0} 
can be used to show that  $\beta_p^{(2)}(m)$ is at most
\begin{align*}
\frac{1}{p^{ms}}
\sum_{\k\in \mathcal{K}_2}
\sum_{\substack{\u\in \mathcal{U}_m\\
p^{k_i}\mid f_i(\u)
}}
\prod_{i=1}^{r}\frac{\rho_i(p^m,f_i(\u))}{p^{m(n_i-1)}}
&\ll
\frac{m^{n_1+\dots +n_r}}{p^{ms}}
\sum_{\k\in \mathcal{K}_2}
\sum_{\substack{\u\in \mathcal{U}_m\\
p^{k_i}\mid f_i(\u)
}}1\\
&=m^{n_1+\dots +n_r}
\sum_{\k\in \mathcal{K}_2}
\alpha_{\mathbf{f}}(p^{k_1},\dots,p^{k_r})\\
&\ll
\frac{m^{n_1+\dots +n_r}}{p^{1+m}},
\end{align*}
by \eqref{eq:ev-alpha}.
Substituting this into our expression for $\beta_p$ and 
taking the limit $m\rightarrow \infty$,  this completes the proof of
Proposition \ref{p:1} when $p\gg_L 1$.
\end{proof}

\subsection{Deduction of Theorem \ref{t:ut}}\label{s:implies}

We proceed to show how Theorem \ref{t:ut} follows from Theorem \ref{t:NB}.
Our task is to establish the Hasse principle and weak approximation for
the smooth variety $\V\subset \AA_\QQ^{n_1+\dots+n_r +s}$, which after
the reductions from the start of Section \ref{s:repr-fn} is given by 
$$
0\neq \bN_{K_i}(\x_i) = f_i(u_1,\dots,u_s), \quad  (1\leq i\leq r),
$$
for pairwise non-proportional linear forms $f_1,\dots,f_r$ defined over
$\ZZ$.

Suppose that we are given a point $(\u,\x_i)\in \V(\QQ)$. 
Then  each point in the orbit 
$\{(\u,\eta_i.\x_i):
 \eta_i \in U_{K_i}^{(+)}\}$ 
also belongs to $\V(\QQ)$.
We will therefore content ourselves with looking for points
$(\u,\x_i)\in \V(\QQ)$ such that each $\x_i$ lies in the fundamental
domain $\mathfrak{D}_{i,+}$ which we constructed in Lemma \ref{lem:F+} and 
\eqref{eq:D+} for the free part of $U_{K_i}^{(+)}$. 
We will call such points {\em primary}.

Let $\Omega$ denote the set of places of $\QQ$. 
We assume we are given  points 
$(\u^{(\nu)},\x_i^{(\nu)})\in \V(\QQ_\nu)$ for every $\nu\in \Omega$.
By possibly replacing the adelic point 
$(\x_i^{(\nu)})_{\nu \in \Omega}$ by 
$(\eta_i.\x_i^{(\nu)})_{\nu \in \Omega}$ for an appropriate
$\eta_i \in U_{K_i}^{(+)}$, 
we may assume that $\x_i^{(\infty)}$ belongs to $\mathfrak{D}_{i,+}$ for
each $1 \leq i\leq r$.
Let $S$ be any finite set of places, including the archimedean place as
well as all non-archimedean places corresponding to primes $p < L'$,
where $L'=O(1)$ was determined in Proposition~\ref{p:1}.

Let $\ve>0$. 
Then, in order to prove Theorem \ref{t:ut}, it suffices to show that 
there is a primary point $(\u,\x_i)\in \V(\QQ)$ 
such that
\begin{equation} \label{eq:WA-conditions}
  |\u-\u^{(\nu)}|_\nu<\ve, \quad 
  |\x_i-\x_i^{(\nu)}|_\nu<\ve, \quad (1\leq i\leq r),
\end{equation}
for every $\nu \in S$. 
Here, $|\cdot|_\nu$ denotes the $\nu$-adic norm extended to vectors in
the obvious way, and we follow the convention that 
$|\cdot|_\infty=|\cdot|$.

On rescaling appropriately we may assume that the points 
$(\u^{(\nu)},\x_i^{(\nu)})$ that we are given belong
to $\ZZ_\nu^{n_1+\dots+n_r+s}$ for every finite $\nu\in S$.
By the Chinese remainder theorem we can then produce an
integer vector
$(\u^{(M)},\x_i^{(M)})$ such that 
\begin{equation}\label{eq:horse}
 |\u^{(M)}-\u^{(\nu)}|_\nu<\ve, \quad
 |\x_i^{(M)}-\x_i^{(\nu)}|_\nu<\ve,  \quad (1\leq i\leq r),
\end{equation}
for all finite $\nu \in S$.
We now seek integral points $(\u,\x_i)\in \V(\ZZ)$ satisfying the
following local conditions.
For the finite places we impose
\begin{equation}\label{eq:size1}
\begin{split}
&\u\equiv \u^{(M)} \bmod{M}, \quad
\x_i\equiv \x_i^{(M)} \bmod{M}, \quad (1\leq i\leq r),
\end{split}
\end{equation}
for an appropriate modulus $M\in\NN$ with the property that
$p \nmid M$ whenever $p \not\in S$.
In view of \eqref{eq:horse} these conditions imply \eqref{eq:WA-conditions}. 
To guarantee that 
\begin{equation}\label{eq:technical}
p^{v_p(M)}\nmid  f_i(\u^{(M)}) \text{ for all $p\mid M$}, 
\quad (1\leq i\leq r),
\end{equation}
it suffices to choose $\ve$ sufficiently small, since 
$f_i(\u^{(\nu)}) \neq 0$ in $\QQ_\nu$.
 
For the infinite place we impose that 
\begin{equation}\label{eq:size2}
|\u-B\u^{(\infty)}|<\ve B, \quad
|\x_i-B^{1/{n_i}}\x_i^{(\infty)}|<\ve B^{1/{n_i}}, \quad
(1\leq i\leq r),
\end{equation}
with $B=P^{n_1\dots n_r}$ and $P\in \NN$  tending to infinity
such that $P\equiv 1 \bmod{M}$.
Thus any point $(\u,\x_i) \in \V(\ZZ)$ satisfying \eqref{eq:size1} and
\eqref{eq:size2} gives rise to $(B^{-1}\u,B^{-1/n_i}\x_i) \in \V(\QQ)$
satisfying the original condition \eqref{eq:WA-conditions}.
We aim to detect the existence of integral points satisfying
\eqref{eq:size1} and \eqref{eq:size2} using Theorem \ref{t:NB}. 
For this reason, we now proceed to replace \eqref{eq:size2} by a condition
that is more suitable for an application of the theorem.

Let $1\leq i\leq r$ and let $\ve'>0$.
We begin by defining a cone that is symmetric
about $\x_i^{(\infty)}$, via
$$
\mathfrak{B}_i(\x_i^{(\infty)};\eps')=
\left\{\x\in \RR^{n_i} \cap \mathfrak{D_{i,+}}:  
\left| \frac{\x}{|\nf_{K_i}(\x)|^{1/n_i}} 
     - \frac{\x_i^{(\infty)}}{|\nf_{K_i}(\x_i^{(\infty)})|^{1/n_i}} \right|
< \eps' 
\right\}.
$$ 
Note that $\mathfrak{B}_i(\x_i^{(\infty)};\eps')\neq \emptyset$, 
since $\x_i^{(\infty)}\in \mathfrak{D}_{i,+}$ by our work above.
It follows from Sections~\ref{s:units} and~\ref{s:geometry-of-numbers},
that for $\epsilon \in \{\pm\}$ each  
$\mathfrak{D}_{i,+}^\epsilon(1) \cap \mathfrak{B}_i(\x_i^{(\infty)};\ve')$
is either empty or such that Lemma \ref{lem:2} applies.
Indeed, these sets arise as the intersection of a bounded convex
set with a set that has an $(n_i-1)$-Lipschitz parametrisable boundary.
Moreover, we clearly have 
\begin{equation}\label{eq:kappa-eps>0}
\vol(\mathfrak{D}_{i,+}^\epsilon(1) \cap
\mathfrak{B}_i(\x_i^{(\infty)};\ve')) > 0
\end{equation}
when $\epsilon=\sign (\nf_{K_i}(\x_i^{(\infty)}))$. 
The second condition in \eqref{eq:size2} now holds whenever
\begin{equation}\label{eq:size2-sufficient}
\x_i \in 
\mathfrak{B}_{i}(\x_i^{(\infty)}; \eps')
\cap 
\{\x \in \RR^{n_i} :
 |\nf_{K_i}(\x) - B\nf_{K_i}(\x_i^{(\infty)})| < \eps'B \} 
\end{equation}
for sufficiently small $\eps'$
in terms of $\max_{1\leq i\leq r}|\nf_{K_i}(\x_i^{(\infty)})|$. We  fix such a choice of  $\eps'<\eps$.

In view of the first part of \eqref{eq:size2}, we define the convex
bounded region
$$
\mathfrak{K}(\u^{(\infty)};\eps'')
=\{\u \in \RR^s: |\u-\u^{(\infty)}|<\eps''\},
$$
for $\ve''>0$. 
Observe that for  sufficiently small   $\eps''<\eps$,
the condition 
$\u \in B\mathfrak{K}(\u^{(\infty)}; \eps'')$ 
implies both the first part of \eqref{eq:size2} and, furthermore, 
$$
|f_i(\u) - Bf_i(\u^{(\infty)})| < \eps'B.
$$
In conclusion, any point $(\u,\x_i) \in \V(\RR)$ with 
\begin{equation}\label{eq:size3}
 \u \in B\mathfrak{K}(\u^{(\infty)}; \eps'') \quad 
 \text{and} \quad
 \x_i \in \mathfrak{B}_{i}(\x_i^{(\infty)}; \eps')
\end{equation}
satisfies \eqref{eq:size2-sufficient} and therefore also
\eqref{eq:size2}.

With these choices of $\ve',\ve''$, we fix the representation
functions 
$$ R_i(m) = R_i(m;\mathfrak{B}_i(\x_i^{(\infty)};\eps'),\x_i^{(M)},M)$$
from  Definition \ref{def:repr-fn}, 
for $1\leq i \leq r$.
It is clear that
$$|f_i(\mathfrak{K}(\u^{(\infty)}; \eps''))| \subset [-H,H]$$
for $H=s(1+\eps)|\u^{(\infty)}|\cdot \max_i \|f_i\|$. In particular, 
$H \asymp 1$.
Thus we observe that, on the one hand, 
$$
N(H B)=
\sum_{\substack{
\u \in \ZZ^s\cap HB (H^{-1}\mathfrak{K}(\u^{(\infty)};\ve''))\\
\u\= \u^{(M)}\bmod{M}}}
\prod_{i=1}^r
R_i(f_i(\u))
$$
counts exactly the primary points  $(\u,\x_i) \in \V(\ZZ)$ which satisfy
\eqref{eq:size1} and \eqref{eq:size3}.
On the other hand, $N(HB)$ takes the shape of the counting
function \eqref{eq:def-count} from Theorem \ref{t:NB}, with
$$
\mathfrak{K}=H^{-1}\mathfrak{K}(\u^{(\infty)};\ve''), \quad 
\mathfrak{X}_i=\mathfrak{B}_{i}(\x_i^{(\infty)};\ve'), \quad 
\a=\u^{(M)}, \quad \b_i=\x_i^{(M)}.
$$
Moreover, all the conditions of Theorem \ref{t:NB} are satisfied. We conclude, for
$T=HB$, that 
$$N(HB)= (HB)^s \beta_\infty \prod_p \beta_p + o(B^s),$$
where the error term may depend on 
$\ve'$, $\ve''$, $M$, $K_1,\dots,K_r$, as well as
on $\u^{(M)}$, $\u^{(\infty)}$, $\x_i^{(M)}$, $\x_i^{(\infty)}$ 
and on the coefficients of $\V$.
All that remains now,  in order to deduce Theorem~\ref{t:ut}, is to show
that 
$$
\beta_\infty \prod_p \beta_p\gg 1.
$$

Beginning with $\beta_\infty$ we recall that 
$$
\kappa_i^{\epsilon}
=\kappa_i^{\epsilon}(\mathfrak{B}_{i}(\x_i^{(\infty)};\ve'))
= \vol\left(\mathfrak{D}_{i,+}^{\epsilon}(1) \cap 
  \mathfrak{B}_{i}(\x_i^{(\infty)};\ve')\right).
$$
Thus, \eqref{eq:kappa-eps>0} yields $\kappa_i^{\epsilon_i}>0$ when 
$\epsilon_i=\sign(f_i(\u^{\infty}))$ for each $1\leq i\leq r$.
Next we check that for these choices of 
$\epsilon_1,\dots, \epsilon_r$ we also have
$$\vol(H^{-1}\mathfrak{K}(\u^{(\infty)};\ve'') \cap
\mathbf{f}^{-1}(\RR_{\epsilon_1} \times \dots \times \RR_{\epsilon_r}))
>0.$$
Since $f_1, \dots, f_r$ are linear homogeneous polynomials, the region 
$\mathbf{f}^{-1}(\RR_{\epsilon_1} \times \dots \times \RR_{\epsilon_r})$
is a cone.
Thus
\begin{align*}
\vol\big(H^{-1}&\mathfrak{K}(\u^{(\infty)};\ve'') \cap
\mathbf{f}^{-1}(\RR_{\epsilon_1} \times \dots \times \RR_{\epsilon_r})
\big) \\
&= H^{-s} \vol\big( \mathfrak{K}(\u^{(\infty)};\ve'') \cap
\mathbf{f}^{-1}(\RR_{\epsilon_1} \times \dots \times \RR_{\epsilon_r})
\big),
\end{align*}
which is positive, since $\u^{(\infty)}$ is an element of the open
set 
$\mathbf{f}^{-1}(\RR_{\epsilon_1} \times \dots \times \RR_{\epsilon_r})$.

Turning to the local factors at the non-archimedean places $p$, we recall that 
$$
\beta_p=
 \lim_{m\rightarrow \infty} 
 \frac{1}{p^{ms}}\sum_{
 \substack{
 \u\in \mathcal{U}_m}}
 \prod_{i=1}^r
 \frac{\rho_i(p^m,f_i(\u);p^{v_p(M)})}{p^{m(n_i-1)}},
$$
where $\mathcal{U}_m$ is given by \eqref{eq:def-Um}, with 
$\a=\u^{(M)}$.
By construction, we have $p \nmid M$ whenever $p \not\in S$ and
it follows from Proposition \ref{p:1} that 
$$
\prod_{p\not\in S} \beta_p \gg 1,
$$
and that $\beta_p = O(1)$ for $p \in S$.
This leaves us to show that $\beta_p>0$ for every $p\in S$ in
order to complete the proof of Theorem \ref{t:ut}.
We will deduce this with the help of Lemma \ref{lem:C6.4} from the
existence of local solutions at these primes.
While we are primarily interested in $p \in S$, the following argument works for 
any prime $p$.

With Lemma \ref{lem:C6.4} in mind, we fix a prime $p$ and let
$$
m' 
=  2\Big( 1+v_p(M) 
 + \sum_{i=1}^r v_p(f_i(\u^{(M)})) + \sum_{i=1}^r  v_p(n_i) \Big).
$$
Recall that we are given $(\u^{(p)},\x_i^{(p)}) \in \cV(\QQ_p)$ such
that \eqref{eq:horse} holds for $\nu = p$.
By solving this approximation problem for a smaller value of $\eps$,
we can find 
$(\u',\x'_i) \in \ZZ^{n_1+\dots+n_r+s}$ such that
$$(\u',\x'_i) \equiv (\u^{(M)},\x_i^{(M)}) \bmod{p^{v_p(M)}}
\quad \text{ and } \quad
f_i(\u') \equiv \bN_{K_i}(\x'_i) \bmod{p^{m'}}.
$$
Thus,
$$
 \prod_{i=1}^r
 \rho_i(p^{m'},f_i(\u');p^{v_p(M)})
 \geq 1,
$$ 
where each $\rho_i$ is defined with respect to $\x_i^{(M)}$.
The technical condition \eqref{eq:technical} ensures that such
approximations $\u'$ satisfy $v_p(f_i(\u'))=v_p(f_i(\u^{(M)}))$
and $f_i(\u')\neq 0$.

The definition of $m'$ ensures that the
conditions of Lemma \ref{lem:C6.4} are satisfied when 
$m\geq m'$, $\ell=v_p(M)$,  $A=f_i(\u')$, $\a=\x_i^{(M)}$ and
$G=\bN_{K_i}$, for any $1\leq i\leq r$.
Hence we obtain
\begin{align*}
 \prod_{i=1}^r 
 \frac{\rho_i(p^{m},f_i(\u);p^{v_p(M)})}{p^{m(n_i-1)}}
&=\prod_{i=1}^r 
 \frac{\rho_i(p^{m'},f_i(\u');p^{v_p(M)})}{p^{m'(n_i-1)}}\\
&\geq \frac{1}{p^{m'(n_1+\dots+n_r - r)}},
\end{align*}
whenever 
$\u \in \mathcal{U}'_{m}=\{ \u\in (\ZZ/p^m\ZZ)^s:
\u\=\u'\bmod{p^{m'}}\}$.
The set $\mathcal{U}'_{m}$ has $p^{(m-m')s}$ elements and is clearly a
subset of $\mathcal{U}_m$.
Therefore
\begin{align*}
\beta_p(m)
&\geq 
 \frac{1}{p^{ms}}\sum_{\u\in \mathcal{U}_m'}
 \prod_{i=1}^r
 \frac{\rho_i(p^m,f_i(\u);p^{v_p(M)})}{p^{m(n_i-1)}}\\
&\geq
 \frac{1}{p^{m'(n_1+\dots+n_r + s - r)}},\end{align*}
for every $m \geq m'$, which provides the desired inequality 
$\beta_p>0$ for $p\in S$.

\bigskip
\section{$W$-trick and non-correlation with nilsequences}
\label{s:nilsequences}

The balance of this paper is dedicated to the proof of  Theorem \ref{t:NB}.  
Our proof proceeds via the methods from \cite{GT} and therefore splits
into two tasks. 
This section accomplishes one them.
Recall Definition \ref{def:repr-fn} of $R_i$ for $1\leq i\leq r$.
We show here that the function $R_i$, when passing to suitable
subprogressions and subtracting off its mean value, 
does not correlate with nilsequences.
In Sections \ref{s:gen-maj} and \ref{s:majorant} we deal with the second
task and construct a pseudorandom majorant for $R_i$. 
To ease notation we shall drop the subscript $i$ and consider the
representation function associated to a typical $K$ of degree $n$.

In order for an arithmetic function to be orthogonal to nilsequences, it
first of all needs to be equidistributed in residue classes to small
moduli.
That is, its average value should not change when passing to
subprogressions with respect to small moduli.
For this to be valid in our situation, we will choose a product $W$ 
of powers of small primes, split
$$R(m) = \sum_{A \bmod{W}} R(m)\1_{m \equiv A \bmod{W}}$$
and consider each of the functions $m \mapsto R(W m + A)$ separately.
This operation is called the ``$W$-trick'' and was introduced in
\cite{GT-longAPs}.

Following \cite[p.260]{lm1}, let $w(T)= \log \log T$ and let
\begin{equation} \label{def:W}
W 
= \prod_{p \leq w(T)} p^{\alpha(p)}, 
\end{equation}
where
 $\alpha(p)=\lceil (C_1+1) \log_p \log T\rceil$
for a constant $1 \leq C_1 \ll 1$ to be specified in 
Proposition~\ref{p:unexceptional'}.
In particular,
$$
p^{\alpha(p)-1} < (\log T)^{C_1 + 1} \leq p^{\alpha(p)}.
$$
Taking $T$ sufficiently large, we may henceforth assume that $M\mid W$. 
Moreover, it is clear that $W=O(T^{o(1)})$.

Our first result concerns the average order of the $W$-tricked
functions $m \mapsto R(Wm+A)$.

\begin{lemma}\label{lem:average-order}
Let $R=R(m;\mathfrak{X},\b,M)$ and let $\epsilon\in \{\pm\}$ be such
that $\mathfrak{X} \cap \mathfrak{D}_{+}^{\epsilon} \not= \emptyset$.
For any  $q\in \NN$ with $M\mid q$, we have 
\begin{equation*}
\sum_{\substack{0< \epsilon m \leq x \\ m \equiv A \bmod{q}}} R(m)
= \frac{\rho(q,A;M)}{q^{n}} \kappa^\epsilon x
 + O(qx^{1-1/n}),
\end{equation*} 
where $\rho(q,A;M)$ is given by \eqref{eq:def-rho} and \eqref{eq:ask} with $\x_0=\b$ and 
$\kappa^\epsilon=\kappa^\epsilon(\mathfrak{X})$ is given by
\eqref{eq:def-kappa}.
\end{lemma}

\begin{proof}
Let us write 
$
\mathfrak{X}(x)
=\mathfrak{X} \cap \mathfrak{D}_{+}^{\epsilon}(x).
$
Breaking the given sum over $R(m)$ into residue classes, we find that 
\begin{align*}
\sum_{\substack{0< \epsilon m \leq x \\ m \equiv A \bmod{q}}} R(m)
&= \sum_{\substack{\y\in (\ZZ/q\ZZ)^n\\
\nf_K(\y)\=A\bmod{q}\\
\y\=\b \bmod{M}}}
\#\left\{ \x\in \ZZ^n\cap \mathfrak{X}(x):  
 \x\=\y\bmod{q}
\right\}.
\end{align*}
The inner cardinality equals
 $\#\left(\ZZ^n\cap q^{-1}(\mathfrak{X}(x)-\y)\right)$, which in turn
equals
\begin{align*}
\frac{\kappa^\epsilon(\mathfrak{X}) x}{q^{n}}
+O\big(q^{1-n}x^{1-1/n}\big),
\end{align*}
by Lemma \ref{lem:2}. 
The statement of the lemma easily follows.
\end{proof}

The  results that follow  no longer hold for arbitrary residue classes 
$A \bmod{W}$ and we will be forced to work with  the set of \emph{unexceptional}
residue classes
\begin{equation}\label{eq:def-A}
\cA=\left\{A \bmod{W}: 
\begin{array}{l}
 0 \leq v_p(A) < v_p(W)/3 \text{ for all } p<w(T) \cr
 0 \leq v_p(A) < v_p(M) \text{ for all } p\mid M \cr
 \rho(W,A;M) >0
\end{array}
\right\}. 
\end{equation}
To justify this, we shall see in Proposition \ref{p:unexceptional'} that integers 
that are divisible by a large prime power make a negligible contribution 
to the asymptotic formula in Theorem \ref{t:NB}. 
Consequently, such integers may be excluded from consideration altogether.
Next, in view of our assumption that 
$p^{v_p(M)} \nmid f_i(\a)$
for any $p\mid M$ and any $1 \leq i \leq r$, it is clear that there is no contribution 
from progressions 
$\{m \equiv A \bmod{W} \}$ such that 
$v_p(A)\geq v_p(M)$ for any $p\mid M$.
Finally, when $\rho(W,A;M)=0$ then $R$ is identically $0$ on the progression
$\{m \equiv A \bmod{W} \}$ and so we may exclude these residue classes $A$
as well.

The next result shows that the function $m \mapsto R(W m + A)$ is
equidistributed in residue classes to $w(T)$-smooth moduli whenever
$A \bmod{W}$ is an unexceptional residue.

\begin{lemma} \label{lem:major-arc}
Let $\epsilon\in \{\pm\}$ be such that 
$\mathfrak{X} \cap \mathfrak{D}_{+}^{\epsilon} \not= \emptyset$.
Let  $A$ be a representative of a class from $\cA$ such that 
$0< \epsilon A < W$, and let $a,q \in \ZZ$ be such that
$0\leq \epsilon a < q< T/W$.
Suppose further that $q$ is $w(T)$-smooth, and assume that
$T',T'' \in \NN$ such that 
$T' \asymp T/W$ and $T'' \asymp T/(qW)$.
Then we have 
\begin{align*}
 \EE_{0 \leq \epsilon m < T'} R(W m + A)
=\EE_{0 \leq \epsilon m < T''} R(W (q m + a) + A)
 + O(q^2 T^{-1/n+o(1)}).
\end{align*}
\end{lemma}
\begin{proof}
Let $E_1$ denote the sum on the left hand side and let $E_2$ be the sum on the
right hand side. 
Since $T'$ is an integer, the summation range of $E_1$ may be written
as $0 \leq \epsilon m \leq T'-1$.
Since $m$ and $A$ are both of sign $\epsilon$, we deduce that
$0< \epsilon(W m + A) \leq (T'-1)W + \epsilon A < T'W$.
Thus after a change of variable we have
$$
E_1
=\frac{1}{T'}
\sum_{\substack{0< \epsilon m' < T'W \\ m' \equiv A \bmod{W}}} R(m'),
$$
and, similarly, 
$$
E_2
=\frac{1}{T''}
\sum_{\substack{0< \epsilon m' < T''Wq \\ m' \equiv A+Wa \bmod{Wq}}}
R(m').
$$
Recall that $M\mid W$.
Two applications of Lemma \ref{lem:average-order} therefore imply that it
suffices to prove that
\begin{equation}\label{eq:him}
\frac{\rho(W,A;M)}{W^{n-1}} 
= \frac{\rho(W q,A + W a;M)}{(W q)^{n-1}}.
\end{equation}
But this follows from the Chinese remainder theorem and applications
of Lemma \ref{lem:C6.4} for each prime $p<w(T)$.
Indeed, let $p<w(T)$, $G=\nf_K$, $\ell=v_p(M)$ and let $m$ be any integer
such that $m \geq v_p(W)$.
Since $A$ describes an unexceptional residue class, we have
$A \not\equiv 0 \bmod{p^{v_p(W)}}$ and furthermore
$$
v_p(M) + v_p(A) + v_p(n) \leq \frac{v_p(W)}{3} + O(1) < \frac{m}{2}, 
$$
provided $T$ is sufficiently large.
Hence, the conditions of Lemma \ref{lem:C6.4} are satisfied for large
$T$ and we deduce \eqref{eq:him} by applying this lemma once for each
value of $m$ in the range $v_p(W) \leq m < v_p(Wq)$.
\end{proof}

The next goal is to establish that the normalised counting function 
$$
m \mapsto \frac{W^{n-1}}{\rho(W,A;M)}R(W m +A)
$$  
does not correlate with nilsequences if $A$ is unexceptional.
A  discussion of the  various objects appearing in the following
proposition may be found in \cite[\S\S 13--15]{lm1}.
A thorough treatment is contained in \cite{GT-polynomialorbits}, which is
the paper that the results from \cite[\S\S 14--16]{lm1} build on and
extend.

\begin{proposition}\label{p:nilsequences}
 Let $G/\Gamma$ be a nilmanifold of dimension $m_G \geq 1$, let
$G_{\bullet}$ be a filtration of $G$ of degree $d \geq 1$, and let
$g \in \mathrm{poly}(\ZZ,G_{\bullet})$ be a polynomial sequence.
Suppose that $G/\Gamma$ has a $Q$-rational Mal'cev basis $\mathcal X$
for some $Q \geq 2$, defining a metric $d_{\mathcal X}$ on $G/\Gamma$.
Suppose that $F: G/\Gamma \to [-1,1]$ is a Lipschitz function.
Then for $\epsilon\in \{\pm\}$, $T'\asymp T/W$ and $A \in \ZZ$ with 
$A \bmod{W} \in \cA$ and $0 \leq \epsilon A < W$,
we have the estimate
\begin{align*}
\Big|
  \EE_{0< \epsilon m \leq T'} 
  \Big(
    R(W m +A) &- \frac{\rho(W,A;M)}{W^{n-1}} \kappa^{\epsilon}
  \Big)
  F(g(|m|)\Gamma)
\Big| \\
&\ll_{m_G,d,E} 
  \frac{\rho(W,A;M)}{W^{n-1}}
  Q^{O_{m_G,d,E}(1)}
 \frac{1+ \|F\|_{\mathrm{Lip}}}{(\log \log \log T)^{E}},
\end{align*}
for any $E>0$.
\end{proposition}

Exactly as in  \cite[Props.~ 17.1 and 17.2]{lm1}
we deduce the above proposition from a special case involving only 
``minor arc nilsequences''.   
This reduction is modelled upon  \cite[\S2]{GT-nilmobius}
and we will not give the details.
The key ingredients are Lemma \ref{lem:major-arc}  and
\cite[Thm.~16.4]{lm1}, which is a factorisation theorem for
nilsequences.
Due to the similar set-up, the choice of parameters from the proof
of \cite[Prop.~17.1]{lm1} remains unchanged.

\begin{proposition}
Let $\epsilon\in \{\pm\}$, $T'\asymp T/W$ and $A \bmod{W} \in \cA$ with
$0 \leq \epsilon A < W$.
Suppose that $\delta\in (0,1/2)$ and $S=O(T^{o(1)})$ are
parameters such that $\delta^{-t} \ll_t T$, for all $t\in \NN$.
 Assume that $(G/\Gamma, d_{\mathcal X})$ is an $m_G$-dimensional
nilmanifold with a filtration $G_{\bullet}$ of degree $d$ and
that $g \in \mathrm{poly}(\ZZ,G_{\bullet})$. 
Finally, suppose that for every $w(T)$-smooth number $\tilde q \leq S$ the
finite sequence $(g(\tilde q m)\Gamma)_{0<m\leq T'/\tilde q}$ is totally 
$\delta$-equidistributed in $G/\Gamma$. 

For every Lipschitz function $F:G/\Gamma \to [-1,1]$ satisfying
$\int_{G/\Gamma}F=0$, 
for every $w(T)$-smooth number $q =O( T^{o(1)})$ and every $0\leq b < q$, 
for every $T''\asymp T/(Wq)$, 
there exists $c\asymp_{m_G,d}1$ such that 
$$
\big|\EE_{0< \epsilon m \leq T''} R(W (qm+b) + A) F(g(|m|)\Gamma)\big| 
\ll \delta^{c} (1+\|F\|_{\mathrm{Lip}}) \frac{\rho(W,A;M)}{W^{n-1}}.
$$
\end{proposition}

\begin{proof}
To begin with we note that in the polynomial
$\bN_K \in \ZZ[X_1,\dots, X_n]$   the coefficient of $X_i^n$ is given by 
$N_{K/\QQ}(\omega_i)\neq 0$, 
for 
$1 \leq i \leq n$.

Our first step is to rewrite the given correlation as a sum over lattice
points.
A change of variables yields 
\begin{equation}\label{eq:him'}
\begin{split}
 \EE_{0< \epsilon m' \leq T''} &R(W(qm' + b)+A) F(g(|m'|)\Gamma) \\
&= \frac{1}{T''}
  \sum_{\substack{0< \epsilon m \leq B \\ m\equiv A+Wb \bmod{Wq}}} R(m)
  F\Big(
    g\Big(\frac{m-A-  Wb} {\epsilon Wq}\Big) \Gamma
   \Big),
\end{split}
\end{equation}
for some $B \asymp T$.
Let 
$$
\mathcal Y = 
 \left\{ 
\y \in (\ZZ/Wq\ZZ)^n: 
\begin{array}{l}
 \bN_K(\y) \equiv A+Wb \bmod{Wq} \\
 \y \equiv \b
 \bmod{M}
\end{array}
 \right\},
$$
so that $ \#\mathcal Y = \rho(Wq, A+Wb;M)$. 
The right hand side of \eqref{eq:him'} becomes
\begin{align}\label{eq:nilsequ-1}
  \frac{1}{T''}
  \sum_{\y \in \mathcal{Y}}
  \sum_{\substack{ \x \in \ZZ^n \\
        Wq\x + \y \in 
  B^{1/n} \mathfrak{X}(1) }}
  F\Big(g\Big(
   \frac{\bN_K(Wq\x + \y)-A - Wb}{\epsilon Wq}\Big)
   \Gamma\Big),
\end{align}
where 
$\mathfrak{X}(1)
= \{\x\in \mathfrak{X} : 0<\epsilon\nf_K(\x)\leq 1\}$.

Since the coefficient of $X_n^n$ in $\bN_K(X_1, \dots, X_n)$ is non-zero,
we obtain an integral polynomial of degree $n$ and leading coefficient
$\epsilon N_{K/\QQ}(\omega_n) (Wq)^{n-1}$ when fixing all but the $n$th
variable in
$$
\frac{\bN_K(Wq\x + \y)-A-  Wb}{\epsilon Wq}.
$$ 
Let $\pi:\RR^n \to \RR^{n-1}$ denote the projection onto the
coordinate plane $\{x_n=0\}$, and let 
$P_{\pi(\x),\y}(x)=\gamma_0+\dots +\gamma_n x^n$ denote the above
polynomial, 
for suitable coefficients  $\gamma_0,\dots,\gamma_n\in \ZZ$, 
with $\gamma_n=\epsilon N_{K/\QQ}(\omega_n) (Wq)^{n-1}$.
If
$Wq\x + \y \in B^{1/n} \mathfrak{X}(1)$,
then it follows that $\gamma_i \ll B^{(n-i)/n} (Wq)^{i-1}$, for
$0\leq i\leq n$.
Thus the hypotheses of \cite[Prop.~15.4]{lm1} are met.
We aim to employ this to bound \eqref{eq:nilsequ-1} by splitting the
range of the $\x$-summation into lines on which $\pi(\x)$ is constant.
With this in mind, we proceed to investigate how such lines intersect the
domain $\mathfrak{X}(1)$.

We have $\mathfrak{X}(1) \subset (-\alpha,\alpha)^n$ for some constant
$0<\alpha =O(1)$.
Let $\a=(a_1,\dots,a_{n-1})$, with  $|\mathbf{a}| < \alpha$,
and consider the line 
$\ell_{\a}: (-\alpha,\alpha) \to \RR^n$
given by $\ell_{\a}(x) = (\a,x)$.
For $\eps \geq 0$, let $\partial_{\eps} \mathfrak{X}(1) \subset \RR^n$
denote the set of points at distance at most $\eps$ to the boundary of
the closure of $\mathfrak{X}(1)$.
We note that the set
$$\left\{x \in (-\alpha,\alpha): \ell_{\a}(x) \in 
\mathfrak{X}(1) \setminus \partial_{0}\mathfrak{X}(1) \right\}$$
is the union of disjoint open intervals.
By removing all intervals of length at most $\eps$, we obtain a collection
of at most $2 \alpha \eps^{-1} \ll \eps^{-1}$ open intervals
$I_1(\a), \dots, I_{k(\a)}(\a) \in (-\alpha, \alpha)$ such
that any $x\in (-\alpha, \alpha)$ satisfies the implication
\begin{align*}
\ell_{\a}(x) \in 
\mathfrak{X}(1) \setminus \partial_{\eps}\mathfrak{X}(1)
\implies 
x \in I_j(\a) \text{ for some }j \in \{1,\dots,k(\a)\} .
\end{align*}
We will choose a suitable value of $\eps$ at the end of the proof.

Observe that any interval $(z_0,z_1) \subset (-\alpha,\alpha)$ can be
expressed as a difference of intervals in $(-\alpha,\alpha)$ that have
length at least $2\alpha/3$.
Indeed, $z_0$ and $z_1$ partition $(-\alpha,\alpha)$ into three (possibly
empty) intervals, at least one of which has length at least
$2\alpha/3$.
Thus, one of the three representations
$$
(z_0,z_1)
=(-\alpha, z_1) \setminus (-\alpha, z_0] 
=(z_0, \alpha)\setminus[z_1, \alpha)
$$
has the required property. 
For each $\a$ and $j \in \{1,\dots,k(\a)\}$, we let 
$I_j(\a) = J_j^{(1)}(\a)\setminus J_j^{(2)}(\a)$ be such a
decomposition, where 
$J_j^{(2)}(\a)$
is possibly empty.

Abbreviating 
$\a'=B^{-1/n}(Wq \a + \pi(\y))$, 
we see that \eqref{eq:nilsequ-1} equals
\begin{equation}
\begin{split} \label{eq:nilsequ-1'}
\frac{1}{T''}
  \sum_{\y \in \mathcal{Y}}
  \sum_{\substack{\a \in \ZZ^{n-1}\\ |\a'|<\alpha}} 
  \sum_{j=1}^{k(\a)}
  \sum_{x \in \ZZ}
\Big\{ 
   \1_{ B^{-1/n}Wq x \in J_j^{(1)}(\a')} - 
   \1_{ B^{-1/n}Wq x \in J_j^{(2)}(\a')} 
\Big\} 
  F\Big(g\big( P_{\a,\y} (x) \big)\Gamma\Big) \\
+O\Big(  \frac{1}{T''}\sum_{\y \in \mathcal{Y}}
  \#\{\x \in \ZZ^n: B^{-1/n}(Wq \x + \y) \in 
      \partial_{\eps} \mathfrak{X}(1)\} \Big). 
\end{split}
\end{equation}
Here, the error term accounts for all points in the
$B^{1/n}\eps$-neighbourhood of the boundary of $B^{1/n}
\mathfrak{X}(1)$,
that were excluded through the choice of intervals $I_j(\a)$.
Observe that we made use of the fact that $\|F\|_{\infty} \leq 1$.
Since $\mathfrak{X}(1)$ is $(n-1)$-Lipschitz parametrisable, we have
$\vol(\partial_{\eps} \mathfrak{X}(1)) \asymp \eps$.
Together with an application of \eqref{eq:him} this shows that the error
term is bounded by
\begin{align*}
T''^{-1} \#\mathcal{Y} \frac{\eps B}{(Wq)^n} 
\ll \eps B T^{-1}\frac{\rho(Wq,A+Wb;M)}{(Wq)^{n-1}} 
\ll \eps \frac{\rho(W,A;M)}{W^{n-1}}.
\end{align*}

Turning towards the main term, \cite[Prop.~15.4]{lm1} implies that for
every polynomial $P_{\a,\y}$ there is a $w(T)$-smooth integer
$\tilde q\in \NN$, with $\tilde q \ll T^{o(1)}$, and a constant
$c\asymp_{m_G,d}1$ such that for each $0\leq \tilde b < \tilde q$ the
sequences
$$(g(P_{\a,\y}(\tilde q x + \tilde b))
  \Gamma)_{x \leq (T'/\gamma_n \tilde q^{n})^{1/n}}$$
are totally $\delta^{c}$-equidistributed in $G/\Gamma$, provided that $T$
 is large enough.
Recall that the leading coefficient of $P_{\a,\y}$ satisfies 
$\gamma_n \asymp (Wq)^{n-1}$.
Since the set
$$
 \left\{ x \in \ZZ: 
    \frac{Wq}{B^{1/n}} (\tilde q x + \tilde b) 
   \in J_j^{(1)}(\a') \right\}
$$
is a discrete interval of length
$$
 \# \left\{ x \in \ZZ: 
    \frac{Wq}{B^{1/n}} (\tilde q x + \tilde b) 
   \in J_j^{(1)}(\a') \right\}
 \asymp \frac{T^{1/n}}{Wq \tilde q},
$$
we may employ the above total $\delta^{c}$-equidistribution property to
deduce that
\begin{align*}
\left|  \sum_{\substack{ x \in \ZZ \\
   B^{-1/n}Wq x  \in J_j^{(1)}(\a')
  }}\hspace{-0.3cm}
  F\Big(g\big( P_{\a,\y} (x) \big)\Gamma\Big)\right|
&\leq 
 \sum_{\tilde b=0}^{\tilde q - 1} 
 \bigg|
  \sum_{\substack{ x \in \ZZ \\
   B^{-1/n}Wq (\tilde q x + \tilde b) 
   \in J_j^{(1)}(\a')
  }}\hspace{-0.4cm}
  F\Big(g\big( P_{\a,\y} (\tilde q x + \tilde b) \big)\Gamma\Big) 
 \bigg| \\
&\ll \tilde q 
 T^{1/n}(Wq \tilde q)^{-1}
 \delta^{c} \|F\|_{\mathrm{Lip}} \\
&\ll T^{1/n}(Wq)^{-1}
 \delta^{c} \|F\|_{\mathrm{Lip}}.
\end{align*}
The same holds for $J_j^{(1)}(\a')$ replaced by any non-empty 
$J_j^{(2)}(\a')$. 
Hence \eqref{eq:nilsequ-1'} is bounded by
\begin{align*}
&\ll T''^{-1} \#\mathcal Y 
  \Big(\frac{T^{1/n}}{Wq}\Big)^{n-1} \eps^{-1}
  T^{1/n}(Wq)^{-1}
 \delta^{c} \|F\|_{\mathrm{Lip}}
+\ve \frac{\rho(W,A;M)}{W^{n-1}} \\ 
&\ll T''^{-1} \#\mathcal Y \frac{T}{(Wq)^{n}} \eps^{-1}
 \delta^{c} \|F\|_{\mathrm{Lip}}
+\ve \frac{\rho(W,A;M)}{W^{n-1}}\\ 
&\ll \frac{\rho(W,A;M)}{W^{n-1}} 
 \Big(
  \eps^{-1} \delta^{c} \|F\|_{\mathrm{Lip}}
+\eps \Big),
\end{align*}
where we applied \eqref{eq:him}.
Choosing $\eps = \delta^{c/2}$ completes the proof.
\end{proof}

\bigskip
\section{Majorants for positive multiplicative functions}
\label{s:gen-maj}
The aim of this section is to construct for every multiplicative function
$f: \NN \to \RR_{>0}$ whose growth is controlled in some precise
sense, for every sufficiently small $\gamma \in (0,1)$ and for any
increasing infinite sequence $\mathcal{T} = \{T_1 < T_2 < \dots\}$ of
sufficiently large positive integers, a family of majorant functions
$$
\Big(\nu^{(T)} : \{1,\dots, T\} \to \RR_{> 0}\Big)_{T \in \mathcal{T}}
$$
with the following properties:
\begin{itemize}
 \item[(i)] $f(m) \leq C \nu^{(T)}(m)$ for all $m \leq T$ and some
absolute constant $C>0$;
 \item[(ii)] $\EE_{m \leq T} f(m) \asymp \EE_{m\leq T} \nu^{(T)}(m)$; and
 \item[(iii)] $\nu^{(T)}$ has the structure of a truncated divisor sum.
That is to say, it takes the form 
$$ \nu^{(T)}(m) 
= \sum_{d \leq T^{\gamma}}
  \lambda_d \1_{d|m},
  $$
for suitable coefficients $\lambda_d\in \RR$, 
for all $m \leq T$ that lie outside a sparse exceptional set.
\end{itemize}
In \cite{lm0} such majorant functions were constructed for
the divisor function,  building on work of Erd\H{o}s \cite{erdos}.
Shiu \cite{shiu} observed that Erd\H{o}s' methods carry over to all
multiplicative functions $f:\NN \to \RR_{\geq 0}$ that satisfy the two
conditions:
\begin{itemize}
 \item[(a)] $f(p^k) \leq H^k$ for all prime powers; and
 \item[(b)] $f(m) \ll_{\delta}m^{\delta}$ as $m \to \infty$ for any
$\delta>0$.
\end{itemize}
Equally, the majorant construction from \cite[\S4]{lm0} has
an analogue for a more general class of multiplicative functions,
which we shall describe below.
The results in this section do not require condition (b).
We employ this condition however in Section \ref{s:linear-forms} when
checking the correlation condition.
In order to ensure that condition (ii) from above applies to the type of
majorant we construct, we impose the further condition that 
$g=\mu * f$ is
non-negative 
\begin{definition}\label{def:M_reg}
Let $\cM(H)$ denote the set of multiplicative
functions $f:\NN\rightarrow \RR_{\geq 0}$ such that:
\begin{itemize}
 \item[(a)] $f(p^k) \leq H^k$ for all primes $p$ and $k\in \NN$;
 \item[(b)] $f(m) \ll_{\delta} m^{\delta}$
 for all $m \in \NN$ and any $\delta>0$; and 
 \item[(c)] $f(p^k)\geq f(p^{k-1})$ for all primes $p$ and $k\in \NN$.
\end{itemize}
Let $\cM'(H)$ denote the set of non-negative multiplicative functions $f$
satisfying (a) and (b).
\end{definition}
Property (c) ensures that any $f\in \cM(H)$ always takes positive values.
Moreover, given $f \in \cM(H)$, we note that 
$g=\mu *f$ satisfies $0\leq g(p^k) \leq H^k$.

\begin{rem}
Examples of functions which belong to $\cM(H)$, for suitable $H$, include
the generalised divisor functions $\tau_k$, which appear as Dirichlet
coefficients in $\zeta^k(s)$, and functions of the form $h^{\omega(m)}$,
for any real number $h > 1$.
\end{rem}

For technical reasons we replace all cut-offs, as in (iii) above, by
smooth cut-offs.
For this purpose, let $\chi: \RR \to [0,1]$ be a smooth function that
is supported on $[-1,1]$, monoton on both $[-1,0]$ and $[0,1]$, and 
satisfies $\chi(x)=1$ for $x \in [-1/2,1/2]$.
\begin{definition}[Truncated multiplicative function]
\label{def:truncation}
 Given a cut-off parameter $T$ and any multiplicative function 
$f:\NN \to \RR_{\geq 0}$, let 
$f_{\gamma}^{(T)}:\{1,\dots,T \} \to \RR_{\geq 0}$ be defined by
$$
f_{\gamma}^{(T)}(m) 
= \sum_{d \in \NN} \1_{d|m}g(d)
  \chi\Big(\frac{\log d}{\log T^{\gamma}}\Big),
$$
where $g = \mu * f$. 
\end{definition}

Since $g$ is non-negative if $f\in \mathcal{M}(H)$, we have
$f_\gamma^{(T)}(m)\leq \sum_{d\in \NN} \1_{d\mid m}g(d)=f(m)$ for 
$m \leq T$.
The fact that $\chi(x)=1$ for $x \in [-1/2,1/2]$ implies the lower bound
\begin{align}\label{eq:truncation-lowerbound}
f_{\gamma}^{(T)}(m) 
\geq \sum_{d\leq T^{\gamma/2}} \1_{d\mid m} g(d),
\end{align}
which is an equality for $m\leq T^{\gamma/2}$.
The following lemma generalises a result of Erd\H{o}s, in the form of 
\cite[Lemma~4.1]{lm0}. 

\begin{lemma}\label{lem:erdos}
Let $f \in \cM(H)$ for $H>1$, let $C_1 > 1$ be a fixed constant
and let $\xi < 1/2$.
Furthermore, let $T>1$ be an arbitrary integer, let $m \leq T$ and suppose
that 
$$
f(m) \geq H^{\kappa} f_{2\xi}^{(T)}(m),
$$
for some $\kappa > 2/\xi$. 
Then one of the following three alternatives holds:
\begin{enumerate}
\item $m$ is excessively ``rough'' in the sense that it is divisible by 
some prime power $p^a$, $a \geq 2$, with $p^a > (\log T)^{C_1};$
\item $m$ is excessively ``smooth'' in the sense that
\[ 
\prod_{p \leq T^{1/(\log \log T)^3}} p^{v_p(m)} 
\geq T^{\xi/\log\log T};
\]
\item $m$ has a ``cluster'' of prime factors in the sense that there is a
$\lambda$ in the interval 
$$\log_2 \kappa - 2 \leq \lambda 
< \log_2 ((\log \log T)^3) < 6 \log \log \log T$$ 
such that $m$ has at least $\xi \kappa(\lambda + 3 - \log_2 \kappa)/100$ distinct
prime factors in the superdyadic range 
$I_{\lambda} = [T^{1/2^{\lambda+1}}, T^{1/2^{\lambda}}]$
and is not divisible by the square of any prime in this range.
\end{enumerate}
\end{lemma}
\begin{proof}
In view of \eqref{eq:truncation-lowerbound}, the proof of 
\cite[Lemma 4.1]{lm0} applies when replacing $\tau_{\gamma}$ by 
$f_{2\xi}^{(T)}$ in such a way that $\xi$ takes the role of $\gamma$.
Two obvious changes are necessary:
$$ f(m)
\leq H^{k-j} f(m') \leq H^{2/\xi}f(m') 
\leq H^{2/\xi} f^{(T)}_{2\xi}(m) 
< H^{\kappa} f^{(T)}_{2\xi}(m), $$
and
$$ f(m)
\leq H^{a_{j+1} + \dots + a_k}f(m') 
\leq H^{(r+1)/\xi}f^{(T)}_{2\xi}(m) 
\leq H^{2r/\xi}f^{(T)}_{2\xi}(m).
$$
The rest of the argument is identical.
\end{proof}

The previous lemma allows us to extract majorant functions of truncated
divisor sum type, at least outside the following exceptional set.  

\begin{definition}[Exceptional set]\label{def:ex}
For fixed $\gamma > 0$, let $\mathcal{S}=\mathcal{S}_{C_1,T}$
denote the set of positive integers satisfying condition (1) or (2) of
Lemma \ref{lem:erdos} with $\xi=\gamma/2$.
(This exceptional set is identical to the exceptional set in the divisor
function case \cite{lm0} for $\gamma/2$ instead of $\gamma$.)
\end{definition}

Lemma \ref{lem:erdos} will provide us with a majorant function of the
correct average order for functions $f \in \cM(H)$.
However, in view of the $W$-trick from the previous section, we require
majorant functions for each of the functions $m \mapsto f(Wm + A)$, where
as in \eqref{def:W},
$$
W= \prod_{p<w(T)}p^{\alpha(p)},
$$
with $w(T)=\log\log T$ and $\alpha(p)=\lceil (C_1+1) \log_p \log T\rceil$, and where
$0<A<W$ is such that $v_p(A)<v_p(W)$ for $p<w(T)$. 
Since 
$f(Wm+A) = f(A') f((Wm+A)/A')$, 
where $A'=\gcd(A,W)$,
it suffices to study a function that ignores the contribution from
small prime factors. 
Define the function $h_T: m \mapsto f(m/m')$, where
$$
m'= \prod_{p<w(T)} p^{v_p(m)},
$$ 
for any integer $m$. Then $h_T$ satisfies (a)--(c) in Definition
\ref{def:M_reg} whenever $f$ does.

\begin{proposition}[Majorant]
\label{p:majorant}
Let $\mathcal T = \{T_1 < T_2 < \dots \} \subset \ZZ_{>1}$ be an infinite
set of positive integers, and let $f$ be a non-negative multiplicative
function.
Assume that for all $T \in \mathcal T$ the function 
$h_T:m \mapsto f(m/m')$, where $m'= \prod_{p<w(T)} p^{v_p(m)}$, belongs to
$\cM(H)$.

Fix $\gamma > 0$ of the form $\gamma = 2^{-z}$ for some $z \in \NN$.
For any positive integers $\lambda$, $\kappa$ and $T>1$ let
$$\omega(\lambda,\kappa)
= \left\lceil \frac{\gamma \kappa(\lambda + 3 - \log_2 \kappa)}{200}
\right\rceil
$$
and $I_\lambda=[T^{1/2^{\lambda+1}}, T^{1/2^{\lambda}}]$, and define the
sets 
$$
U(\lambda,\kappa) =
\begin{cases}
\{1\},     &\text{if }\kappa=4/\gamma \text{ and } \lambda =   \log_2\kappa -2,\\
\emptyset, &\text{if }\kappa=4/\gamma \text{ and } \lambda\neq \log_2\kappa -2,\\
\bigg\{ p_1 \dots p_{\omega(\lambda,\kappa)} :
  \begin{array}{l}
   p_i \in I_\lambda \mbox{ distinct primes} \\
   f(p_i) \not= 1
  \end{array}
\bigg\}, & \mbox{if }\kappa > 4/\gamma.
\end{cases}
$$
Let the family of functions
$$
\left(\nu_{f}^{(T)} : \{1,\dots, T\} \rightarrow \RR_{\geq 0}
\right)_{T \in \mathcal{T}}
$$
be defined via
\[ 
 \nu_f^{(T)}(m) =
 \sum_{\kappa = 4/\gamma}^{[(\log \log T)^3]}
 \sum_{\lambda = \lceil \log_2 \kappa - 2\rceil}^{[\log_2((\log \log T)^3)]}
 \sum_{u \in U(\lambda ,\kappa)}
 H^\kappa \1_{u|m} f(u) 
 {h}_{\gamma}^{(T)}\left(\frac{m}{\prod_{p|u} p^{v_p(m)}}\right) 
+
 \1_{m \in \mathcal{S}} h_T(m),
\] 
where ${h}_{\gamma}^{(T)}$ is associated to $h_T$ via 
Definition \ref{def:truncation} and where $\mathcal{S}=\mathcal{S}_{C_1,T}$ is the
exceptional set from Definition \ref{def:ex}.

Then, for all sufficiently large $T>1$ and for all $m \leq T$, we have the 
majorisation property 
$$f(m) \leq H f(m')\nu_f^{(T)}(m) \ll f(m')\nu_f^{(T)}(m).$$
Furthermore, for any $0<A<W$ such that $v_p(A)<v_p(W)$ for $p<w(T)$, we
have 
\begin{equation} \label{eq:average-nu_f}
\EE_{m \leq (T-A)/W} \nu_f^{(T)} (W m + A)
\ll \frac{\EE_{m \leq (T-A)/W} f(W m + A)}{f(\gcd(A,W))},  
\end{equation}
as $T \to \infty$ through $\mathcal{T}$.
The implied constant may depend on $\gamma$ and $H$, but not on $T$.
\end{proposition}

\begin{rem}\label{back}
Apart from the term $\1_{m \in \mathcal{S}} h_T(m)$, 
the majorant $\nu_f^{(T)}$ has a truncated divisor sum
structure, since each $u \in U(\lambda,\kappa)$ satisfies 
$u \leq  T^{\gamma}$,
by remark~(3) after \cite[Prop.\,4.2]{lm0}.
\end{rem}

\begin{rem}\label{rem:coprime}
In view of Definition \ref{def:truncation} it is clear that
$h_\gamma^{(T)}$ is a divisor sum.
It is not difficult to deduce some information on the set of
positive integers $d$ that \emph{cannot} occur in this sum.
This is the set of integers $d$ such that 
$g(d) = \mu *h_T (d) = 0$.
The definition of $h_T$ implies that $g(d)= 0$
whenever $d$ has a prime factor that is smaller than $w(T)$.
Similarly, $g(d)=0$ if $d$ has a prime factor $p>w(T)$, such that 
$h_T(p^{v_p(d)})=1$.
The latter condition certainly holds when $f(p^k)=1$ for all $k \in \NN$.
Thus the truncated divisor sum ${h}_{\gamma}^{(T)}$ only runs through
divisors that are free from prime factors of both these types.
Moreover, we remark that 
the definition of $U(\lambda, \kappa)$ shows that the sum over $u$ in $\nu_f^{(T)}$ only
contains divisors that are free from primes $p$ with $f(p)=1$.
\end{rem}

Our proof of Proposition \ref{p:majorant} does not actually require
property (b) of Definition \ref{def:M_reg}.
This property will be used when establishing the pseudorandomness of
our majorant function in Section \ref{s:linear-forms}.

\begin{proof}[Proof of Proposition \ref{p:majorant}]
We define for each ``level'' $\kappa \geq 0$ an exceptional set
$$S(\kappa) = \{m \leq T: h_T(m) \geq H^{\kappa} h_{\gamma}^{(T)}(m)\}.$$
For any $m \leq T$ and any $\kappa_0 \geq 0$, we then either have 
$m \not\in S(\kappa_0)$, or else there is some integer $\kappa \geq \kappa_0$ 
such that $m \in S(\kappa) \setminus S(\kappa + 1)$.
For $\kappa_0=4/\gamma + 1$, this yields
\begin{equation}\label{eq:h-upper}
h_T(m) 
\leq
H^{4/\gamma + 1} h_{\gamma}^{(T)}(m)
+ \sum_{\kappa > 4/\gamma} 
H^{\kappa + 1}
\1_{S(\kappa)}(m)
h_{\gamma}^{(T)}(m).
\end{equation}
We claim that Lemma \ref{lem:erdos}, applied with $\xi=\gamma/2$, provides
an upper bound of the form 
$$\1_{S(\kappa)}(m) 
\leq \sum_{\lambda = \lceil\log_2 \kappa - 2\rceil}^{[\log_2((\log \log T)^3)]}
 \sum_{u \in U(\lambda,\kappa)}
 \1_{u|m},
$$
valid for every $ m \not\in \mathcal{S}$.

Taking this claim on trust for the moment, let us first deduce that 
$h_T(m) \leq H \nu_f^{(T)}(m)$.
Since $H>1$, this bound is only non-trivial if $m \not\in \mathcal{S}$.
Note that if $m \in S(\kappa)$ for some $\kappa > (\log \log T)^3$, then the 
third alternative from Lemma \ref{lem:erdos} is empty and we must have 
$m \in \mathcal{S}$.
Thus, if $m \not\in \mathcal{S}$, we can truncate the summation in 
\eqref{eq:h-upper} at $\kappa = [(\log \log T)^3]$. 
Further, if $m \not\in \mathcal{S}$ and $u|m$ for some $u \in U(\lambda,\kappa)$, 
$u>1$, then $v_p(m)=1$ for any $p|u$. 
Thus, for any  $m \not\in \mathcal{S}$, the properties of $\chi$ and 
Definition \ref{def:truncation} imply 
$$
\1_{u|m}h_{\gamma}^{(T)}(m)
\leq \1_{u|m} h_{T}\bigg(\prod_{p|u}p^{v_p(m)}\bigg)
h_{\gamma}^{(T)} \bigg(\frac{m}{\prod_{p|u}p^{v_p(m)}} \bigg)
= \1_{u|m} h_{T}(u)
h_{\gamma}^{(T)} \bigg(\frac{m}{\prod_{p|u}p^{v_p(m)}} \bigg).
$$
Inserting the claimed bound on $\1_{S(\kappa)}(m)$ in all remaining terms of 
the sum in \eqref{eq:h-upper}, making use of the inequality above, and comparing 
with the definition of $\nu_f^{(T)}(m)$, 
we indeed obtain that $h_T(m) \leq H \nu_f^{(T)}(m)$, provided the summations in 
$\kappa$ and $\lambda$ in $\nu_f^{(T)}$ contain each at least one term;
i.e.\ provided $T$ is sufficiently large.

To prove the claim it suffices to check that Lemma \ref{lem:erdos}
guarantees for $m \not\in \mathcal{S}$ (that is, for $m$ which do
not have property (1) or (2)) that there actually is a cluster of
prime divisors all satisfying $f(p)\not=1$.
Thus, suppose $m \not\in \mathcal{S}$.
Every prime $p$ that can appear in a cluster satisfies 
$p>T^{(\log \log T)^{-3}/2}$, which is larger than both $w(T)$ and 
$(\log T)^{C_1}$, provided $T$ is sufficiently large.
Hence $f(p)=h_T(p)$ for such primes and, furthermore, 
$p^2\nmid m$ since $m \not\in \mathcal{S}_{C_1,T}$.
Let 
$$
q=\prod_{\substack{p>T^{(\log \log T)^{-3}/2}
\\
f(p)=1, ~p|m}} p.
$$
Then $g(d)=0$ for all $d|m$ with $\gcd(d,q)>1$, by Remark \ref{rem:coprime}.
This, in turn, implies that $h_{\gamma}^{(T)}(m)=h_{\gamma}^{(T)}(m/q)$. 
Since $h_T(q)=1$, we also have $h_T(m)=h_T(m/q)$. 
Hence $m/q \in S(\kappa)$ and we may apply Lemma \ref{lem:erdos} to $m/q$
in order to obtain a cluster of prime factors as required.

It remains to check \eqref{eq:average-nu_f}. 
We certainly have
$$
\EE_{m \leq (T-A)/W} \1_{Wm+A \in \mathcal{S}} h_T(Wm+A)
\leq \frac{\EE_{m \leq (T-A)/W} f(W m + A)}{f(\gcd(A,W))},
$$
which reduces matters to considering the triple sum from $\nu_f^{(T)}$.
Since $2^\lambda \leq (\log\log T)^3$, any prime divisor $p$ of an element
$u \in U(\lambda, \kappa)$ satisfies $p \geq T^{1/(2 (\log \log T)^3)}$, which
is larger than $w(T)$ when $T$ is large enough. 
Thus we may assume $\gcd(u,W)=1$.
Let $g=\mu*h_T$, which is a non-negative multiplicative function.
If $P=\{m\leq T : m\equiv A \bmod{W}\}$, then 
\begin{align*}
 \frac{1}{|P|}
   \sum_{m\in P} 
   \sum_{u \in U(\lambda,\kappa)} \1_{u\mid m}
    h_T(u)
    h_{\gamma}^{(T)}\left(\frac{m}{\prod_{p|u} p^{v_p(m)}} \right) 
& \leq \frac{2W}{T} 
   \sum_{\substack{m\leq T/A'\\ A'm \in P}}
   \sum_{u \in U(\lambda,\kappa)}  h_T(u)
   \hspace{-.3\baselineskip}
   \sum_{\substack{d \leq T^{\gamma}\\ \gcd(d,u)=1}} 
   g(d) \1_{du \mid m},
 \end{align*}
where $A'=\gcd(A,W)$.
Note that  $g(d)=0$ unless $\gcd(d,W)=1$, by Remark \ref{rem:coprime}.
Since $\gcd(du,W)=1$, 
the right hand side above is
\begin{align} \label{eq:majorant-proof-1}
\nonumber    
&\ll \sum_{u \in U(\lambda,\kappa)} 
  \frac{h_T(u)}{u}
  \sum_{\substack{d \leq T^{\gamma}\\ \gcd(d,uW)=1}} 
  \frac{g(d)}{d}\\
&\ll \sum_{u \in U(\lambda,\kappa)}
  \frac{H^{\omega(u)}}{u}
  \sum_{\substack{d \leq T^{\gamma}\\ \gcd(d,W)=1}} 
  \frac{g(d)}{d}.
\end{align}
Note that 
\begin{align*}
   \sum_{u \in U(\lambda,\kappa)} 
   \frac{H^{\omega(u)}}{u}
\ll\frac{H^{\omega(\lambda,\kappa)}}{\omega(\lambda,
\kappa)!} 
    \Bigg( \sum_{p \in I_{\lambda}} \frac{1}{p} 
    \Bigg)^{\omega(\lambda,\kappa)} 
\ll\frac{(H \log 2 + o(1))^{\omega(\lambda,\kappa)}}
        {\omega(\lambda,\kappa)!}.
\end{align*}
If $\gcd(d,W)=1$ and $d \leq T^{\gamma}$, then the number of integers
$x\leq T$ for which $x\equiv A \bmod{W}$ and $d\mid x$ has order
$T/(Wd)$. Since $h_T=1*g$, 
we deduce that
\begin{equation}\label{eq:sq}
\sum_{\substack{d \leq T^{\gamma} \\ \gcd(d,W)=1}} \frac{g(d)}{d}
\asymp \frac{W}{T}\sum_{\substack{m\leq T\\ m\equiv A \bmod{W}}} h_T(m).
\end{equation}
Hence the inner sum from \eqref{eq:majorant-proof-1} is bounded by 
$T^{-1}W\sum_{m \in P} h_T(m).$

The above estimates allow us to bound the average value of 
$\nu_f^{(T)}(m) - \1_{m\in \mathcal{S}}h_T(m)$
via
\begin{align*}
\left(\sum_{m\in P} h_T(m)\right)^{-1}
&\sum_{m\in P} \Big(\nu_f^{(T)}(m) - \1_{m\in \mathcal{S}}h_T(m)\Big)\\
&=\left(\sum_{m\in P} h_T(m)\right)^{-1}
  \sum_{m\in P} 
  \sum_{\kappa \geq 4/\gamma} 
  \sum_{\lambda \geq \log_2 \kappa - 2}
  \sum_{u \in U(\lambda,\kappa)} 
   H^\kappa \1_{u|m} h^{(T)}_{\gamma}(m) \\
&\ll
  \sum_{\kappa \geq 4/\gamma}
  \sum_{\lambda \geq \log_2 \kappa - 2} 
  \frac{ H^{\kappa} \cdot  
  (H \log 2 +o(1))^{\omega(\lambda,\kappa)}}{\omega(\lambda,\kappa)!}
 \\
&\ll 
  \sum_{\kappa \geq 4/\gamma}
  \sum_{j \geq 1} H^{\kappa}  \cdot 
  \left(\frac{200 \cdot e \cdot (H\log 2 + o(1))}{\gamma \kappa j}
  \right)^{\gamma \kappa j/200} \\
&\ll 
  \sum_{\kappa \geq 4/\gamma} 
  \frac{H^{\kappa}}{\kappa^{\kappa\gamma/200}} 
  \left( 
  \sum_{j \geq 1} 
  \left(\frac{200 \cdot e \cdot (H\log 2 + o(1))}{\gamma j}
  \right)^{\gamma j/200} \right)^{\kappa}.
\end{align*}
This converges absolutely, and hence completes the proof.
\end{proof}

Our final objective in this section is to show that the exceptional set
$\mathcal{S}_{C_1,T}$ is negligible when evaluating correlations such as
the counting function $N(T)$ given by \eqref{eq:def-count},
provided $C_1$ is sufficiently large.

Let $\Psi=(\psi_1, \dots, \psi_r): \ZZ^s \to \ZZ^r$ be a
system of non-constant linear polynomials whose non-constant  parts
are  pairwise non-proportional and have coefficients 
bounded by $L$ in absolute value. 
Let $\mathfrak{K}\subset [-1,1]^{s}$ be such that 
$\Psi(T\mathfrak{K}) \subset (0,T]^r$ 
and assume that
$$
\EE_{\m \in \ZZ^s \cap T\mathfrak{K}} 
\1_{d\mid \psi_i(\m)} \ll_{L} \frac{1}{d},
$$ 
for any $d \in \NN$ and each $1\leq i\leq r$.
Since no $\psi_i$ is constant, the latter condition is guaranteed to
hold when $\mathfrak{K}$ is convex or when $\mathfrak{K}$ has an
$(s-1)$-Lipschitz parametrisable boundary.
Then it follows from \cite{erdos} (cf. \cite[Lemmas 3.2 and 3.3]{lm0})
that the exceptional set satisfies
\begin{equation}\label{eq:exceptional}
\EE_{\m \in \ZZ^s \cap T\mathfrak{K}} 
 \1_{\psi_i(\m)\in \mathcal{S}_{C_1,T}}
\ll_{L,C_1} (\log T)^{-{C_1}/2},
\end{equation}
for each $1\leq i\leq r$.
We shall combine this estimate with the following bound on the $k$th
moment of a non-negative multiplicative function $f$.

\begin{lemma}[$k$th moment bound for $f$]
\label{lem:kth-moment}
Suppose $f$ satisfies Definition \ref{def:M_reg}(a) and (b). 
Let $k$ be a positive integer and let $\Psi$ and $\mathfrak{K}$ be as
above.
Then
$$\EE_{\m \in \ZZ^s \cap T\mathfrak{K}} \prod_{i=1}^r f^k(\psi_i(\m)) 
\ll_{L,r,k} (\log T)^{O_{r,k,H}(1)}.$$
\end{lemma}

\begin{proof}
Let $g = \mu * f$. Then H{\"o}lder's inequality implies
\begin{align*}
  \EE_{\m \in \ZZ^s \cap T\mathfrak{K}} \prod_{i=1}^r f^k(\psi_i(\m)) 
 &\leq \prod_{i=1}^r 
   \bigg(\EE_{\m \in \ZZ^s \cap T\mathfrak{K}} 
    f^{kr}(\psi_i(\m))\bigg)^{1/r}\\
 &\leq \prod_{i=1}^r \left(
    \sum_{\substack{d_1,\dots, d_{rk} \leq T }}
    \EE_{\m \in \ZZ^s \cap T\mathfrak{K}} 
    \prod_{j = 1}^{rk}
    \1_{d_j|\psi_i(\m)} 
    g(d_j)
    \right)^{1/r}\\
 &\ll_{L,\ve} \sum_{\substack{d_1,\dots, d_{rk} \leq T }}
    \frac{\min\{(d_1\dots d_{rk})^\ve,H^{\Omega(d_1 \dots d_{rk})}\}}
         {\lcm(d_1, \dots, d_{rk})}\\
&\ll_{L,\ve}         \prod_{p\leq T}\left(1+\sum_{\substack{\delta_1,\dots,\delta_{rk}\geq 0\\ \delta_1+\dots+\delta_{rk}\geq 1}} \frac{\min\{ p^{\ve(\delta_1+\dots+\delta_{rk})},
H^{\delta_1+\dots+\delta_{rk}}\}}{p^{\max\{\delta_1,\dots, \delta_{rk}\}}}\right),
\end{align*}
for any $\ve>0$. 
For given $\ell\in \NN$ there are at most $(1+\ell)^{rk}$ choices of $\delta_1,\dots,\delta_{rk}$ for which $\max\{\delta_1,\dots,\delta_{rk}\}=\ell$.
Taking $\ve=\frac{1}{2rk}$, it follows that the right hand side is
\begin{align*}
&\leq 
\prod_{p\leq T}\left(1+
\sum_{\ell\geq 1} \frac{(1+\ell)^{rk}\min\{ p^{\ell/2},  H^{\ell rk}\}}{p^{\ell}}
\right)\\
&\leq 
\prod_{p\leq H^{2rk}}\left(1+
\sum_{\ell\geq 1} \frac{(1+\ell)^{rk}}{p^{\ell/2}}
\right)
\prod_{H^{2rk}<p\leq T}\left(1+
\sum_{\ell\geq 1} \frac{(1+\ell)^{rk}H^{\ell rk}}{p^{\ell}}
\right)\\
&\ll_{r,k,H} (\log T)^{O_{r,k,H}(1)},
\end{align*}
as required. 
\end{proof}

\begin{proposition}[Reduction to unexceptional residues]
\label{p:unexceptional}
Let 
$\mathcal{S}_{C_1,T}$ be the exceptional set from  Definition~\ref{def:ex}. 
Suppose that $f_1, \dots, f_r : \ZZ_{\geq 0} \to \RR$ are functions that are all
bounded pointwise in modulus by some function $f \in \cM'(H)$.
Let $i\in \{1,\dots, r\}$. 
Suppose that 
$f'_i: \{1,\dots, T\} \to \RR$ denotes a function which agrees with $f_i$
on $\{1,\dots, T\} \setminus \mathcal{S}_{C_1,T}$ and satisfies
$|f'_i(m)| \leq f(m)$ for all $m \in \mathcal{S}_{C_1,T}$.
If the parameter $C_1$ of the exceptional set is sufficiently large
depending on $r$ and $H$, and
if $\Psi$ and $\mathfrak{K}$ are as above, then
$$
\mathcal E=
\Big| 
\sum_{\m \in \ZZ^s \cap T\mathfrak{K}}
\prod_{i=1}^r
f_i(\psi_i(\m))
- 
\sum_{\m \in \ZZ^s \cap T\mathfrak{K}}
\prod_{i=1}^r
f'_i(\psi_i(\m)) \Big|
= O(T^s (\log T)^{-C_1/4}).
$$
\end{proposition}
\begin{proof}
Since
\begin{align*}
\mathcal E
 \leq 2 
 \sum_{\m \in \ZZ^s \cap T\mathfrak{K}}
 \sum_{i=1}^r \1_{\psi_i(\m) \in \mathcal{S}_{C_1,T}}
 \prod_{j=1}^r
 f(\psi_j(\m)),
\end{align*}
the proposition follows by the Cauchy--Schwarz inequality from
\eqref{eq:exceptional}
and Lemma \ref{lem:kth-moment}. 
\end{proof}

\bigskip
\section{Construction of the majorant}\label{s:majorant}

The previous section described the construction of majorants for a general class
of positive  multiplicative functions.
Returning to the proof of Theorem \ref{t:NB}, we shall now consider the
representation functions $R_i$ from Definition \ref{def:repr-fn}.
Building on the results from Section~\ref{s:gen-maj}, we construct for
each of these representation functions a family of majorant functions 
$$(\nu_{R_i}^{(T)}:\{1,\dots,T\} \to \RR_{\geq 0})_{T \in \mathcal{T}}$$ 
with the properties (i)--(iii) described at the start of 
Section \ref{s:gen-maj}.
(The cut-off parameter $T \in \mathcal{T}$ will later correspond to the
parameter $T$ that appears in Theorem \ref{t:NB}.)

We begin with an easy estimate for $R_i$ that relates it to the 
multiplicative function $r_{K_i}(|m|)$ whose values are given by
the coefficients of the Dedekind zeta function \eqref{eq:Dedekind} for the
number field $K_i/\QQ$ of degree $n_i$.

\begin{lemma}\label{lem:7.1}
We have
$
R_i(m) \leq 2|\mu_{K_i}|r_{K_i}(|m|),
$
for non-zero $m \in\ZZ$.
\end{lemma}
\begin{proof}
Replacing $\mathfrak{X}_i\subset \mathfrak{D}_{i,+}$ by $\mathfrak{D}_{i,+}$ and
dropping the congruence condition from Definition~\ref{def:repr-fn},
we obtain the upper bound
\begin{align*}
R_i(m) 
&\leq
\#\left\{
   \x\in \ZZ^{n_i}\cap \mathfrak{D}_{i,+}: \nf_{K_i}(\x)=m 
  \right\}.
\end{align*}
The discussion of the unit groups $U_{K_i}$ and $U_{K_i}^{(+)}$ in 
Section \ref{s:units} showed that the index of $Y_{K_i}^{(+)}$ in $Y_{K_i}$ is at
most two, which implies that
\begin{align*}
R_i(m) 
&\leq  2 \#\left\{\x\in \ZZ^{n_i}\cap \mathfrak{D}_i: 
 \nf_{K_i}(\x)=m 
\right\}\\
&\leq  2 |\mu_{K_i}| \#\left\{\x\in \ZZ^{n_i}/U_{K_i}:
|\nf_{K_i}(\x)|=|m| 
\right\},
\end{align*}
where $\mathfrak{D}_i$ is a fundamental domain (in the coordinate space)
for the free part of $U_{K_i}$.
The cardinality in the final line equals 
$\#\{(\alpha) \subset \fo: \n (\alpha)= |m|\}$, 
which itself is bounded by $r_{K_i}(|m|)$, as required.
\end{proof}

Let $\mathcal{S}_{C_1,T}$ be the  set from Definition
\ref{def:ex} and recall the definition \eqref{eq:def-count} of
$N(T)$.

\begin{proposition} \label{p:unexceptional'}
For each $1 \leq i \leq r$, let $R'_i: \{-T,\dots,T\} \to \RR_{\geq 0}$ 
denote a function 
such that $R_i'(m)=R_i(m)$ for all 
$m$ satisfying $|m|\in\{1,\dots, T\} \setminus \mathcal{S}_{C_1,T}$,
and which further satisfies
$0 \leq R'_i(m) \leq R_i(m)$ when  $|m| \in \mathcal{S}_{C_1,T}$.
If the parameter $C_1$ of the exceptional set is sufficiently large, then
$$ 
N(T)
=\beta_\infty \prod_p \beta_p \cdot T^s +o(T^{s})
$$
if and only if
$$ 
\sum_{\substack{
\u \in \ZZ^s\cap T\mathfrak{K}\\
\u\= \a\bmod{M}}}
\prod_{i=1}^r
R_i'(f_i(\u))
=\beta_\infty \prod_p \beta_p \cdot T^s +o(T^{s}).
$$
\end{proposition}
\begin{proof}
We have 
$R_i(m) \leq 
2 |\mu_{K_i}| \tau(|m|)^{n_i}$ by Lemma \ref{lem:7.1} and
\eqref{eq:r-upper}, 
where $\tau^{n_i} \in \cM(2^{n_i})$.
For $\beps=(\epsilon_1,\dots,\epsilon_r)\in \{\pm \}^r$ we 
let 
$\mathfrak{K}_{\beps}
=\mathfrak{K}\cap \mathbf{f}^{-1}(\RR_{\epsilon_1}\times \dots \times
\RR_{\epsilon_r})$.
Then we may decompose $\mathfrak{K}$ as the union of the $2^{r}$ sets
$\mathfrak{K}_{\beps}$, 
together with one set  $\mathfrak{K}_{0}$ such that at each
point $\u \in \mathfrak{K}_{0}$ at least one $f_i$ vanishes.  
We may discard $\mathfrak{K}_{0}$ since $\prod_{i=1}^r R_i(f_i(\u))=0$
at all of its integral points.
Proposition \ref{p:unexceptional} 
may be applied 
separately
to each of the $2^r$ remaining sums,  by reinterpreting $R_i$ and $R'_i$ as functions on
$\{1,\dots,T\}$ via $m \mapsto R_i(\epsilon_i m)$.
\end{proof}

Proposition  \ref{p:unexceptional'} allows us to work with functions 
\begin{equation}\label{eq:Ri'}
R_i'(m) =
\begin{cases}
 R_i(m), & \mbox{if $|m| \not\in \mathcal{S}_{C_1,T}$,}\\
 0,& \mbox{if $|m| \in \mathcal{S}_{C_1,T}$,}
\end{cases}
\end{equation}
instead of the original counting functions.
Thus the majorant function only needs to majorise $R_i$ outside the
exceptional set. 

\begin{rem}
For later applications it is convenient at this point to note that
\begin{equation}\label{eq:R_i-in-S}
\EE_{|m| < T/W} 
 R_{i}(Wm+A) 
\1_{Wm+A \in \mathcal{S}_{C_1,T}}
\ll (\log T)^{-C_1/4}
\end{equation}
and
\begin{equation}\label{eq:sunny}
\begin{split}
\EE_{m \leq (T-A)/W} R'_{i}(Wm+A)
&= \EE_{m \leq (T-A)/W} R_{i}(Wm+A) + O\left((\log T)^{-C_1/4} \right)\\
&\asymp \frac{\rho_i(W,A;M)}{W^{n_i-1}}.
\end{split}\end{equation}
The first bound follows from Proposition \ref{p:unexceptional}, applied
with (cf.\ Proposition \ref{p:unexceptional'})
$r=s=1$, $ f_1=R_i$ and $ \psi_1(m)=m.$ 
The second part follows by combining \eqref{eq:R_i-in-S} with 
Lemma \ref{lem:average-order}.
\end{rem}

For most of the remainder of this section we consider a typical representation function and
will drop the index $i$ in such situations. 
 Lemma \ref{lem:7.1} implies that 
$R(m)\ll r_K(|m|)$. 
Taking $q=1$ in Lemma \ref{lem:average-order}, we deduce that for 
$\epsilon \in \{\pm\}$
$$
\frac{1}{T}
\sum_{0 < \epsilon m \leq T} R(m)\sim \kappa^{\epsilon}, \quad
(T\rightarrow\infty),
$$
whereas \eqref{eq:210'} yields
$$
\frac{1}{T}\sum_{0< m \leq T} r_K(m)\sim h\kappa, \quad
(T\rightarrow\infty).
$$
Hence $R$ is 
majorised by 
 $r_K$ and has the same average order as it.

Given $T \in \NN$, consider $P_A = \{m \equiv A \bmod{W}\}$, for 
$A \in \cA$ given by \eqref{eq:def-A} and $W$ as in \eqref{def:W}.
Lemma \ref{lem:A>0} only provides us with precise
information on $\rho(p^m,A)$ when $p \nmid D_K$ and $v_p(A) < m$. 
This limits our ability to deduce that $R$ and $r_K$ have  the same
average order on progressions $P_A$ unless $\gcd(A,D_K)=1$.
For this reason we will need to refine the bound on $R$ from 
Lemma~\ref{lem:7.1} to one that is tighter at integers $m$ with 
$\gcd(m,D_K)>1$.
We set
$$
W_0 = \prod_{p\mid D_K} p^{\alpha(p)}
$$
and proceed to establish the following result.

\begin{lemma}\label{lem:7.1'}
Let $A \in \cA$ and write $A_0 = \gcd(A,W_0)$.
Then
\begin{align*}
 R(Wm + A) 
\ll 
 \frac{\rho(W_0,A)}{W_0^{n-1}} 
 r_K\left(\frac{|Wm + A|}{A_0}\right).
\end{align*}
\end{lemma}

\begin{proof}
The following shorter proof of this result was suggested to us by the referee.
To begin with, Lemma  
\ref{lem:7.1} implies that $ R(Wm + A)\ll r_K(|Wm+A|)$. Moreover, 
for any $A\in \cA$ it is clear that $A_0$ is coprime to $|Wm+A|/A_0$.
Since $r_K$ is multiplicative it is therefore enough to show that 
$$
r_K(A_0)\ll \frac{\rho(W_0,A)}{W_0^{n-1}}.
$$
Now since $A\in \cA$ we must have $v_p(A_0)<v_p(M)\ll 1$ for any $p\mid W_0$. Thus $r_K(A_0)\ll 1$ by \eqref{eq:r-upper}. Hence it remains to prove that 
$\rho(W_0,A)/W_0^{n-1}\gg 1$.
Let $p\mid D_K$. 
Setting 
$$m_0 = 1 + 2\left(v_p(M) + v_p(A) +v_p(n)\right) < 1 + 4v_p(M) + 2v_p(n) \ll 1,$$
we apply Lemma \ref{lem:C6.4} to obtain
\begin{equation}\label{eq:dove}
\frac{\rho(p^{\alpha(p)},A)}{p^{\alpha(p)(n-1)}}
\geq 
\frac{\rho(p^{\alpha(p)},A;p^{v_p(M)})}{p^{\alpha(p)(n-1)}}
=\frac{\rho(p^{m_0},A;p^{v_p(M)})}{p^{m_0(n-1)}}
\geq\frac{1}{p^{m_0(n-1)}}
\gg 1,
\end{equation}
since $\rho(W,A;M)>0$ in \eqref{eq:def-A}.
This completes the proof of the lemma. 
\end{proof}

In view of Lemma \ref{lem:7.1'} we are therefore led to construct a
majorant function for each restriction of the multiplicative function
$r_K$ to a progression $P_A$, with $A\in \cA$.
Given  $T \in \ZZ_{>1}$, we write 
\begin{equation}\label{eq:m'}
m'=\prod_{p<w(T)}p^{v_p(m)},
\end{equation}
for $m\in \ZZ$.
Then $r_K(m')$ is constant for $m \in P_A$.
Thus we seek a majorant for the function $m \mapsto r_K(m/m')$, which is
free from the contributions of small prime factors.

Note that $r_K$ is an unbounded function with sparse support.
Our next aim, accomplished in Lemma \ref{lem:r_K-decomposition} below, is
to simplify the task by separating these two properties, replacing $r_K$
by the product of a bounded function with sparse support and an unbounded
function with dense support.

In general we write
$\<\cP\> = \{n \in \NN: p\mid n \Rightarrow p \in \cP\}$
for a set $\cP$ of rational primes.
Recall the definition \eqref{eq:PPP} of $\cP_{0}, \cP_1$ and $\cP_2$.
Additionally, we require the sets
$$
\bar\cP_j = \{p \in \cP_j: p>w(T)\},
$$ 
for $j=1,2$.
It follows from \eqref{eq:r} that the restriction of $r_K$ to square-free
numbers is supported on  $\<\cP_0 \cup \cP_1\>$. 
Let $r_{\mathrm{res}}$ denote the multiplicative function defined via
\begin{equation}\label{eq:r_res}
r_{\mathrm{res}}(p^m) = 
\begin{cases}
 r_K(p^m), &  \mbox{if $p \in \cP_0 \cup \cP_1$,}\\
 1,        &  \mbox{if $p \in \cP_2$.}
\end{cases}
\end{equation}
We have the following result. 

\begin{lemma} \label{lem:r_K-decomposition}
For all $m \in \NN$, we have
\begin{equation*}
r_K(m) \leq 
 r_{\mathrm{\rm res}}\left(m\right)
 \sum_{\substack{q \in \<\cP_2\> \\ v_p(q) \not= 1 ~ \forall p}} 
 \1_{q|m}
 \tau(q)^n
 \1_{\<\cP_0 \cup \cP_1\>}\left(\frac{m}{q}\right). 
\end{equation*}
\end{lemma}

\begin{proof}
If $r_K(m)$ is positive, then $m$ has no prime divisor
$p \in \cP_2$ for which $p^2 \nmid m$.
In this case the sum on the right hand side has exactly one term,
corresponding to the factorisation of $m$ into the product of 
$q \in \<\cP_2\>$ and $m/q \in \<\cP_0\cup\cP_1\>$.
The multiplicativity of $r_K$ implies
$$
r_K(m) 
= r_K(q)r_K\left(\frac{m}{q}\right)
= r_K(q)r_{\mathrm{res}}\left(\frac{m}{q}\right)
=  r_K(q)r_{\mathrm{res}}\left(m\right)
\leq \tau(q)^n r_{\mathrm{res}}\left(m\right),
$$
where we used \eqref{eq:r-upper} to bound $r_K(q)$.
\end{proof}

As a direct consequence of this lemma, we obtain
\begin{equation}\label{eq:r-model-bound'}
r_K(m) \leq r_K(m') \cdot r_{\mathrm{\rm res}}\left(\frac{m}{m'}\right)
\sum_{\substack{q \in \<\bar\cP_2\> \\ v_p(q) \not= 1 ~ \forall p}} 
 \1_{q|m}
 \tau(q)^n
 \1_{\<\cP_0 \cup \cP_1\>}\left(\frac{m}{qm'}\right),
\end{equation}
where $m'$ is given by \eqref{eq:m'} and we have observed that $r_K(m)=r_K(m')r_K(m/m'),$ by multiplicativity.
In view of \eqref{eq:r-model-bound'} we proceed by constructing two
families of majorant functions in Sections \ref{s:rres} and
\ref{s:sieve}: 
one for the positive
multiplicative function $m \mapsto r_{\mathrm{res}}(m/m')$ and one
 for the characteristic function $\1_{\<\cP_0 \cup \cP_1\>}$.
Inserting these majorants into the bound \eqref{eq:r-model-bound'}, we
will obtain a family of majorant functions for $r_K$.
In Section \ref{s:linear-forms} we check that the resulting majorants for
$r_K$ form a family of pseudorandom majorants when restricting them to
the arithmetic progressions $P_A$.

\subsection{Majorant for $r_{\mathrm{res}}$}\label{s:rres}

Our first task is to check that Proposition \ref{p:majorant}
applies to the function~$r_{\mathrm{res}}$.

\begin{lemma}\label{lem:start}
Let $T \in \ZZ_{>1}$. Given $m \in \NN$, let $m'$ be defined by
\eqref{eq:m'}.
If $h_T$ denotes the function $m \mapsto r_{\mathrm{res}} (m/m')$,
then $h_T\in \cM(2^{n})$, provided that $T\geq \exp(\exp(2|D_K|))$.
\end{lemma}
\begin{proof}
We need to check conditions (a)--(c) in Definition \ref{def:M_reg}.
By \eqref{eq:r-upper}, we have 
$$
r_{\mathrm{res}}(p^m)
 \leq r_{K}(p^m)
 \leq (m+1)^n
 = (\tau(p^m))^n
 \leq 2^{mn}.
$$
Thus part (a) holds with $H=2^{n}$. 
Part (b) follows immediately from the respective property for the divisor
function. 
To check part (c), we may restrict attention to $p \in \cP_1 \cup \cP_2$,
since $p<w(T)$ when $p\in \cP_0$ and $T\geq \exp(\exp(2|D_K|))$.
Recalling \eqref{eq:r_res} we see that  condition (c) is trivially
satisfied for $p \in \cP_2$.
If $p \in \cP_1$, then there is a prime ideal $\fp|(p)$ of residue degree
$1$ in $K/\QQ$. 
Thus, if $\fa \subset \fo$ is counted by $r_{\mathrm{res}}(p^{m})$,
that is to say $\n\fa = p^{m}$, then $\fa\fp$ is an ideal counted by
$r_{\mathrm{res}}(p^{m+1})$.
Hence 
$r_{\mathrm{res}}(p^m) \leq r_{\mathrm{res}}(p^{m+1}),$
as required for (c).
\end{proof}

Lemma \ref{lem:start} implies that 
$r_{i,\mathrm{res}}(m/m')\in \mathcal{M}(2^{n_i})$ 
for each $1\leq i\leq r$, with $m'$ given by \eqref{eq:m'}.
Taking 
$f=r_{i,\mathrm{res}}$ and $h_T=r_{i,\mathrm{res}}(m/m')$, let
$h_\gamma^{(T)}$ be as in Definition \ref{def:truncation}, with
$g=\mu*h_T$. 
Let 
\begin{align}\label{eq:nu_i}
 \nu_i^{(T)}(m) = 
 \sum_{\kappa = 4/\gamma}^{[(\log \log T)^3]}
 \sum_{\lambda = \lceil\log_2 \kappa - 2\rceil}^{[\log_2((\log \log T)^3)]}
 \sum_{u \in U(\lambda,\kappa)}
 2^{n_i\kappa} \1_{u|m} r_{i,\mathrm{res}}(u)
 h_{\gamma}^{(T)}\left(\frac{m}{m'\prod_{p|u}p^{v_p(m)}}\right).
\end{align}
Then Proposition \ref{p:majorant} implies that $\nu_i^{(T)}(m)$
majorises $r_{i,\mathrm{res}}(m/m')$ on 
$\{1,\dots,T\} \setminus \mathcal{S}_{C_1,T}$.
According to Remark \ref{rem:coprime}, furthermore, it is a truncated
divisor sum that only involves divisors from $\<\bar\cP_1^{(i)}\>$
provided $T \geq \exp(\exp(2 |D_{K_i}|))$.
Indeed, if $p \in \cP_0^{(i)}$ then $p<w(T)$, and if $p \in
\bar\cP_2^{(i)}$ then $r_{i,\mathrm{res}}(p^k)=1$ for any $k\in \NN$. 

Proposition \ref{p:majorant} provides the upper bound
\begin{align*}
\EE_{m \leq (T-A)/W} \nu_i^{(T)}(Wm + A) \ll 
\frac{\EE_{m \leq (T-A)/W} r_{i,\mathrm{res}}(Wm + A)}
 {r_{i,\mathrm{res}}(\gcd(A,W))},
\end{align*}
for any $A$ belonging to the set $ \cA_i$ defined in \eqref{eq:def-A}.
We proceed by deducing the following bound in terms of the arithmetic data
that is involved.
\begin{lemma}\label{lem:nu_i-average}
For $A \in \mathcal{A}_i$ we have
\begin{align*}
\EE_{m \leq (T-A)/W} &\nu_i^{(T)}(Wm + A) \\
\ll& 
 (\log T)^{1-\delta_i}
 \prod_{\substack{p<w(T)\\ p \in \cP_0^{(i)} \cup \cP_1^{(i)}}} 
 \Big(1-\frac{1}{p}\Big)^{-1}
 \prod_{\substack{\fp|p\\ \fp\subset \fo_{K_i}}}
 \Big(1-\frac{1}{\n \fp}\Big),
\end{align*}
where $\cP_0^{(i)},  \cP_1^{(i)}, \cP_2^{(i)}$ are the sets \eqref{eq:PPP}
of rational primes corresponding to $K_i$ and $\delta_i$ is the Dirichlet 
density of $\cP_1^{(i)}$.
\end{lemma}
\begin{proof}
Let $g = \mu * h_T$ and $h_T = r_{i,\mathrm{res}}(m/m')$ with 
$m' = \prod_{p<w(T)}p^{v_p(m)}$, as before. 
We may assume that $T$ is large enough to guarantee that $g$ is
non-negative, with $g \in \cM'(2^{n_i})$.
Then it follows from \eqref{eq:sq} that
\begin{align*}
\frac{\EE_{m \leq (T-A)/W} r_{i,\mathrm{res}}(Wm + A)}
 {r_{i,\mathrm{res}}(\gcd(A,W))}
\ll \sum_{d \leq T} \frac{g(d)}{d}
\leq
\prod_{p\leq T}\left(1+\sum_{k\geq 1}\frac{g(p^k)}{p^k}\right)
\ll  
\exp \left(\sum_{p\leq T} \frac{g(p)}{p} \right).
\end{align*}
Since $r_{i,\mathrm{res}}(p)=r_{K_i}(p) + \1_{\cP_2^{(i)}}(p)$ and 
$g(p)=0$ for $p\leq w(T)$, the sum in the argument of the exponential
function is equal to
$$
\sum_{p \leq T} \frac{g(p)}{p}
= \sum_{p \leq T} \frac{r_{K_i}(p)-1}{p}
+ \sum_{p \leq T} \frac{\1_{\cP_2^{(i)}}(p)}{p}
+ \sum_{\substack{p < w(T) \\ p \in \cP_0^{(i)} \cup \cP_1^{(i)}}}
\left(\frac{1}{p} - \frac{r_{K_i}(p)}{p}\right).
$$
The prime ideal theorem \cite[Satz~192]{landau_alg} implies that the
first sum is 
$$
\sum_{p \leq T} \frac{r_{K_i}(p)-1}{p}
= \log\log T - \log\log T + O(1) = O(1).
$$
Corollary \ref{cor:log-asymp} shows that the second sum satisfies
$$
\exp\left(\sum_{p \leq T} \frac{\1_{\cP_2^{(i)}}(p)}{p}\right) 
\asymp (\log T)^{1-\delta_i}.
$$
For the final sum we obtain
\begin{align*}
\exp 
\Bigg(
\sum_{\substack{p < w(T) \\ p \in \cP_0^{(i)} \cup \cP_1^{(i)}}}
\left(\frac{1}{p} - \frac{r_{K_i}(p)}{p}\right)
\Bigg)
\asymp  
\prod_{\substack{p<w(T)\\ p \in \cP_0^{(i)} \cup \cP_1^{(i)}}} 
 \Big(1-\frac{1}{p}\Big)^{-1}
 \prod_{\substack{\fp|p\\ \fp\subset \fo_{K_i}}}
\Big(1-\frac{1}{\n \fp}\Big)
\end{align*}
by combining the approximation $\log(1 \pm m^{-1})= \pm m^{-1}+O(m^{-2})$,
valid for integers $m>2$, with the identity 
$r_{K_i}(p)=\#\{ \fp\mid (p): \fp\subset \fo_{K_i},~\n \fp =p \}$. 
This completes the proof.
\end{proof}

\subsection{Sieve majorant} \label{s:sieve}
In this section we drop the index $i$ and work with a typical number field
$K$ of degree $n$ over $\QQ$.
Our next objective is to construct a majorant function of the correct
average order for the characteristic function $\1_{\<\cP_0 \cup \cP_1\>}$
in any of the arithmetic progressions
$\{m \equiv A \bmod W\}$, for $A \in \cA$.

Let
$\chi: \RR \to \RR_{\geq 0}$ be a smooth even function with 
$\supp \chi \subset [-1,1]$ and $\chi(x)=1$ for 
$x \in [-1/2,1/2]$.
As before we let $\gamma>0$, to be viewed as a small fixed
constant.
In analogy to the construction from \cite[App.~D]{GT}, which itself
builds on work of Goldston and Y{\i}ld{\i}r{\i}m 
\cite{GY2,GPY}, we consider the functions
$\nu_{\mathrm{sieve}}^{(T)}, {\nu'}_{\mathrm{sieve}}^{(T)}
:\{1,\dots,T\} \to \RR_{\geq0}$, defined via
\begin{align} \label{eq:sieve}
 \nu_{\mathrm{sieve}}^{(T)}(m)= 
 \left( \sum_{\substack{d \in \<\bar\cP_2\> \\ d|m}} \mu(d) 
  \chi\Big(\frac{\log d}{\log T^{\gamma}}\Big) 
  \right)^2
\end{align}
and
\begin{align} \label{eq:sieve'}
{\nu'}_{\mathrm{sieve}}^{(T)}(m)=
\sum_{\substack{q\in \<\bar\cP_2\> \\
 v_p(q) \not= 1 ~\forall p}} \1_{q|m}
 \tau(q)^n \chi\Big(\frac{\log q}{\log T^{\gamma}}\Big)
 \nu_{\mathrm{sieve}}^{(T)}\left(\frac{m}{q}\right).
\end{align}
Both of these functions are non-negative. 
Moreover we note that for $m\in \<\mathcal{P}_0\cup\mathcal{P}_1\>$ we
have $\nu_{\mathrm{sieve}}^{(T)}(m) = 1$.
Hence $\nu_{\mathrm{sieve}}^{(T)}$ majorises $\1_{\<\cP_0 \cup \cP_1\>}$.
The main goal of this section is to establish  the following lemma.
\begin{lemma}\label{lem:sieve-upperbound}
For every $A \in \cA$ we have
\begin{align*}
\frac{W}{T}
 \sum_{m\leq (T-A)/W} {\nu'}_{\mathrm{sieve}}^{(T)} (W m+A)
\ll (\log T)^{\delta-1}
\prod_{\substack{p \in \cP_2 \\ p \leq w(T)}}
\Big(1+\frac{1}{p}\Big),
\end{align*}
where $\delta$ is the Dirichlet density of $\mathcal{P}_1$.
\end{lemma}

The key element used for both the proof of 
Lemma \ref{lem:sieve-upperbound} and for asymptotically evaluating linear
correlations of $\nu'_{\mathrm{sieve}}$ in the next section is the
observation, due to Green and Tao \cite[App.~D]{GT}, that one can turn the
smooth cut-off $\chi$ in \eqref{eq:sieve} into multiplicative functions as
follows.
Let $\vartheta$ be the transform of $\chi$ that is defined via
$$e^x \chi(x) = \int_{\RR} \vartheta(\xi)e^{-ix\xi} \d\xi.$$
Recall that $\chi$ has compact support and is smooth. 
Fourier inversion and partial integration therefore yield the bound
\begin{equation}\label{eq:theta-bound}
\vartheta(\xi) \ll_E (1 + |\xi|)^{-E} 
\end{equation}
for any  $E > 0$.
Following \cite[App.~D]{GT}, we make use of this rapid decay to truncate
the integral representation of $\chi$ which will enable us to swap
integrations and summations later on.
Let $I=\{\xi \in \RR : |\xi|\leq \sqrt{\log T^{\gamma}}\}$, then for any 
 $m\in \NN$ we have
\begin{equation}
\begin{split}\label{eq:chi-integral}
\chi\left(\frac{\log m}{\log T^{\gamma}}\right) 
&= \int_{\RR}
   {m}^{-\frac{1+i\xi}{\log T^{\gamma}}} \vartheta(\xi)~\d\xi \\
&= \int_{I} m^{-\frac{1+i\xi}{\log T^{\gamma}}} \vartheta(\xi)~\d\xi
 + O_{E}\left(\frac{m^{-1/\log T^{\gamma}}}{ (\log T^{\gamma})^{E}}\right).
\end{split}
\end{equation}

\begin{proof}[Proof of Lemma \ref{lem:sieve-upperbound}]
We begin by estimating, for any parameter $T'\leq T$ and $0\leq A' < W$, 
the sum
$$
S(T') = \sum_{m\leq T'/W} \nu_{\mathrm{sieve}}^{(T)} (W m+A').
$$
We will show that 
\begin{equation}\label{eq:rain}
S(T')
\ll_E \frac{T'}{W} (\log T)^{\delta - 1} 
 \prod_{\substack{p\in \mathcal{P}_2\\ p<w(T)}} 
 \left(1 - \frac{1}{p}\right)^{-1} 
+ \frac{T'}{W (\log T)^{E}} 
+ T^{2\gamma},
\end{equation}
for any $E>0$.
In order to bound the average order of 
${\nu'}_{\mathrm{sieve}}^{(T)}(Wm+A)$ we apply this estimate with 
$T'= T/q-A'$, for $q\leq T^{\gamma}$ and for $A'$ with $qA' \equiv A 
\bmod{W}$. 
Since $W \ll T^{o(1)}$, the first term in the bound dominates and
we obtain
\begin{align*}
\Bigg(
(\log T)^{1-\delta} 
\prod_{\substack{p\in \mathcal{P}_2\\ p<w(T)}} 
\left(1 - \frac{1}{p}\right)
\Bigg)
\sum_{m\leq (T-A)/W} {\nu'}_{\mathrm{sieve}}^{(T)} (W m+A)
&\ll \frac{T}{W} 
\sum_{\substack{q \in \<\cP_2\> \\ v_p(q)\not=1 ~\forall p}}  
\frac{\tau(q)^n}{q}  \\
&\ll \frac{T}{W}.
\end{align*}
This shows that the lemma will follow if we can establish \eqref{eq:rain}.

Let $d \in \<\bar\cP_2\>$.
We first note that 
$$
\#\left\{m\leq T'/W: W m+A'\equiv 0\bmod{d}\right\}=
\frac{T'}{Wd}+O(1),
$$
since $\gcd(d,W)=1$.
Hence 
\begin{align*}
S(T')
&= 
 \sum_{m\leq T'/W} \sum_{\substack{d,d' \in \<\bar\cP_2\> \\ [d,d']|W m+A'}}
 \mu(d) \mu(d') 
 \chi\Big(\frac{\log d}{\log T^{\gamma}}\Big)
 \chi\Big(\frac{\log d'}{\log T^{\gamma}}\Big) \\
&= 
 \sum_{\substack{d,d' \in \<\bar\cP_2\> }}
 \mu(d) \mu(d') 
 \chi\Big(\frac{\log d}{\log T^{\gamma}}\Big)
 \chi\Big(\frac{\log d'}{\log T^{\gamma}}\Big) 
  \left(\frac{T'}{W[d,d']}+O(1)\right).
\end{align*}
The overall contribution from the error term is $O(T^{2\gamma})$.
Applying \eqref{eq:chi-integral}, we obtain
\begin{align*}
S(T')=~& \frac{T'}{W}
 \sum_{\substack{d,d' \in \<\bar\cP_2\> }}
 \frac{ \mu(d) \mu(d')}{[d,d']}
 \int_{I} \int_{I} 
 d^{-\frac{1+i\xi}{\log T^{\gamma}}}
 d'^{-\frac{1+i\xi'}{\log T^{\gamma}}}
 \vartheta(\xi) \vartheta(\xi') ~\d\xi~\d\xi' \\
& 
 + O_{E}\Big(\frac{T' }{W (\log T^{\gamma})^{E}} 
   \sum_{d,d'} 
   \frac{(dd')^{-1/\log T^{\gamma}}}{[d,d']} 
   \Big) + O(T^{2\gamma}),
\end{align*}
for any $E>0$.
Let us denote the main term, temporarily, by $M(T')$.
The first of the error terms may be bounded by noting that
\begin{align*}
\frac{T' }{W (\log T^{\gamma})^{E}} 
   \sum_{d,d'} 
   \frac{(dd')^{-1/\log T^{\gamma}}}{[d,d']} 
&\leq
\frac{T' }{W (\log T^{\gamma})^{E}} 
   \sum_{d,d',d''} (dd'd'')^{-1-1/\log T^{\gamma}}\\
&\leq 
\frac{T' }{W (\log T^{\gamma})^{E-3}}.
\end{align*}
Thus both error terms are satisfactory for \eqref{eq:rain}, 
on redefining $E$.

It remains to estimate the main term $M(T')$. On interchanging the sum
over $d,d'$ with the double integral and taking the Euler product, we obtain
\begin{align*}
|M(T')|&\leq
\frac{T'}{W} \left|\int_{I} \int_{I} 
 \prod_{p \in \bar\cP_2} 
 \Big(1 
 - p^{-1-\frac{1+i\xi}{\log T^{\gamma}}}
 - p^{-1-\frac{1+i\xi'}{\log T^{\gamma}}}
 + p^{-1-\frac{1+i\xi + 1 + i\xi'}{\log T^{\gamma}}}\Big)
 \vartheta(\xi) \vartheta(\xi') ~\d\xi~\d\xi' \right|\\
& \ll \frac{T'}{W} \int_{I} \int_{I} 
\left| \Pi(\xi,\xi')
 \vartheta(\xi) \vartheta(\xi')\right| ~\d\xi~\d\xi',
\end{align*}
where
$$
\Pi(\xi,\xi')=
 \prod_{p \in \bar\cP_2} 
 \Big(1- p^{-1-\frac{1+i\xi}{\log T^{\gamma}}} \Big)
 \Big(1- p^{-1-\frac{1+i\xi'}{\log T^{\gamma}}} \Big)
 \Big(1 + p^{-1-\frac{1+i\xi + 1 + i\xi'}{\log T^{\gamma}}}\Big).
$$
We denote the final integral by
$$
\mathcal I=\int_{I} \int_{I} 
\left| \Pi(\xi,\xi')
 \vartheta(\xi) \vartheta(\xi')\right| ~\d\xi~\d\xi'.
$$
Our aim is to estimate $\mathcal I$ by bounding the
product $\Pi(\xi,\xi')$ from above.

The product $\Pi(\xi,\xi')$ is intimately related to the Euler product 
$$
F(s)
=\prod_{\substack{p\in \cP_2}} 
\Big(1 - \frac{1}{p^{s}}\Big)^{-1}, \quad (\Re(s)>1) 
$$
that we met in \eqref{eq:F(s)}. 
From Lemma \ref{lem:F(s)} we deduce that 
there is a function $G(s)$, which is holomorphic and non-zero in the
closed half-plane $\Re(s)\geq 1$, such that 
$$
F(1+s) 
= G(1+s)\left(\frac{1}{s} + O(1)\right)^{1-\delta} 
= G(1+s)\frac{1}{s^{1-\delta}}(1 + O(|s|))^{1-\delta} 
$$
for $\Re (s) > 0$. 
The primes in $\Pi(\xi,\xi')$ only run over $\bar\cP_2$.
Thus, since we may freely disregard finitely many primes, it suffices to
ally our knowledge of $F(s)$ with an investigation of 
$$
\tilde F(s)
=\prod_{\substack{p\in \cP_2 \\ C<p<w(T)}} 
 \Big(1 - \frac{1}{p^{s}}\Big)^{-1}
$$
near $s=1$, for a suitable absolute constant $C=O(1)$. 
Lemma \ref{lem:E-asymp} applies to $\tilde F(s)$ with
$h=\1_{\<\mathcal{P}_2\>}$ and $H=1$.
The primes in $\tilde F$ run up to $x=w(T)=\log\log T$ and we are
interested in $s_0=s$ satisfying
$\Re (s) = (\log T^{\gamma})^{-1}$ or
$\Re (s) = 2(\log T^{\gamma})^{-1}$
and 
$|s| \leq 3 (\log T^{\gamma})^{-1/2}
 = 3 \gamma^{-1/2} e^{-w(T)/2}$. 
Thus the conditions of Lemma \ref{lem:E-asymp} are satisfied and we obtain
\begin{equation*}
|\tilde F(1+s) |
\asymp \tilde F(1) 
\asymp \prod_{\substack{p\in \mathcal{P}_2\\ p<w(T)}}
\left(1 - \frac{1}{p}\right)^{-1}.
\end{equation*}
Thus, invoking \eqref{eq:theta-bound}, we obtain
\begin{align*}
 \mathcal I 
&\ll \int_{I} \int_{I}
 \Big|\frac{1+i\xi}{\log T^{\gamma}}\Big|^{1 - \delta}
 \Big|\frac{1+i\xi'}{\log T^{\gamma}}\Big|^{1 - \delta}
 \Big|\frac{1+i\xi + 1 + i\xi'}{\log T^{\gamma}}\Big|^{\delta - 1} \\
& \qquad 
 \times  \left|\tilde F\Big(1+\frac{1+i\xi}{\log T^{\gamma}}\Big)
 \tilde F\Big(1+\frac{1+i\xi'}{\log T^{\gamma}}\Big)
 \tilde F\Big(1+\frac{2+i(\xi+\xi')}{\log T^{\gamma}}\Big)^{-1}
 \theta(\xi)\theta(\xi')\right|~\d\xi~\d\xi' \\
&\ll (\log T)^{\delta - 1} \tilde F(1)
 \int_{I} \int_{I}
 (1+|\xi|)^{-2} (1+|\xi'|)^{-2} ~\d\xi~\d\xi' \\
&\ll (\log T)^{\delta - 1} 
\prod_{\substack{p\in \mathcal{P}_2\\ p<w(T)}} 
\left(1 - \frac{1}{p}\right)^{-1}.
\end{align*}
This concludes the proof of \eqref{eq:rain} and so completes the proof of
the lemma. 
\end{proof}

\subsection{Conclusion}
Let $1\leq i\leq r$.
We are finally in a position to reveal the majorant for the 
representation function $R_i':\{-T,\dots,T\} \to \RR_{\geq 0}$ in
\eqref{eq:Ri'}, where  $\mathcal{S}_{C_1,T}$ is the exceptional set from
Definition~\ref{def:ex}.

Let $A\in \mathcal{A}_i$ 
and let $W_0 = \prod_{p\mid D_{K_i}}p^{\alpha(p)}$.
Let 
$A'= \gcd(A,W)$ and
let $A_0=\gcd(A,W_0)$. Then 
$$
\gcd\left( \frac{A'}{A_0}, \frac{Wt+A}{A'}\right)=1,
$$
for any $t\in \ZZ$. Hence it follows from 
 Lemma \ref{lem:7.1'} 
 \begin{equation}\label{eq:lemma7.1'-implies}
R_i(Wt+A) 
\ll 
 \frac{\rho_i(W_0,A)}{W_0^{n_i-1}} 
 r_{K_i}\left(\frac{A'}{A_0}\right)
 r_{K_i}\left(\frac{Wt+A}{A'}\right),
\end{equation}
when $Wt+A>0$.
Put $m=Wt+A$
and assume that $m\not \in  \mathcal{S}_{C_1,T}$.
Then $A'=m'$, in the notation of \eqref{eq:m'}.
Combining the majorants \eqref{eq:nu_i} and \eqref{eq:sieve'} according
to \eqref{eq:r-model-bound'}, we obtain 
\begin{equation}\begin{split} \label{eq:r_K-majorant}
r_{K_i}\left(\frac{m}{m'}\right) 
&\leq
 r_{i,\mathrm{res}} \left(\frac{m}{m'}\right)
\sum_{\substack{q \in \<\bar\cP_2^{(i)}\> \\ 
   v_p(q)\not=1 ~\forall p }} 
 \1_{q|m}
 \tau(q)^{n_i}
 \1_{\<\cP_0^{(i)} \cup \cP_1^{(i)}\>}
 \left(\frac{m}{qm'}\right)
\\
& \ll 
 \nu_i^{(T)}(m)
 \sum_{\substack{q \in \<\bar\cP_2^{(i)}\> \\ 
 v_p(q)\not=1 ~\forall p  }} \1_{q\mid m}
 \tau(q)^{n_i}
 \chi\Big(\frac{\log q}{\log T^{\gamma}}\Big) 
 \nu_{i,\mathrm{sieve}}^{(T)}
 \left(\frac{m}{q}\right) \\
&=
 \nu_i^{(T)}(m) 
 {\nu'}^{(T)}_{i,\mathrm{sieve}}(m).
\end{split}
\end{equation}
Here we have 
noted that 
$
\nu_{i,\mathrm{sieve}}^{(T)}(m/(qm'))=  \nu_{i,\mathrm{sieve}}^{(T)}(m/q) 
$
and 
truncated the $q$ summation using $\chi$.
To see that the latter  is valid, suppose that $q\mid m$ with $q> T^{\gamma/2}$
and recall that $\chi(x)=1$ for $x\in[-1/2,1/2]$. 
If there is a prime divisor 
$p> T^{1/(\log \log T)^3}$ of $q$, then there exists a divisor $p^2\mid m$
with $p^2>(\log T)^{C_1}$, since $q$ is square-full, which implies that $m$
is rough in the sense of part (1) of Lemma \ref{lem:erdos}. 
If, on the other hand,  $p\leq  T^{1/(\log \log T)^3}$ for every 
$p\mid q$, then $m$ is smooth in the sense of part (2) of the lemma. 
Neither case can occur since $m\not\in\mathcal{S}_{C_1,T}$.

Our final task is to check condition (ii) from the start of Section
\ref{s:gen-maj}, which states that the mean value of our majorant
should agree with the mean value of  $R_i'$, with respect to $T$.

\begin{lemma}\label{lem:majorant-average}
Let
 $1\leq i \leq r$ and define $\mathcal{A}_i$ as in \eqref{eq:def-A}.
Suppose $1\leq A <W$ with $A \in \mathcal{A}_i$ and write
$A'=\gcd(A,W)$.
Then
$$
\frac{\rho_i(W_0,A)}{W_0^{n_i-1}} 
r_{K_i}\left(\frac{A'}{A_0}\right)
\EE_{m \leq (T-A)/W} 
\nu_i^{(T)}(Wm + A) 
{\nu'}^{(T)}_{i,\mathrm{\rm sieve}}(Wm+A)
\asymp 
\frac{\rho_i(W,A;M)}{W^{n_i-1}},
$$
provided the parameter $\gamma$ appearing in $\nu_i^{(T)}$ and 
${\nu'}^{(T)}_{i,\mathrm{\rm sieve}}$ is sufficiently small.
\end{lemma}
\begin{proof}
To begin with, recall that $\rho_i(W,A;M)>0$ and that $v_p(A) < v_p(M) \ll 1$ for 
all $p\mid M$ by our assumption \eqref{eq:def-A}.
Thus, for $p\mid M$, the first part of Lemma \ref{lem:A>0} yields
$$
0< \frac{\rho_i(p^{\alpha(p)},A;p^{v_p(M)})}{p^{\alpha(p)(n_i-1)}} 
\leq \frac{\rho_i(p^{\alpha(p)},A)}{p^{\alpha(p)(n_i-1)}}
\ll 1.
$$
Since the above is positive, we may use Lemma \ref{lem:C6.4} 
to deduce a matching 
lower bound  as in \eqref{eq:dove}.
Thus, the multiplicativity of $\rho_i$ implies that 
\begin{equation}\label{eq:sunny'}
\frac{\rho_i(W,A)}{W^{n_i-1}}\ll 
\frac{\rho_i(W,A;M)}{W^{n_i-1}}
\leq  \frac{\rho_i(W,A)}{W^{n_i-1}}.
\end{equation}
Next we note that 
$A'/A_0=\gcd(W/W_0,A)$.
Hence  the second part of Lemma \ref{lem:A>0} yields
\begin{equation}\label{eq:cloudy}
\frac{\rho_i(W,A)}{W^{n_i-1}}
= 
 \frac{\rho_i(W_0,A)}{W_0^{n_i-1}} 
 r_{K_i}\left(\frac{A'}{A_0}\right)
 \prod_{\substack{p<w(T)\\ p\nmid D_{K_i}}} 
 \Big(1-\frac{1}{p}\Big)^{-1}
 \prod_{\substack{\fp|p\\ \fp\subset \fo_{K_i}}}
 \Big(1-\frac{1}{\n \fp}\Big).
\end{equation}
This  reduces our task to establishing, for sufficiently small $\gamma$, the 
estimate
$$
\EE_{m \leq (T-A)/W} 
 \nu_i^{(T)}(Wm + A) 
 {\nu'}^{(T)}_{i,\mathrm{sieve}}(Wm+ A)
\asymp  
 \prod_{\substack{p<w(T)}} 
 \Big(1-\frac{1}{p}\Big)^{-1}
 \prod_{\substack{\fp|p\\ \fp\subset \fo_{K_i}}}
 \Big(1-\frac{1}{\n \fp}\Big).
$$

We temporarily set
$$
S(T')=
\EE_{m \leq T'/W}
 \nu_i^{(T)}(Wm + A)
 {\nu'}^{(T)}_{i,\mathrm{sieve}}(Wm+ A),
$$
for any $T'\leq T$.
The two factors of the majorant are truncated divisor sums. 
According to the discussion following \eqref{eq:nu_i}, the first function
$\nu_i^{(T)}$ is constructed from divisors in $\<\bar\cP_1^{(i)}\>$
provided $T$ is sufficiently large, 
whereas the function ${\nu'}^{(T)}_{i,\mathrm{sieve}}$ is constructed from
divisors belonging to $\<\bar\cP_2^{(i)}\>$.  
In particular the divisors used in the construction of the former are all
coprime to the divisors appearing  in the latter. 
We therefore deduce  (cf.~\cite[p.~262]{lm1}) that
\begin{align*}
S(T')=~& \EE_{m \leq T'/W} \nu_i^{(T)}(Wm + A)\\ &\times
\EE_{m \leq T'/W} {\nu'}^{(T)}_{i,\mathrm{sieve}}(Wm + A)
+O(T^{O(\gamma)}W/T'),
\end{align*}
which provides an asymptotic formula whenever $T^{O(\gamma)}W = o(T')$.
Combining Lemmas \ref{lem:nu_i-average} and
\ref{lem:sieve-upperbound}, this allows us to deduce the upper bound
\begin{align*}
S(T-A)
&\ll
 \prod_{p<w(T)} 
 \Big(1-\frac{1}{p}\Big)^{-1}
 \prod_{\substack{\fp|p\\ \fp\subset \fo_{K_i}}}
 \Big(1-\frac{1}{\n \fp}\Big)
\end{align*}
if $\gamma$ is sufficiently small. 
To obtain the lower bound for $S(T-A)$, we combine 
\eqref{eq:lemma7.1'-implies}, \eqref{eq:r_K-majorant}, \eqref{eq:sunny} to get
\begin{align*}
\frac{\rho_i(W_0,A)}{W_0^{n_i-1}} 
r_{K_i}\left(\frac{A'}{A_0}\right)
S(T-A)
&\gg \EE_{m \leq (T-A)/W} R'_{i}(Wm+A)\\
&\gg 
\frac{\rho_i(W,A;M)}{W^{n_i-1}}.
\end{align*}
But then it follows that 
$$S(T-A)\gg  \prod_{p<w(T)} 
 \Big(1-\frac{1}{p}\Big)^{-1}
 \prod_{\substack{\fp|p\\ \fp\subset \fo_{K_i}}}
 \Big(1-\frac{1}{\n \fp}\Big),
 $$
by  \eqref{eq:sunny'}  and \eqref{eq:cloudy}.
\end{proof}

For every $i \in \{1,\dots,r\}$ let $A_i$ be such that $1\leq A_i< W$ and
$A_i\bmod{W}\in\cA_i$.  The proof of 
Lemma~\ref{lem:majorant-average} shows that there is a function
\begin{equation}\label{eq:clock}
\phi_i(T;A_i)
\asymp
\prod_{p<w(T)} 
 \Big(1-\frac{1}{p}\Big)
 \prod_{\substack{\fp|p\\ \fp\subset \fo_{K_i}}}
 \Big(1-\frac{1}{\n \fp}\Big)^{-1}
\end{equation}
such that
\begin{equation}\label{eq:C_AB-def} 
\phi_i(T;A_i)
\EE_{m < T/W} 
\nu_i^{(T)}(Wm + A_i) {\nu'}_{i, \mathrm{sieve}}^{(T)}(Wm + A_i) 
=1.
\end{equation}
We define the joint normalised majorant function 
\begin{equation}\label{eq:sim}
\varpi^{(T)}_{A_1,\dots,A_r}(m)= 
\frac{1}{r} \sum_{i=1}^r 
\phi_i(T;A_i)
\nu_i^{(T)}(Wm + A_i) {\nu'}^{(T)}_{i, \mathrm{sieve}}(Wm + A_i).
\end{equation}
We will often write
$$
\varpi^{(T)}(m)=\varpi^{(T)}_{A_1,\dots,A_r}(m),
$$
for short. It 
satisfies
$
\EE_{m < T/W} \varpi^{(T)}(m) = 1.
$
Moreover, $\varpi^{(T)}$ simultaneously majorises the
{\em normalised} counting functions
$$
\Big(\frac{\rho_i(W,A_i;M)}{W^{n_i-1}}\Big)^{-1}
R'_i(Wm+A_i), 
$$
for $1 \leq i \leq r$ and $R'_i$ as in \eqref{eq:Ri'},
in the sense of (i) from the start of Section \ref{s:gen-maj}.

\bigskip
\section{The majorant is pseudorandom}\label{s:linear-forms}
Let $A_i \in \cA_i$ for $1 \leq i \leq r$, in the notation of \eqref{eq:def-A},
and recall the definition \eqref{eq:sim} of 
$\varpi^{(T)}=\varpi_{A_1,\dots,A_r}^{(T)}$.
Given $D >1$, our aim in this section is to show that the family 
$(\varpi^{(T)})_{T \in \NN}$ gives rise to a family of $D$-pseudorandom
majorants, in the sense of \cite[\S 6]{GT} with $m_0=d_0=L_0=D$, provided
that the parameter $\gamma$  appearing in the truncations is sufficiently
small.  In our setting it suffices to consider $D \ll_L 1$, where $L$ is
as in \eqref{eq:def-L}.

For each $T$ let $\tilde T$ be a prime number such that 
$T/W < \tilde T \ll_{L} T/W$.
Choosing $\tilde T$ sufficiently large in terms of $L$ allows us to pass
from counting problems within the set of integers $\{1,\dots, [T/W]\}$ to
counting problems in the group $\ZZ/\tilde T\ZZ$, without creating new
solutions due to the wrap-around effect.
The majorants are extended to $\ZZ/\tilde T\ZZ$ by defining
${\varpi'}^{(T)}:\ZZ/\tilde T\ZZ \to \RR_{>0}$ via
$$
{\varpi'}^{(T)}(m)=
\begin{cases}
(1+\varpi^{(T)}(m))/2, & \text{if } m \leq T/W, \\
 1,                     & \text{if } T/W <m \leq \tilde T.
\end{cases}
$$
By \cite[App.~D]{GT} it suffices to prove the following two propositions
in order to show that $({\varpi'}^{(T)})_{T \in \NN}$ is a family of
$D$-pseudorandom majorants. As indicated above, we will apply them with
$D\ll_L 1$.

\begin{proposition}[$D$-Linear forms estimate]
\label{p:linear-forms}
Let $T \in \NN$, $T' = [\frac{T}{W}]$, and let $D>1$. 
Suppose that $1 \leq r',s' \leq D$ and let
$\mathbf{h}=(h_1,\dots,h_{r'}):\ZZ^{s'} \to \ZZ$
be a system of linear polynomials whose non-constant  parts are pairwise
non-proportional.
Suppose that coefficients of each $h_i$, other than possibly the constant terms, 
are bounded in absolute value by $D$, while $h_i(0)= O_D(T)$. 
Suppose $\mathfrak K \subset [-1,1]^{s'}$ is a convex body such
that $\mathbf{h}(T' \mathfrak{K}) \subset [1,T']^{r'}$ and 
$\vol(\mathfrak{K})\gg 1$.
Then we have 
\begin{align}\label{eq:linearforms-1} 
 \frac{1}{\vol(T'\mathfrak K)} \sum_{\m \in \ZZ^{s'} \cap T'\mathfrak K} 
 \prod_{j=1}^{r'} \varpi_{A_1,\dots,A_r}^{(T)}(h_j(\m))
= 1 
 + o_D(1),
\end{align}
provided $\gamma$ is small enough.
\end{proposition}

\begin{proposition}[$D$-correlation estimate]
\label{p:correlation-condition}
Let $T \in \NN$, $T' = [\frac{T}{W}]$ and let $D>1$.
Then there exists a function 
$\sigma: \{-T',\dots,T'\} \to \RR_{\geq 0}$ with bounded moments
$$
\EE_{|m| \leq T'} \sigma^q(m) \ll_{D,q} 1,
$$
such that for every discrete interval $I \subset \{1, \dots, T'\}$, 
every $1 \leq d \leq D$, every  $(i_1, \dots, i_{d}) 
\in \{1,\dots,r\}^{d}$ and every choice of (not necessarily distinct)
$a_1,\dots,a_{d} \in \{1, \dots, T'\}$, we have
\begin{align*}
\sum_{m \in I} \prod_{j=1}^d \phi_{i_j}(T;A_{i_j})
  &\nu_{i_j}^{(T)}(W(m+a_j) + A_{i_j})
  {\nu'}_{i_j, \mathrm{sieve}}^{(T)}(W(m+a_j) + A_{i_j}) \\
&\leq T' \sum_{1 \leq j < j' \leq d} \sigma(a_j - a_{j'}),
\end{align*}
provided $\gamma$ is small enough.
\end{proposition}

In proving Propositions \ref{p:linear-forms} and 
\ref{p:correlation-condition}, we will allow all of our implied constants
to depend on the parameter $D$. 
We begin with the proof of the former. 
Unravelling definitions, we see that \eqref{eq:linearforms-1} is implied
by the estimate
\begin{equation}
\begin{split}
 \label{eq:linearforms-2} 
\frac{1}{\vol(T'\mathfrak K)}
 &\sum_{\m \in \ZZ^{s'} \cap T'\mathfrak K} 
 \prod_{j=1}^{r'} 
 \nu_{i_j}^{(T)}(Wh_j(\m) + A_{i_j})
 {\nu'}^{(T)}_{i_j, \mathrm{sieve}}(Wh_j(\m) + A_{i_j}) \\
&= \left( 1      + o(1)
   \right) 
   \prod_{j=1}^{r'}
   \phi_{i_j}(T;A_{i_j})^{-1},
\end{split}
\end{equation}
for every collection of indices
$1\leq i_1,\dots,i_{r'} \leq r$.  Here we have 
$
\phi_{i_j}(T;A_{i_j})^{-1} \asymp 
\Pi_{i_j}$, 
by \eqref{eq:clock}, where
\begin{equation}\label{eq:def-Pi}
\Pi_i=\prod_{p<w(T)} 
 \Big(1-\frac{1}{p}\Big)^{-1}
 \prod_{\substack{\fp|p\\ \fp\subset \fo_{K_i}}}
 \Big(1-\frac{1}{\n \fp}\Big).
\end{equation}

The strategy to proving \eqref{eq:linearforms-2} is the same as in \cite[\S9]{lm1}, which is
related to that of \cite[App.~D]{GT}.
Inserting all definitions and writing $g_i=\mu * r_{i,\mathrm{res}}$,
we have
\begin{align*}
 \nu_{i}^{(T)}(Wm + A)
  {\nu'}^{(T)}_{i, \mathrm{sieve}}&(Wm + A) \\
=~& 
 \sum_{\kappa = 4/\gamma}^{[(\log \log T)^3]}
 \sum_{\lambda = \lceil\log_2 \kappa - 2\rceil}^{[\log_2((\log \log T)^3)]}
 \sum_{u \in U(\lambda,\kappa)}
 2^{n_i\kappa} \1_{u|Wm + A} \, r_{i,\mathrm{res}}(u)\\
 &\times
 \sum_{\substack{d \in \<\bar\cP_1^{(i)} \> \\ \gcd(d,u)=1}} 
  \1_{d|Wm + A} \, g_i(d)
  \chi\Big(\frac{\log d}{\log T^{\gamma}}\Big)   \\
&\times\sum_{\substack{q\in \<\bar\cP_2^{(i)}\> \\
 v_p(q) \not= 1~\forall p}} \1_{q|Wm + A} \,
 \tau(q)^{n_i} \chi\Big(\frac{\log q}{\log T^{\gamma}}\Big)
 \Big( \sum_{\substack{e \in \<\bar\cP_2^{(i)}\> \\ qe|Wm + A}} \mu(e) 
  \chi\Big(\frac{\log e}{\log T^{\gamma}}\Big) 
 \Big)^2.
\end{align*}
Here the restriction to $d\in \langle \bar\cP_1^{(i)}\rangle$ arises from 
Remark \ref{rem:coprime}
and the fact that $g_i(d)=0$ when $d$ has a prime factor $p<w(T)$.
Noting that $\gcd(qe,d)=1$, the right hand side is seen to be 
\begin{align*}
&\sum_{\kappa =4/\gamma}^{[(\log \log T)^3]}
 \sum_{\lambda = \lceil\log_2 \kappa - 2\rceil}^{[\log_2((\log \log T)^3)]}
 \sum_{u \in U(\lambda,\kappa)}
 \sum_{\substack{d \in \<\bar\cP_1^{(i)} \> \\ \gcd(d,u)=1}} 
 \sum_{\substack{q\in \<{\bar\cP_2^{(i)}}\> \\ v_p(q) \not= 1~\forall p}}\\
&\qquad\times  \sum_{e,e' \in \<\bar\cP_2^{(i)}\>}
 2^{n_i\kappa} 
 r_{i,\mathrm{res}}(u)
 \tau(q)^{n_i}
 \mu(e)\mu(e')
 g_i(d)
 \1_{\Delta|Wm+A}
 \prod_{x \in \{d,e,e',q\}} 
 \chi\Big(\frac{\log x}{\log T^{\gamma}}\Big),
\end{align*}
where
$
\Delta=\lcm(u,d,qe,qe').
$
Together, Remark \ref{back} and  the compact support of $\chi$ ensure that all 
divisors $d,e,e' ,q,u$ are bounded by $T^{\gamma}$. 
For each $1\leq j \leq r'$ we define the linear polynomial
$$
h'_j(\m) = Wh_j(\m)+A_{i_j}.
$$
We may assume that $T$ is  sufficiently large in terms of $D$ to ensure that the
non-constant
 parts of the polynomials $h_1, \dots, h_{r'}$ are pairwise
non-proportional modulo any prime $p>w(T)$.
The same then holds for the polynomials $h'_1,\dots,h'_{r'}$.

Let $\Delta_j=\lcm(u_j,d_j,q_je_j,q_je'_j)$, for $1\leq j\leq r'$.
We are interested in estimating the cardinality 
\begin{align*}
\#\{\m \in \ZZ^{s'} &\cap T'\mathfrak{K}: 
 \Delta_j \mid h'_j(\m)\}\\
 &=
 \sum_{\substack{\s\bmod{\Delta_{\u,\bd,\q,\e,\e'}}\\  
       \Delta_j \mid h'_j(\m)}}
\hspace{-0.4cm}
\#\{\m \in \ZZ^{s'}\cap T'\mathfrak{K}: 
    \m\equiv \s\bmod{\Delta_{\u,\bd,\q,\e,\e'}}\},
 \end{align*}
where
$
\Delta_{\u,\bd,\q,\e,\e'} 
= \lcm (\Delta_1, \dots ,\Delta_{r'})
\leq T^{O(\gamma)}.
$
Extending the notion of local divisor densities multiplicatively from 
\eqref{eq:def-div.density}, with $\mathcal{U}_m=(\ZZ/p^m\ZZ)^{s'}$ and the
set of polynomials $\mathbf{h}'$, the outer sum has cardinality 
$
\alpha_{\h'}(\Delta_1,\dots,\Delta_{r'})
\Delta_{\u,\bd,\q,\e,\e'}^{s'}.
$
The inner cardinality is equal to
$$
\# \left((\Delta_{\u,\bd,\q,\e,\e'} \ZZ^{s'} + \s) 
         \cap T' \mathfrak{K}\right)
=\#\left(\ZZ^{s'}
   \cap (\Delta_{\u,\bd,\q,\e,\e'}^{-1} T' \mathfrak{K}
   + \Delta_{\u,\bd,\q,\e,\e'}^{-1} \s) \right).
$$
We may therefore apply Lemma \ref{lem:2} with   
$\mathcal B = \mathfrak{K}$ and $T=\Delta_{\u,\bd,\q,\e,\e'}^{-1} T'$ to
each of the above cardinalities.
This leads to the conclusion that
\begin{equation}
\begin{split}
 \label{eq:linearforms-3} 
\frac{1}{\vol(T'\mathfrak K)}
& \sum_{\m \in \ZZ^{s'} \cap T'\mathfrak K} 
 \prod_{j=1}^{r'}
 \nu_{i_j}^{(T)}(h_j'(\m) )
 {\nu'}^{(T)}_{i_j, \mathrm{sieve}}(h_j'(\m) ) 
\\ 
=~&
 \sum_{\bka}
 \sum_{\bla}
 \sum_{\substack{\u\\  u_j \in U(\lambda_j,\kappa_j)}}
 \sum_{\substack{\bd \\ \gcd(d_j,u_j)=1}} 
 \sum_{\q}
 \sum_{\e,\e'} 
 \bigg(
 \alpha_{\h'}(\Delta_1,\dots,\Delta_{r'})
 + O\Big(\frac{T'^{-1 + O(\gamma)}}{\vol( \mathfrak K)}\Big)
 \bigg)\\
&\times \prod_{j=1}^{r'}
 2^{\kappa_jn_{i_j}} 
 \tau(q_j)^{n_{i_j}}
 \mu(e_j)\mu(e'_j)
 g_{i_j}(d_j)  
 \prod_{x \in \{d_j,e_j,e'_j,q_j\}} 
 \chi\Big(\frac{\log x}{\log T^{\gamma}}\Big),
\end{split}
\end{equation}
where $\bka,\bla,\u,\bd,\q,\e,\e' \in \ZZ^{r'}$ are assumed to 
satisfy the correct multiplicative restrictions component-wise.
Thus, for example, the sum over $\bd$ is restricted to
$$
\{\bd: d_j \in \<\bar\cP_{1}^{(i_j)}\>, 1\leq j\leq r'\}.
$$
Similarly, those over $\q$,$\e$ and $\e'$ are restricted to 
$$
\{\x: x_j \in \<\bar\cP_{2}^{(i_j)}\>, 1\leq j\leq r'\}.
$$
We  assume, furthermore,  that all coordinates of $\q$ satisfy
$v_p(q_j) \not= 1$ for all primes $p$.

We begin by examining the error term in \eqref{eq:linearforms-3}.
As mentioned above, each of the sums over $u_j$, $d_j$, $e_j$, $e'_j$ and
$q_j$ have at most $T^{\gamma}$ terms.
Together with the trivial bounds
$\tau(q) \leq T^{\gamma}$ for $q \leq T^{\gamma}$ and
$
2^{ \kappa_j} \leq 2^{(\log \log T)^3}\ll  T^{o(1)},
$
this implies that the error term makes a total contribution of
$$
O\Big(\frac{T'^{-1 + O(\gamma)}}{\vol \mathfrak K}\Big)=o(1),
$$
by our assumptions on $\mathfrak{K}$.

The main term will now be analysed in much the same way as in
\cite[\S6]{lm0} and \cite[\S9]{lm1}.
Our majorant very closely resembles that from \cite{lm1},
the latter in fact being a special case of it.
The analysis of  \eqref{eq:linearforms-3} is therefore
only a minor adaptation of what is established in \cite[\S9]{lm1}.
Given the length of the argument we include an overview here as guidance,
and only include the details of the more complicated proofs where it
may not be immediately clear that the corresponding argument from
\cite[\S9]{lm1} still applies.

Any prime $p\mid \Delta$ satisfies $p>w(T)$. 
Hence $\alpha_{\h'}(\Delta_1,\dots,\Delta_{r'})$ will be determined using
the first three alternatives from \eqref{eq:ev-alpha}. 
In particular, 
$$
\alpha_{\h'}(\Delta_1,\dots,\Delta_{r'})
= \prod_{j=1}^{r'}\frac{1}{\Delta_j}
$$
whenever $\Delta_1,\dots,\Delta_{r'}$ are pairwise coprime.
Put $\tilde \Delta_j=d_j q_j \lcm(e_j,e'_j)$, for $1\leq j\leq r'$.
The first step is to show that we may replace
$\alpha_{\h'}(\Delta_1,\dots,\Delta_{r'})$
in the main term by
$$
\frac{\alpha_{\h'}(\tilde\Delta_1,\dots,\tilde \Delta_{r'})}
     {u_1\dots u_{r'}},
$$
at the expense of an overall error term $o(1)$.
To prove this, it suffices to show that we may restrict the
summation to vectors $(\u,\bd,\q,\e,\e')$ for which 
$$
 \gcd(u_j, \tilde \Delta_j)=1
 \quad \text{and} \quad
 \gcd(u_i,u_j \tilde \Delta_j)=1
$$
for all $j$ and all $i\not=j$.
Since $\gcd(d_j,u_j)=1$, it follows from Remark~\ref{rem:coprime} that the first 
condition is always satisfied.
Furthermore, the set of all vectors $(\u,\bd,\q,\e,\e')$ failing the second 
condition makes a negligible contribution.
The proof follows (cf.~the proof of \cite[Claim 2]{lm1}) by the
Cauchy--Schwarz inequality from a second-moment estimate together with a
lower bound on the prime divisors of any $u_j$.

We note that
\begin{equation} \label{eq:c_0<infty}
 \sum_{\bka}
 \sum_{\bla}
 \sum_{\substack{\u\\  u_j \in U(\lambda_j,\kappa_j)}}
 \prod_{j=1}^{r'} 
 \frac{
 2^{\kappa_jn_{i_j}} r_{j,\mathrm{res}}(u_j)}
 {u_{j}}
 < \infty,
\end{equation}
The absolute convergence of this sum follows from the proof of 
Proposition~\ref{p:majorant}.

The next step is to replace $\chi$
by a multiplicative function using  \eqref{eq:chi-integral}.
For 
$1\leq j \leq r'$ and $1 \leq k \leq 4$ we
write $$
z_{j,k} = \frac{1+i\xi_{j,k}}{\log T^{\gamma}}.
$$ 
Likewise we set 
$\d \bxi=\prod_{j,k}\d \xi_{j,k}$
and
\begin{equation}\label{eq:def-jj}
J_j=\mu(e_j)\mu(e'_j) 
 e_j^{-z_{j,1}} e_j'^{-z_{j,2}}
 g_{i_j}(d_j) 
 d_j^{-z_{j,3}}
 \tau(q_j)^{n_{i_j}}
 q_j^{-z_{j,4}},
\end{equation}
for $1\leq j\leq r'$.
With this notation the  new main term is equal to
\begin{align*}
 &\sum_{\bka}
 \sum_{\bla}
 \sum_{\u}
 \prod_{j=1}^{r'} 
 \frac{
 2^{\kappa_jn_{i_j}} r_{j,\mathrm{res}}(u_j)}
 {u_{j}}\\
&\qquad
 \left(
 J(\u)+
 O_{E}\left(
 \frac{1}{(\log T)^{E}}
 \sum_{\substack{\bd\\(d_j,u_j)=1}} 
 \sum_{\q}
 \sum_{\e,\e'} 
 \prod_{j=1}^{r'} \frac{
 H^{\Omega(d_jq_j)}  \alpha_{\h'}(\tilde\Delta_1, \dots, \tilde\Delta_{r'})
 }{ (e_je'_jd_jq_j)^{1/\log T^{\gamma}}}
 \right)
 \right)
\end{align*}
for any $E>0$, 
where $H = \max_{1\leq i \leq r'} 2^{n_i}$ and 
\begin{align*}
J(\u)=~&
 \sum_{\substack{\bd\\(d_j,u_j)=1}} 
 \sum_{\q}
 \sum_{\e,\e'} 
 \alpha_{\h'}(\tilde\Delta_1, \dots, \tilde\Delta_{r'})
 \int_{I} \dots \int_{I} \left(\prod_{j=1}^{r'}
J_j
 \prod_{k=1}^4 \vartheta(\xi_{j,k})\right) \d\bxi.
\end{align*}
The error terms that appear in the next step will again depend on how
small the prime factors of the relevant numbers can be.
This time these are the coordinates of $\bd,\q,\e$ and $\e'$ instead of
$\u$ and we can only assume that the primes are larger than $w(T)$,
which is much smaller than the lower bound on prime factors of the
$u_i$.
For this reason it was essential to treat the $u_j$ separately first,  in
order to make use of the convergence of the sums over $\bka$, $\bla$ and
$U(\lambda_j,\kappa_j)$ when showing that the new error term is negligible.

The next step is to show that we may swap the product over $j$ with all
the sums. That is, we replace 
$\alpha_{\h'}(\tilde\Delta_1, \dots, \tilde\Delta_{r'})$
by 
$(\tilde\Delta_1 \dots \tilde\Delta_{r'})^{-1}$, while only
introducing a small error. 
We will show that
\begin{equation}
\begin{split} \label{eq:linearforms-5}
J(\u)+o\left(\prod_{j=1}^{r'}\Pi_{i_j}\right)=~&
 \int_{I} \dots \int_{I}
\left(
 \prod_{j=1}^{r'}
 \sum_{\substack{d_j\\(d_j,u_j)=1}} 
 \sum_{\substack{q_j\in \<\bar\cP_2^{(i_j)}\>\\v_p(q_j)\not= 1~\forall p}}
 \sum_{e_j,e'_j \in \<\bar\cP_2^{(i_j)}\>} \frac{J_j}{\tilde \Delta_j}
 \prod_{k=1}^4 \vartheta(\xi_{j,k})\right) \d\bxi 
\\ 
=~&
 \prod_{j=1}^{r'}
 \sum_{\substack{d_j\\(d_j,u_j)=1}} 
 \sum_{\substack{q_j\in \<\bar\cP_2^{(i_j)}\>\\v_p(q_j)\not= 1~\forall p}}
 \sum_{\substack{e_j,e'_j \in \<\bar\cP_2^{(i_j)}\>}} 
  \int_{I^4} \left(\frac{J_j}{\tilde \Delta_j}
 \prod_{k=1}^4 \vartheta(\xi_{j,k})\right) \d\xi_{j,1}\dots \d\xi_{j,4},
\end{split}
\end{equation}
where
$\Pi_i$ is given by \eqref{eq:def-Pi} for $1\leq i\leq r$.
Before establishing this estimate, we remark that the final main term is
now a product of $r'$ factors that are independent of each other and
independent of the system $\h$ of linear polynomials that we started with.
In particular, we may consider this estimate in the special case where
$s'=r'=1$ and where $h_1(m)=m$.
Reinstating the sums over $\bka$, $\bla$ and $\u$, this relates the $j$th factor 
of the above product to the average value of the majorant function. 
By \eqref{eq:C_AB-def} and \eqref{eq:c_0<infty}, we therefore deduce that
$$
 \sum_{\bka}
 \sum_{\bla}
 \sum_{\u}
 \prod_{j=1}^{r'} 
 \frac{
 2^{\kappa_jn_{i_j}} r_{j,\mathrm{res}}(u_j)}
 {u_{j}}
 \left( J(\u) + o\left(\prod_{j=1}^{r'}\Pi_{i_j}\right) \right)
 =(1 + o(1)) \prod_{j=1}^{r'} 
 \phi_{i_j}(T;A_{i_j})^{-1}.
$$
This completes the proof of \eqref{eq:linearforms-2}, and hence the proof
of Proposition \ref{p:linear-forms}, subject to the verification of
\eqref{eq:linearforms-5}.
Our proof of  \eqref{eq:linearforms-5} will be undertaken in two steps, 
as recorded in the following two results.
We fix values of $\bka$ and $\bla$ for now.

\begin{lemma} \label{lem:L1-bound}
For each $1\leq j \leq r'$ and each $u_j \in U(\lambda_j,\kappa_j)$, we have 
\begin{align*}
 \int_{I^4} 
 \bigg|
\sum_{\substack{d_j\\ \gcd(d_j,u_j)=1}}
 \sum_{\substack{q_j\in \<\bar\cP_2^{(i_j)}\> \\ 
   v_p(q_j) \not= 1~\forall p}}
&\sum_{e_j,e'_j \in \<\bar\cP_2^{(i_j)}\>} 
\frac{J_j}{\tilde \Delta_j}
 \prod_{k=1}^4 \vartheta(\xi_{j,k})\bigg|
 \d\xi_{j,1}\dots \d\xi_{j,4} 
\ll \Pi_{i_j},
\end{align*}
where $\Pi_{i_j}$ is given by \eqref{eq:def-Pi}, and where the implied constant 
is independent of $\lambda_j$, $\kappa_j$, $u_j$.
\end{lemma}

This lemma corresponds to \cite[Claim 5]{lm1}.
We take the opportunity to provide a full proof here, since the extra
factors $g_{i_j}(d_j) \tau(q_j)^{n_{i_j}}$ implicit in $J_j$ make the 
analysis slightly more delicate.
Moreover, while the proof of \cite[Claim 5]{lm1} is correct, it requires
an application of Lemma \ref{lem:E-asymp}, which is not present in
\cite{lm1}.

\begin{proof}[Proof of Lemma \ref{lem:L1-bound}]
The first step is to express the integrand, which we denote by 
$K=K(\xi_{j,1},\dots, \xi_{j,4})$, as an Euler product. 
The fact that $\Re (z_{j,k}) = (\log T^{\gamma})^{-1}>0$ will allow us to
restrict to the square-free part. 
Recall that 
$$
\frac{J_j}{\tilde \Delta_j}= \frac{\mu(e_j)\mu(e'_j) 
 e_j^{-z_{j,1}} e_j'^{-z_{j,2}}
 g_{i_j}(d_j) 
 d_j^{-z_{j,3}}
 \tau(q_j)^{n_{i_j}}
 q_j^{-z_{j,4}}}
 {d_j q_j \lcm(e_j,e'_j)} 
$$
and put 
$$
\tilde L=\prod_{k=1}^4 (1+|\xi_{j,k}|).
$$
By \eqref{eq:theta-bound} we have 
\begin{align*}
K\ll_{E}~& \tilde L^{-E}
 \bigg| \prod_{p \in \bar\cP_{2}^{(i_j)}}
  (1 - p^{-1-z_{j,1}} - p^{-1-z_{j,2}} + p^{-1-z_{j,1}-z_{j,2}} +
   O(p^{-2})) 
 \bigg|
\\
&\times
 \bigg|
 \prod_{\substack{p \in \bar\cP_{1}^{(i_j)}\\ p \nmid u}}
 (1 + (r_{K_{i_j}}(p) - 1) p^{-1-z_{j,3}} + O(p^{-2})) 
 \bigg|,
\end{align*}
for any $E>3$.
Since for sufficiently large primes $p$ each of the factors of these
products can be analysed via the logarithmic series, we deduce that 
\begin{align*}
K
\ll~& \tilde L^{-E}
 \bigg|
 \prod_{p \in \bar\cP_{2}^{(i_j)}}
 (1 - p^{-1-z_{j,1}})
 (1 - p^{-1-z_{j,2}})
 (1 + p^{-1-z_{j,1}-z_{j,2}})
 \bigg|
\\
& \times
 \bigg|
 \prod_{p \in \bar\cP_{1}^{(i_j)}}
 (1 - p^{-1-z_{j,3}})
 (1 + r_{K_{i_j}}(p) p^{-1-z_{j,3}})
 \bigg|
 \prod_{p | u }
 \bigg(1 + \frac{r_{K_{i_j}}(p) + 1}{p}\bigg).
\end{align*}
Since any prime divisor of $u$ comes from an interval of the form $[y,y^2]$, 
with $y=T^{1/(2^{\lambda + 1})}$, the final product over $p|u$ is easily seen to 
be $O_{n_{i_j}}(1)$, and can be ignored. 

We will now proceed as in Section \ref{s:sieve}. 
Let $1\leq i\leq r$.
Recall that there are  functions $G_1,G_2$, which are 
non-zero and 
 holomorphic  
on $\Re (s) \geq 1$, such that 
\begin{align*}
F_{i,1}(s)&=\prod_{p \in \cP_{1}^{(i)}}\left(1 - \frac{1}{p^{s}}\right)^{-1}
= \zeta^{\delta_{i}}(s) G_1(s), \\
F_{i,2}(s) &=\prod_{p \in \cP_{2}^{(i)}}\left(1 - \frac{1}{p^{s}}\right)^{-1}
= \zeta^{1-\delta_{i}}(s) G_2(s).
\end{align*}
Hence, when $\Re s > 0$ and $|s| \ll 1$, then 
$$
|F_{i,1}(1+s)| \asymp |s|^{-\delta_{i}},
\quad
|F_{i,2}(1+s)| \asymp |s|^{-1+\delta_{i}}.
$$
Likewise, 
$|F_{i,1}^{-1}(1+s)| \asymp |s|^{\delta_{i}}$
 and 
$|F_{i,2}^{-1}(1+s)| \asymp |s|^{1-\delta_{i}}$.
In order to employ these asymptotic orders to bound the integral above,
we apply Lemma~\ref{lem:E-asymp} to deduce that there is an absolute
positive constant $C$ such that each of the three Euler products $E(s)$,
given by
\begin{align*}
H_{i,1}(s)
=\prod_{\substack{p\in \cP_{1}^{(i)}\\C<p<w(T)}}\left(1 
-\frac{1}{p^{s}}\right)^{-1},
\quad 
H_{i,2}(s)
=\prod_{\substack{p\in \cP_{2}^{(i)}\\C<p<w(T)}}\left(1 - 
\frac{1}{p^{s}}\right)^{-1},
\end{align*}
and
$$
H_{K_i}(s)=\prod_{\substack{C<p<w(T)}}
\Big(1+\frac{r_{K_i}(p)}{p^s}\Big)^{-1},
$$
satisfies $|E(1+s)| \asymp E(1)$ and \eqref{eq:number}
when $|s| \ll (\log T^{\gamma})^{-1/2}$.
Recall the definition \eqref{eq:def-Pi} of $\Pi_{i_j}$.
We may conclude that 
\begin{align*}
K \ll~&
\tilde L^{-E}
 |z_{j,1}|^{1-\delta_{i_j}}
 |z_{j,2}|^{1-\delta_{i_j}}
 |z_{j,1} + z_{j,2}|^{\delta_{i_j}-1}
 |z_{j,3}|^{\delta_{i_j}}
 |z_{j,3}|^{-1} \\
&  \times
 H_{i_j,2}(1)
 H_{i_j,1}(1)
 \prod_{\substack{p<w(T)}}
 \Big(1-\frac{r_{K_i}(p)}{p}\Big) \\
\ll~&  \tilde L^{-E}
 \bigg(\frac{1 + |\xi_{j,1}|}{\log T^{\gamma}}\bigg)^{1-\delta_{i_j}}
 \bigg(\frac{1 + |\xi_{j,2}|}{\log T^{\gamma}}\bigg)^{1-\delta_{i_j}}
 \bigg(\frac{1 + |\xi_{j,1}+\xi_{j,2}|}{\log T^{\gamma}}
 \bigg)^{\delta_{i_j}-1} \\
& \qquad \times
 \bigg(\frac{1 + |\xi_{j,3}|}{\log T^{\gamma}}\bigg)^{\delta_{i_j}-1}
\Pi_{i_j}\\
\ll~& \tilde L^{-E/2} \Pi_{i_j},
\end{align*}
since $0< \delta_{i_j}\leq 1$.
The lemma now follows since
$
\int_{I^4} \tilde L^{-E/2} \d\xi_{j,1}\dots \d\xi_{j,4} = O_E(1).
$
\end{proof}

\begin{lemma}\label{lem:ratio}
For every $\u \in U(\bka,\bla)$, we have
\begin{align*}
\sum_{\substack{\bd,\q, \e,\e' \\ \gcd(d_j,u_j)=1}} 
 &\alpha_{\h'}\big(\tilde \Delta_1, \dots\tilde \Delta_{r'}\big)
 J_1\dots J_{r'} =  
 (1+ o(1))
 \prod_{j=1}^{r'}
 \sum_{\substack{d_j: \\ (d_j,u_j)=1}} 
 \sum_{\substack{q_j\in \<\bar\cP_2^{(i_j)}\> \\ 
    v_p(q_j) \not= 1~\forall p}}
 \sum_{e_j,e'_j \in \<\bar\cP_2^{(i_j)}\>} \frac{J_j}{\tilde\Delta_j}.
\end{align*}
\end{lemma}
Before establishing this result, let us indicate how it suffices to
conclude the proof of \eqref{eq:linearforms-5}.
The second equality in \eqref{eq:linearforms-5} is obvious, 
and so only the first requires a proof. 
Lemma \ref{lem:ratio} implies that the difference of the two integrands
is pointwise bounded by
$$
o\Bigg( \bigg|
 \prod_{j=1}^{r'}
 \sum_{\substack{d_j: \\ (d_j,u_j)=1}} 
 \sum_{\substack{q_j\in \<\bar\cP_2^{(i_j)}\> \\ 
    v_p(q_j) \not= 1~\forall p}}
 \sum_{e_j,e'_j \in \<\bar\cP_2^{(i_j)}\>} \frac{J_j}{\tilde\Delta_j}
\bigg|\Bigg).
$$ 
Lemma \ref{lem:L1-bound} implies that the integral over this bound equals
$o(\prod_{j}\Pi_{i_j})$,
 which implies the first part of \eqref{eq:linearforms-5}.

\begin{proof}[Proof of Lemma \ref{lem:ratio}]
Our argument is identical to that of \cite[Claim 3]{lm1},
but we provide more detail here. 
Throughout this proof we assume, without explicitly mentioning so, that all 
entries $d_j$ of any vector $\bd$ satisfy $\gcd(d_j,u_j)=1$.
The aim is to  study the multiplicative function
\begin{align*}
&\eta(\bd,\q,\e,\e') 
= \alpha_{\h'}\big(\tilde \Delta_1,\dots, \tilde \Delta_{r'} \big)
J_1\dots J_{r'},
\end{align*}
where $J_j$ is given by \eqref{eq:def-jj}.
We may factorise
\begin{align*}
&\eta(\bd,\q,\e,\e')=
\eta(\tilde\bd,\tilde\q,\tilde\e,\tilde\e')
\eta(\bar\bd,\bar\q,\bar\e,\bar\e')
\end{align*}
in such a way that in the first factor the $r'$ entries 
$\tilde d_j \tilde q_j \tilde e_j \tilde e'_j$ for 
$1\leq j \leq r'$ of $\alpha_{\h'}$ are pairwise coprime, while in the
second factor, any prime that divides one entry of $\alpha_{\h'}$ also
divides a second entry.
The aim is to show that the main contribution from either side of the
expression in the statement of the lemma comes from such vectors
$(\bd,\q,\e,\e')$ for which the second factor in this decomposition is 
$1=\eta(\1,\1,\1,\1)$.

We begin with the left hand side.
Let $k$ be an integer.
Then, with the notation
$p^{\x}=(p^{x_1},\dots,p^{x_{r'}})$
for a prime $p$ and $\x \in \ZZ_{\geq 0}^{r'}$, we have
\begin{align*}
\dsum_{\bd,\q,\e,\e'} \eta(\bd,\q,\e,\e')
=
 \bigg(
  \dsum_{\substack{\bd,\q,\e,\e'\\ \gcd(d_jq_je_je'_j,k)=1 \\ 
         1\leq j\leq r'}}
  \eta(\bd,\q,\e,\e')
 \bigg)
 \prod_{p|k}
 \bigg(
  \dsum_{\substack{\bd,\q,\e,\e' \in \ZZ_{\geq 0}^{r'}}}
  \eta(p^{\bd},p^{\q},p^{\e},p^{\e'})
 \bigg),
\end{align*}
where $\sum'$ denotes that the sum is restricted to coprime vectors in
the above sense; i.e.
$(\bd,\q,\e,\e')=(\tilde\bd,\tilde\q,\tilde\e,\tilde\e')$ in the first
two sums, and in the third sum only one of the $r'$ integers
$\max(d_j,q_j,e_j,e'_j)$ may be non-zero.
We claim that for $p>w(T)$ and sufficiently large $T$ the latter sum
satisfies
\begin{equation} \label{eq:splitting-1}
 \dsum_{\bd,\q,\e,\e' \in \ZZ_{\geq 0}^{r'}}
 \eta(p^{\bd},p^{\q},p^{\e},p^{\e'})
= 1 + O(p^{-1}).
\end{equation}
Taking this on trust for a moment, we see that the previous two equations
imply
\begin{align*}
  \dsum_{\substack{\bd,\q,\e,\e'\\ \gcd(d_jq_je_je'_j,k)=1 \\ 
         1\leq j\leq r'}}
  \eta(\bd,\q,\e,\e')
= \dsum_{\bd,\q,\e,\e'} \eta(\bd,\q,\e,\e')
 \prod_{p|k} (1 + O(p^{-1})).
\end{align*}
Applying this with 
$k= \prod_{j=1}^{r'} \bar d_j\bar q_j\bar e_j\bar e'_j$, we obtain
\begin{align*}
 \sum_{\substack{(\bd,\q,\e,\e') \\
   \{d_j q_j e_j e_j':1 \leq j \leq r'\}\\
   \text{not pairwise coprime}}}
&\eta(\bd,\q,\e,\e')\\
&\hspace{-2cm}= \left(
   \dsum_{\bd,\q,\e,\e'} \eta(\bd,\q,\e,\e')
   \right)
   \left(
   \dagsum_{\substack{(\bar\bd,\bar\q,\bar\e,\bar\e')\\
     \neq(\1,\1,\1,\1)}}
   \eta(\bar\bd,\bar\q,\bar\e,\bar\e')
   \prod_{p|k}
   (1 + O(p^{-1}))
   \right),
\end{align*}
where $\sum^\dagger$ denotes that the sum is restricted to vectors failing
coprimality at every prime $p$; i.e. the vector
$(v_p(\bar d_j\bar q_j\bar e_j \bar e'_j))_{j=1}^{r'}$ has either no or
at least two non-zero entries.

Our next aim is to bound this second factor from above. 
We will do this by writing it as an Euler product and analysing
contributions for each prime factor separately.
The saving in the bound will come from the factor $\alpha_{\h'}$ in
$\eta$.
The remaining factors may be bounded trivially by
\begin{align}\label{eq:splitting-2}
 \Big|
 \mu(p^{e_j})\mu(p^{e'_j}) 
 p^{-e_j z_{j,1}} p^{-e_j' z_{j,2}}
 g_{i_j}(p^{d_j}) 
 p^{- d_j z_{j,3}}
 \tau(p^{q_j})^{n_{i_j}}
 p^{-q_j z_{j,4}}
 \Big|
\leq 2^{d_j n_{i_j}} (q_j +1)^{n_{i_j}}.
\end{align}
In order to turn the sum over $\bar\bd,\bar\q,\bar\e,\bar\e'$ into one
that directly runs over the entries of $\alpha_{\h'}$, note that any
integer $k_j$ may be factorised as $\tilde \Delta_j$
in at most $\tau_5(k_j)$ ways, corresponding to the five factors $d_j$,
$q_j$, $\gcd(e_j,e'_j)$, $e_j/\gcd(e_j,e'_j)$ and
$e'_j/\gcd(e_j,e'_j)$.
We will employ the crude bound $\tau_5(p^{a_j}) \ll a_j^4$.
Let $n(\x)$ denote the number of non-zero components of $\x \in \ZZ^{r'}$.
Then the previous inequality implies
\begin{align*}
 \sum_{\substack{\bd,\q,\e,\e' \\
 n(\bd+\q+\max(\e,\e'))\geq 2}}   
 \eta(p^{\bd},p^{\q},p^{\e},p^{\e'})
\ll 
 \sum_{\substack{\a \in \ZZ_{\geq 0}^{r'} \\ n(\a) \geq 2}}
 \alpha_{\h'}(p^{a_1},\dots,p^{a_{r'}})
 \prod_{j=1}^{r'} C^{a_j} (a_j+1)^C,
\end{align*}
for some absolute positive constant $C$. 
Assuming $p>w(T)$ for sufficiently large $T$ and
introducing the variable $J = \max_{j \not=j'}(a_j+a_{j'})$,
the third case of \eqref{eq:ev-alpha} shows that this in turn is bounded
by
\begin{align*}
\ll 
 \sum_{J \geq 2} p^{-J} C^{r' J}J^{r'C}
\ll \frac{1}{p^2}.
\end{align*}
Recall that all components of $\bd$,$\q$,$\e$ and $\e'$ are composed only
of prime factors larger than $w(T)$.
In total, we deduce that
\begin{align*}
\sum_{\substack{(\bd,\q,\e,\e') \\
  (\bar\bd,\bar\q,\bar\e,\bar\e') \not= (\1,\1,\1,\1)}}
\eta(\bd,\q,\e,\e')
&\leq 
 \dsum_{\bd,\q,\e,\e'} \eta(\bd,\q,\e,\e')
 \Bigg(
 \prod_{p>w(T)}
 \bigg(
 1+ O\bigg(
 \frac{1}{p^2}\bigg)
 \bigg)
 -1
 \Bigg) \\
&\leq \dsum_{\bd,\q,\e,\e'} \eta(\bd,\q,\e,\e')
 \Bigg( \sum_{m>w(T)} m^{-3/2}
 \Bigg) \\
&\ll w(T)^{-1/2} \dsum_{\bd,\q,\e,\e'} \eta(\bd,\q,\e,\e').
\end{align*}

Thus, for the treatment of the left hand side of the expression from the lemma, 
it remains to prove \eqref{eq:splitting-1}.
Employing \eqref{eq:splitting-2} and the bound on $\tau_5$ another time,
we turn the sum into one that only involves $\alpha_{\h'}$ and may be
estimated using the second part of \eqref{eq:ev-alpha}. 
Introducing the variable $Q= \max_j(a_j)$, we have
\begin{align*}
 \dsum_{p^{\bd},p^{\q},p^{\e},p^{\e'}}
 \eta(p^{\bd},p^{\q},p^{\e},p^{\e'})
&= 1 + O\left(
 \sum_{\substack{\a \in \ZZ_{\geq 0}^{r'}: \\ n(\a) = 1}}
 \alpha_{\h'}(p^{a_1},\dots,p^{a_{r'}})
 \prod_{j=1}^{r'} C^{a_j} (a_j+1)^C \right)\\
&= 1 + O\left(\sum_{Q \geq 1} p^{-Q} C^Q (Q+1)^{C+4}\right)\\
&= 1 + O(p^{-1}).
\end{align*}
This completes the proof of the estimate
\begin{align*}
\sum_{\bd,\q, \e,\e'} 
 \alpha_{\h'}\big(\tilde \Delta_1, \dots\tilde \Delta_{r'}\big)
 J_1\dots J_{r'} 
&= (1+ o(1))
 \dsum_{\bd,\q, \e,\e'} 
 \alpha_{\h'}\big(\tilde \Delta_1, \dots\tilde \Delta_{r'}\big)
 J_1\dots J_{r'} 
\\ &= (1+ o(1))
 \dsum_{\bd,\q, \e,\e'}
 \prod_{j=1}^{r'}
 \frac{J_j}{\tilde \Delta_j}.
\end{align*}
To complete the proof of the lemma, we need to show that in fact
$$
 \dsum_{\bd,\q, \e,\e'}
 \prod_{j=1}^{r'}
 \frac{J_j}{\tilde \Delta_j}
= (1+o(1))  \sum_{\bd,\q, \e,\e'}
 \prod_{j=1}^{r'}
 \frac{J_j}{\tilde \Delta_j}. 
$$
This follows by arguing as above when $\eta(\bd,\q,\e,\e')$ is redefined to equal 
$\prod_{j=1}^{r'} J_j /\tilde \Delta_j$, and when taking into account that the 
product $(p^{a_1} \dots p^{a_{r'}})^{-1}$, which replaces 
$\alpha_{\h'}(p^{a_1},\dots,p^{a_{r'}})$, trivially satisfies the bounds 
\eqref{eq:ev-alpha}. 
\end{proof}

\begin{proof}[Proof of 
Proposition \ref{p:correlation-condition}]
A slight adaptation of \cite[Lemma 9.9]{GT} yields the following.
Let $\Delta:\ZZ \to \ZZ$ denote the polynomial
$$\Delta(m)= \prod_{1 \leq j < j' \leq d} (W m + A_{i_j} - A_{i_{j'}}).$$
Suppose $\sigma: \{-T',\dots,T'\} \to \RR$ 
satisfies the two conditions $\sigma(0) = O({T'}^{1/q})$ and
$$
\sigma(m) 
= \exp \left( \sum_{p>w(T),~ p|\Delta(m)} O_T(p^{-1/2}) \right)
$$
for $m\not=0$.
Then $\EE_{0\leq m \leq \tilde T} \sigma^q(m) \ll_q 1$.

Whenever the collection of $a_j$ contains two identical elements, then 
$\sigma(0)$ appears in the bound we seek to establish.
Following \cite{GT-longAPs,GT} closely, we use the fact that $\sigma(0)$
may be chosen to be rather large in order to handle this case.
More precisely, it follows from H\"{o}lder's inequality and the fact that 
$r_{i,\mathrm{res}}$ satisfies part (b) of Definition~\ref{def:M_reg}, that
\begin{align*}
\frac{1}{T'}\sum_{m \in I} \prod_{j=1}^d \phi_{i_j}(T;{A_{i_j}})
  \nu_{i_j}^{(T)}(W(m+a_j) + A_{i_j})
  {\nu'}^{(T)}_{i_j, \mathrm{sieve}}(W(m+a_j) + A_{i_j})
\ll_{c,d} T^{cd}. 
\end{align*}
See \cite[\S9]{GT-longAPs} and \cite[\S7]{lm0} for details.
Choosing $c = 1/(2qd)$ ensures that the value on the right hand side is of 
order $o((T/W)^{1/q})$, so that we may set $\sigma(0) = T^{1/2q}$.

In the remaining case where the $a_j$ are pairwise distinct, the system
of linear forms is less degenerate and we may employ the same techniques
used to prove Proposition \ref{p:linear-forms}.
The key observation is that whenever a prime $p$ divides two distinct
polynomials $W(m+a_j) + A_{i_j}$ and $W(m+a_{j'}) + A_{i_{j'}}$ at $m$,
then it divides 
$W(a_j-a_{j'}) + A_{i_j} - A_{i_{j'}}$.
This provides sufficient information to handle the divisor densities 
$\alpha(p^{a_1},\dots,p^{a_d})$ that occur.
See \cite[\S7]{lm0} for details.
\end{proof}

\bigskip
\section{Conclusion of the proof}\label{s:proof}
We have now everything in place in order to complete the proof of Theorem
\ref{t:NB}.
Recall from \eqref{eq:def-count} that
$$
N(T)=
\sum_{\substack{
\u \in \ZZ^s\cap T\mathfrak{K}\\
\u\= \a\bmod{M}}}
\prod_{i=1}^r
R_i(f_i(\u)),
$$
where 
$R_i(m)=R_i(m;\mathfrak{X}_i,\b_i;M)$ 
is given by Definition \ref{def:repr-fn}
for non-zero $m\in \ZZ$.
Here, $M\in \NN$, 
 $\a\in (\ZZ/M\ZZ)^s$ and 
$\b_i\in (\ZZ/M\ZZ)^{n_i}$ for $1\leq i\leq r$.
Moreover,  $\mathfrak{X}_i\subset \mathfrak{D}_{i,+}$
is a cone for which the bounded
set $\mathfrak{X}_i \cap \mathfrak{D}^{\epsilon}_{i,+}(1)$ has
an $(n_i-1)$-Lipschitz parametrisable boundary, unless it is empty. Finally, 
  $\mathfrak{K} \subset \RR^{s}$ is a convex bounded set.

Recall the definition \eqref{eq:def-A} of $\mathcal{A}_i$, for each 
$1\leq i\leq r$.
Let 
$$
\cW =\left\{ \u_0 \in (\ZZ/W\ZZ)^s: 
\begin{array}{l}
 f_i(\u_0) \in \cA_i \text{ for } i = 1,\dots,r \cr
 \u_0 \equiv \a \bmod{M}
\end{array}
\right\},
$$
where $W$ is given by \eqref{def:W} and is divisible by $M$.
For each $p < w(T)$ we define the corresponding set
$$
\cW_p =\left\{ \u_0 \in (\ZZ/p^{v_p(W)}\ZZ)^s: 
\begin{array}{l}
 v_p(f_i(\u_0)) < v_p(W)/3 \text{ for } i = 1,\dots,r \cr
 \u_0 \equiv \a \bmod{p^{v_p(M)}}\\
 \rho_i(p^{v_p(W)},f_i(\u_0);p^{v_p(M)})>0
\end{array}
\right\},
$$
where we recall that 
$ \rho_i(p^{v_p(W)},f_i(\u_0);p^{v_p(M)})$ also depends on 
$\b_i$.
Define
the functions $R'_i$ as in \eqref{eq:Ri'} and recall the
exceptional set $\mathcal{S}_{C_1,T}$ from Definition
\ref{def:ex}.
We note that 
$$
R'_i(m)=\sum_{A\in \mathcal{A}_i} \1_{m\=A\bmod{W}} R'_i(m),
$$
for $1\leq i\leq r$.
Indeed, suppose  $|m|\not\in \mathcal{S}_{C_1,T}$ with $R'_i(m)\neq 0$. 
Then it is clear that the reduction of $m$ modulo $W$ must belong to
$\mathcal{A}_i$. 
By Proposition \ref{p:unexceptional'}, it therefore  suffices to obtain an 
asymptotic for 
\begin{equation}
\label{eq:warwick}
N'(T) 
=
\sum_{\u_0 \in \cW}
\sum_{\substack{
\u_1 \in \ZZ^s\\
W \u_1 + \u_0 \in T\mathfrak{K}}}
\prod_{i=1}^r
R'_i(f_i(W \u_1 + \u_0)).
\end{equation}

Next, let $i \in \{1,\dots,r\}$. 
We have 
$f_i(W \u_1 + \u_0)=Wf_i( \u_1) +f_i( \u_0)$, since each $f_i$ is a
linear form.
Let $0<A'_i(\u_0)<W$ be such that $A'_i(\u_0) \equiv f_i(\u_0)\bmod{W}$. 
We proceed to define a linear polynomial $h_i\in \ZZ[\u]$ via
$$
Wf_i( \u) +f_i( \u_0)=
W h_i(\u) + A'_i(\u_0).
$$
Note that $h_i$ may be  inhomogeneous.

For any fixed residue $\u_0 \in \cW$, let $\mathfrak{K}_{\u_0,T}$ be the
set of $\u_1 \in \RR^s$ for which $W \u_1+ \u_0 \in T\mathfrak{K}$.
Thus, $W\mathfrak{K}_{\u_0,T}+\u_0 = T\mathfrak{K}$.
We proceed to split $\mathfrak{K}_{\u_0,T}$ into regions on which the 
sign of $W h_i(\u_1) + A'_i$ is constant for $1\leq i \leq r$.
Thus, for $\beps=(\epsilon_1,\dots,\epsilon_r )\in \{\pm\}^r$ let 
$$\mathfrak{K}_{\u_0,T}(\beps)
=\{\u_1 \in \mathfrak{K}_{\u_0,T}: 
\mathbf{f}(W\u_1 + \u_0) \in 
\RR_{\epsilon_1} \times \dots \times \RR_{\epsilon_r}
\},
$$
where $\RR_\epsilon=\{x\in \RR: \epsilon x>0\}$. Note 
 that this is a finite union of convex subsets of
$[-CT/W,CT/W]^s$ for some absolute constant $C>0$.
Furthermore
\begin{align*}
\vol\left(\mathfrak{K}_{\u_0,1}(\beps)\right)
= \frac{\vol(\mathfrak{K} \cap 
 {\mathbf f}^{-1}(\RR_{\epsilon_1} \times \dots \times \RR_{\epsilon_r}))}
 {W^s}.
\end{align*}
Since $|\mathbf{f}(T\mathfrak{K})|\leq T$, it follows that 
$|Wh_i(\mathfrak{K}_{\u_0,T})+A'_i(\u_0)|\leq T$ for each $\u_0\in \cW$.

The existence of a simultaneous pseudorandom majorant for each collection
of functions
$$
\widetilde R_i:\quad 
m \mapsto \Big(\frac{\rho_i(W,f_i(\u_0);M)}{W^{n_i-1}}\Big)^{-1}
R'_i(W \epsilon_i m + A'_i(\u_0))$$
defined on the range 
$\{m: 0 < W m + \epsilon_i A'_i(\u_0) \leq T\}$ was
established in Sections \ref{s:majorant} and  \ref{s:linear-forms}.
This existence allows us to employ the generalised von Neumann theorem
\cite[Prop.~7.1]{GT} to deduce that
the sum over $\u_1$ in \eqref{eq:warwick} is equal to 
$$
T^s
\sum_{\beps\in \{\pm\}^r}
 \vol(\mathfrak{K}_{\u_0,1}(\beps))
\prod_{i=1}^r
 \frac{\rho_i(W,f_i(\u_0);M)}{W^{n_i-1}} 
 \kappa_i^{\epsilon_i}(\mathfrak{X}_i) 
 + o\left(\frac{T^s}{W^s}\right),
$$
provided that for each $\u_0 \in \cW$ the normalised representation
function $\widetilde R_i(m)$
satisfies
$$
\max_{1 \leq i \leq r} 
\| \widetilde R_i - \kappa_i^{\epsilon_i}(\mathfrak{X}_i)\|_{U^{r-1}} 
= o(1).
$$
The latter, however, follows from the inverse theorem \cite{GTZ} for the
Gowers uniformity norms from Proposition \ref{p:nilsequences} and the
bound
$$
\EE_{|m| < T/W} R_{i}(Wm+A) \1_{Wm+A \in \mathcal{S}_{C_1,T}}
\ll (\log T)^{-C_1/4},
$$
 provided by \eqref{eq:R_i-in-S}.

Let
$$
\mathfrak{S}(T)=
 \frac{1}{W^s}
 \sum_{\u_0 \in \cW} \prod_{i=1}^r \prod_{p \leq w(T)}
 \frac{\rho_i(p^{\alpha(p)},f_i(\u_0);p^{v_p(M)})}{p^{\alpha(p) (n_i-1)}},
$$
with $\alpha(p)=v_p(W)$. 
We conclude that
\begin{align*}
N'(T)
=~& \beta_\infty \mathfrak{S}(T) T^s 
+ o(T^s),
\end{align*}
with
$\beta_\infty$ as in  the statement of Theorem \ref{t:NB}.
It therefore remains to analyse $\mathfrak{S}(T).$ 
An application of the Chinese remainder theorem yields 
\begin{align*}
\mathfrak{S}(T)=
 \prod_{p \leq w(T)}
 \frac{1}{p^{s\alpha(p)}} \sum_{\u_0 \in \cW_p}
 \prod_{i=1}^r
 \frac{\rho_i(p^{\alpha(p)},f_i(\u_0);p^{v_p(M)})}{p^{\alpha(p) (n_i-1)}}.
\end{align*}
Let us define $\ve_p$ via
$$
\beta_p=
 \frac{1}{p^{s\alpha(p)}} \sum_{\u_0 \in \cW_p}
 \prod_{i=1}^r
 \frac{\rho_i(p^{\alpha(p)},f_i(\u_0);p^{v_p(M)})}{p^{\alpha(p) (n_i-1)}}
 +\eps_p,
$$
where $\beta_p$ is as in the statement of Theorem \ref{t:NB}.
In order to complete the proof, it remains to check that
$\ve_p$ is sufficiently small to be able to  conclude that
$$
\prod_{p<w(T)} (\beta_p - \eps_p)
= \prod_{p<w(T)} \beta_p + o(1),
$$
as $T \to \infty$.  This will certainly suffice,  since Proposition
\ref{p:1} implies that 
$$
\prod_{p>w(T)} \beta_p = 1 + o(1),
$$
as $T\rightarrow\infty$. 
Recalling that  $w(T)= \log\log T$, it will be enough to show that
$\eps_p \ll (\log T)^{-C}$ for some absolute constant $C>0$.

For this we shall apply 
Lemma \ref{lem:C6.4}  to 
$\gamma(p^m,A;p^{\ell})=\rho_i(p^m,f_i(\u_0);p^{v_p(M)})$
for $m\geq \alpha(p)$. This  yields
\begin{align*}
\ve_p
& = \lim_{m\to \infty} \frac{1}{p^{ms}} 
 \sum_{\substack{\u \in \cU^*_m \\ \u \equiv \a \bmod{M}}} 
 \prod_{i=1}^r
 \frac{\rho_i(p^m,f_i(\u);p^{v_p(M)})}{p^{m(n_i-1)}} \\
& \leq \lim_{m\to \infty} \frac{1}{p^{ms}} 
 \sum_{\u \in \cU^*_m} 
 \prod_{i=1}^r
 \frac{\rho_i(p^m,f_i(\u))}{p^{m(n_i-1)}} ,
\end{align*}
where
$$
\cU^*_m = \left\{ \u \in (\ZZ/p^m\ZZ)^s:
 \max_{1\leq  i \leq r} v_p(f_i(\u)) \geq \min \{\alpha(p)/3,m\}
\right\}.
$$
Proceeding as in the analysis of $\beta_p$ for large $p$ in the proof of
Proposition \ref{p:1}, we obtain for $m \geq \alpha(p)$ the bound 
\begin{align*}
\frac{1}{p^{ms}} 
 \sum_{\u \in \cU^*_m} 
 \prod_{i=1}^r
 \frac{\rho_i(p^m,f_i(\u))}{p^{m(n_i-1)}} 
&\ll \frac{1}{p^{ms}}
 \sum_{\u \in \cU^*_m} 
 \prod_{i=1}^r \min\{v_p(f_i(\u))+1,m\}^{n_i}\\
&\leq \sum_{\substack{\k \in \ZZ_{\geq0}^r \\ 
   \max_i k_i \geq \alpha(p)/3}}
 \alpha_{\mathbf{f}}(p^{k_1}, \dots, p^{k_r})
 \prod_{i=1}^r k_i^{n_i} \\
&\ll  p^{-\alpha(p)/3}\\ 
&\leq (\log T)^{-C_1/4}.
\end{align*}
This completes the proof of Theorem \ref{t:NB}.


\bigskip
\begin{thebibliography}{99}


\bibitem{BHB}
T.D. Browning and D.R. Heath-Brown,
 Quadratic polynomials represented by norm forms.
{\em GAFA} {\bf 22} (2012), 1124--1190.


\bibitem{bms}
T.D.~Browning, L.~Matthiesen and A.N.~Skorobogatov,
Rational points on pencils of conics and quadrics with many degenerate
fibres. 
{\em Annals of Math.} {\bf 180} (2014), 381--402. 

\bibitem{ct-4}
J.-L. Colliot-Th\'el\`ene, Surfaces rationnelles fibr\'ees en coniques
de degr\'e $4$.
{\em S\'eminaire de th\'eorie des nombres, Paris 1988--1989},  43--55,
Progr.\ Math. {\bf  91}, Birkh\"auser, 1990.

\bibitem{bud}
J.-L.~Colliot-Th\'el\`ene. Points rationnels sur les fibrations.  
{\em Higher dimensional varieties and rational points (Budapest, 2001)},
171--221, Springer-Verlag,  2003.

\bibitem{ct-salb}
J.-L.~Colliot-Th\'el\`ene and P.~Salberger,
Arithmetic on some singular cubic hypersurfaces.
{\em Proc.\ London Math.\ Soc.}  {\bf 58} (1989),  519--549.


\bibitem{CTSb}
J.-L.~Colliot-Th\'el\`ene and J.J.~Sansuc,
La $R$-\'equivalence sur les tores. {\em Ann.\ Sci.\ \'Ecole
Norm.\ Sup.} {\bf 10} (1977), 175--229.


\bibitem{D2}
J-L.~Colliot-Th\'el\`ene and J-J.~Sansuc,
La descente sur les vari\'et\'es  rationnelles, II.
{\em Duke Math.\ J.} {\bf 54} (1987),  375--492.


\bibitem{ct-swd}
J.-L. Colliot-Th\'el\`ene and P. Swinnerton-Dyer, 
Hasse principle and weak approximation for pencils of Severi--Brauer
and similar varieties. {\em J.\ reine angew.\ Math.}  {\bf 453}  (1994),
49--112.



\bibitem{CTHS}
J.-L.~Colliot-Th\'el\`ene, D.~Harari and
  A.N.~Skorobogatov, 
Valeurs d'un polyn\^ome \`a une variable repr\'esent\'es par une 
norme. {\em Number theory and algebraic geometry}, 69--89,
London Math.\ Soc.\ Lecture Note Ser.  {\bf 303}
Camb.\ Univ.\ Press,  2003.

\bibitem{crelle-a}
J.-L. Colliot-Th\'el\`ene, J.-J. Sansuc and P. Swinnerton-Dyer, 
Intersections of two quadrics and Ch\^atelet surfaces, I.
{\em J.\ reine angew.\ Math.}  {\bf 373}  (1987), 37--107.

\bibitem{crelle-b}
J.-L. Colliot-Th\'el\`ene, J.-J. Sansuc and P. Swinnerton-Dyer, 
Intersections of two quadrics and Ch\^atelet surfaces, II.
II. {\em J.\ reine angew.\ Math.}  {\bf 374}  (1987), 72--168.

\bibitem{98a}
J.-L. Colliot-Th\'el\`ene, A.N. Skorobogatov and 
P. Swinnerton-Dyer, 
Rational points and zero-cycles on fibred varieties: Schinzel's hypothesis
and Salberger's device.
{\em J.\ reine angew.\ Math.} {\bf 495} (1998), 1--28.



\bibitem{DSW}
U.~Derenthal, A.~Smeets and D.~Wei,
Universal torsors and values of quadratic polynomials represented by
norms. {\em Math. Annalen} {\bf  361} (2015), 1021--1042.

\bibitem{erdos}
P.~Erd\H{o}s,
On the sum $\sum_{k=1}^x d(f(k))$.
{\em J.\ London Math.\ Soc.} {\bf 27} (1952),  {7--15}.



\bibitem{GY2}
D.A.~Goldston and C.Y.~Y{\i}ld{\i}r{\i}m,
{Higher correlations of divisor sums related to primes. {III}.
{S}mall gaps between primes.}
{\em Proc. Lond. Math. Soc.} {\bf 95} (2007), 653--686.

\bibitem{GPY}
D.A.~Goldston, J.~Pintz and C.Y.~Y{\i}ld{\i}r{\i}m,
{Primes in tuples. {I}.}
{\em Annals of Math.} {\bf 170} (2009), 819--862.


\bibitem{GT-longAPs}
B.~Green and T.~Tao, 
The primes contain arbitrarily long arithmetic progressions.
{\em Annals of Math.} {\bf 167} (2008),  481--547.

\bibitem{GT}
B.~Green and T.~Tao, Linear equations in primes.
{\em Annals of Math.} {\bf 171} (2010),  1753--1850.

\bibitem{GT-polynomialorbits}
B.~Green and T.~Tao,
The quantitative behaviour of polynomial orbits on nilmanifolds.
{\em Annals of Math.} {\bf 175} (2012),  465--540.

\bibitem{GT-nilmobius}
B.~Green and T.~Tao, 
The M\"obius function is strongly orthogonal to nilsequences.
{\em Annals of Math.} {\bf 175} (2012),  541--566.

\bibitem{GTZ}
B.~Green, T.~Tao and T.~Ziegler, 
An inverse theorem for the Gowers $U_{s+1}[N]$-norm.
{\em Annals of Math.} {\bf 176} (2012),  1231--1372.


\bibitem{HSW}
Y.~Harpaz, A.N.~Skorobogatov and O.~Wittenberg,
The Hardy--Littlewood conjecture and rational points. 
{\em Compositio Math.} {\bf150} (2014), 2095--2111.

\bibitem{HBS}
D.R.~Heath-Brown and A.N.~Skorobogatov, Rational
solutions of certain equations involving norms. 
{\em Acta Math.} {\bf 189} (2002), 161--177.

\bibitem{landau_alg}
E.~Landau,
\emph{{E}inf\"{u}hrung in die elementare und analytische {T}heorie der
algebraischen {Z}ahlen und der {I}deale.}
Teubner Verlag, Leipzig, 1918.

\bibitem{marcus}
D.A.~Marcus, 
{\em Number fields}. Springer-Verlag, 1977.

\bibitem{lm0}
L.~Matthiesen,
Correlations of the divisor function. 
{\em Proc.\ London Math.\ Soc.} {\bf 104} (2012), 827--858.

\bibitem{lm1}
L.~Matthiesen,
Linear correlations amongst numbers represented by positive definite 
binary quadratic forms.
{\em Acta Arith.} {\bf 154} (2012), 235--306.

\bibitem{lm2}
L.~Matthiesen,
Correlations of representation functions of binary quadratic forms.
{\em Acta Arith.} {\bf 158} (2013), 245--252.

\bibitem{MV}
H.L.~Montgomery and R.C.~Vaughan,
{\em Multiplicative Number Theory, I. Classical Theory}.
Camb.\ Univ.\ Press,  2007. 

\bibitem{neukirch}
J.~Neukirch, {\em Algebraic number theory}. Springer-Verlag, 1991.

\bibitem{s-s}
D.~Schindler and A.N.~Skorobogatov, 
Norms as products of linear polynomials.
{\em J. London Math. Soc.} {\bf 89} (2014), 559--580. 

\bibitem{serre}
J.-P. Serre, {\em A course in arithmetic}.
Springer-Verlag, 1996.

\bibitem{shiu}
P.~Shiu, A Brun--Titchmarsh theorem for multiplicative functions.
{\em J. reine angew. Math.} {\bf 313} (1980), 161--170.

\bibitem{SJ}
M. Swarbrick Jones, A note on a theorem of Heath-Brown and Skorobogatov. 
{\em Q. J. Math.}  {\bf 64} (2013),  1239--1251.

\bibitem{wei}
D.~Wei. On the equation $N_{K/k}(\Xi) = P(t)$. {\em 
Proc. London Math. Soc.} {\bf 109} (2014), 1402--1434.



\end{thebibliography}
\end{document}